\documentclass[12pt,a4paper]{amsart}
\usepackage{amsfonts, amsmath, amssymb, amscd, amsthm, graphicx}
\usepackage{hyperref}

\makeatletter
\@addtoreset{equation}{section} 

\makeatother

\newtheorem{thm}{Theorem}[section]
\newtheorem{cor}[thm]{Corollary}
\newtheorem{lemma}[thm]{Lemma}
\newtheorem{prop}[thm]{Proposition}

\theoremstyle{definition}
\newtheorem{df}[thm]{Definition}
\theoremstyle{remark}
\newtheorem{rem}[thm]{Remark}
\newtheorem{ex}[thm]{Example}

\newcommand{\Z}{\mathbb{Z}}
\newcommand{\N}{\mathbb{N}}
\newcommand{\ccg}{${\rm CC}_G$}
\newcommand{\cc}{{\rm CC}}
\newcommand{\bA}{{\bar{A}}}
\newcommand{\bB}{{\bar{B}}}
\newcommand{\bH}{{\bar{H}}}
\newcommand{\bK}{{\bar{K}}}
\newcommand{\bx}{{\bar{x}}}
\newcommand{\bb}{{\bar{b}}}
\newcommand{\ba}{{\bar{a}}}
\newcommand{\bD}{{\bar{D}}}

\newcommand{\by}{{\bar{y}}}
\newcommand{\bg}{{\bar{g}}}
\newcommand{\bh}{{\bar{h}}}
\newcommand{\PT}{\mathcal{PT}}
\newcommand{\supp}{{\rm supp}}
\renewcommand{\star}{{\rm star}}
\newcommand{\FL}{{\rm FL}}
\newcommand{\LL}{{\rm LL}}

\newcommand{\vr}{{\mathcal{VR}}}
\newcommand{\avr}{{\mathcal{AVR}}}

\newcommand{\rank}{{\rm rank}}

\newcommand{\e}{{\varepsilon}}
\newcommand{\z}{{\zeta}}

\begin{document}
\title[Hereditary conjugacy separability of right angled Artin groups]{
Hereditary conjugacy separability of right angled Artin groups and its applications}

\author{Ashot Minasyan}
\address[Ashot Minasyan]{School of Mathematics,
University of Southampton, Highfield, Sout\-hamp\-ton, SO17 1BJ, United
Kingdom.}  \email{aminasyan@gmail.com}


\begin{abstract}
We prove that finite index subgroups of right angled Artin groups are conjugacy separable.
We then apply this result to establish various properties of other classes of groups.
In particular, we show that any word hyperbolic Coxeter group contains a conjugacy
separable subgroup of finite index and has a residually finite outer automorphism group.
Another consequence of the main result is that Bestvina-Brady groups are conjugacy separable
and have solvable conjugacy problem.
\end{abstract}

\keywords{Hereditary conjugacy separability,
right angled Artin groups, graph groups, partially commutative groups, Coxeter groups,
Bestvina-Brady groups.}

\subjclass[2010]{20F36, 20E26, 20F55, 20F28, 20F10}
\maketitle
\tableofcontents

\section{Introduction}
If $G$ is a group, the \textit{profinite topology} $\PT(G)$ on $G$ is the topology
whose basic open sets are cosets to finite index normal subgroups in $G$. It follows that
every finite index subgroup $K \le G$ is both closed and open in $\PT(G)$, and $G$, equipped
with $\PT(G)$, is a topological group (that is, the group operations are continuous
with respect to this topology).
This topology is Hausdorff if and only if the
intersection of all finite index normal subgroups is trivial in $G$. In this case $G$ is said to
be \textit{residually finite}.

We will say that a subset $A \subseteq G$ is \textit{separable in} $G$
if $A$ is closed in $\PT(G)$.
Suppose that for every element $g \in G$, its conjugacy class
$g^G:= \{hgh^{-1}\,|\, h\in G\} \subseteq G$ is closed
in $\PT(G)$. Then $G$ is called \textit{conjugacy separable}. In other words,
$G$ is conjugacy separable if and only if for any two non-conjugate elements $x,y \in G$
there exists a homomorphism $\varphi$ from $G$ to a finite group $Q$ such that $\varphi(x)$
is not conjugate to $\varphi(y)$ in $Q$.

Conjugacy separability is evidently stronger than residual finiteness, and is (usually)  much
harder to establish. The following classes of groups are known to be conjugacy separable:
virtually free groups (J. Dyer \cite{Dyer}); virtually surface groups (A. Martino \cite{Martino});
virtually polycyclic groups (V. Remeslennikov \cite{Rem-polyc}; E. Formanek \cite{Form});
limit groups (S. Chagas and P. Zalesskii \cite{Chag-Zal-limit}) and, more generally,
finitely presented residually free groups (S. Chagas and P. Zalesskii \cite{Chag-Zal}).

 Unfortunately, conjugacy separability does not behave very well under free constructions.
V. Remeslennikov \cite{Rem-free_prod} and P. Stebe \cite{Stebe} showed that the free product of two
conjugacy separable groups is conjugacy separable.
But so far we do not know of any global
criteria which tell when an amalgamated product (or an HNN-extension) of conjugacy separable groups
is conjugacy separable. Perhaps the most general of local results can be found in \cite{R-S-Z},
where L. Ribes, D. Segal and  P. Zalesskii define a new class of conjugacy separable
groups $\mathfrak X$, which
is closed under taking free products with amalgamation along cyclic subgroups and contains all
virtually free and virtually polycyclic groups. Note that there is no analogue of this
result for HNN-extensions with associated cyclic subgroups, because an HNN-extension of the infinite
cyclic group may fail to be residually finite, as it happens for many Baumslag-Solitar groups.

Let $\Gamma$ be a finite simplicial graph, and
let $\mathcal V$ and $\mathcal E$ be the sets of vertices and edges of $\Gamma$ respectively.
The \textit{right angled Artin group} $G$, associated to $\Gamma$, is given by the presentation

\begin{equation} \label{eq:RAAG-def} G:=\langle \mathcal V \,\|\,uv=vu, \,
\mbox{whenever } u,v \in \mathcal{V} \mbox{ and }  (u,v) \in \mathcal E \rangle.
\end{equation}

In the literature, right angled Artin groups are also called \textit{graph groups} or \textit{partially commutative groups}. These groups received a lot of attention in the recent years: they seem
to be interesting from both combinatorial and geometric viewpoints (they are fundamental groups
of compact non-positively curved cube complexes). A good overview of the current results
concerning right angled Artin groups can be found in R. Charney's paper \cite{Charney}.

In the case when the finite graph $\Gamma$ is a simplicial tree,
conjugacy separability of the associated right angled Artin group was proved by E. Green \cite{Green}.
It also
follows from the result of  Ribes, Segal and  Zalesskii \cite{R-S-Z} mentioned above,
because such \textit{tree groups} are easily seen to belong to the class $\mathfrak X$.

Following \cite{Chag-Zal}, we will say that a group $G$ is \textit{hereditarily conjugacy separable}
if every finite index subgroup of $G$ is conjugacy separable.

Note that all of the classes of conjugacy separable groups that we mentioned above (possibly, with
the exception of class $\mathfrak X$)
consist, in fact, of hereditarily conjugacy separable groups due to the obvious reason:
these classes are closed under taking subgroups of finite index.
However, there exist conjugacy separable, but not hereditarily separable groups.
The first (infinitely generated) example, demonstrating this, was constructed by Chagas
and Zalesskii in \cite{Chag-Zal}. It is also possible to find finitely
generated and finitely presented examples of this sort even among subgroups
of right angled Artin groups (see \cite{Mart-Min}).

The main result of this work is the following theorem:

\begin{thm}\label{thm:RAAG-main} Right angled Artin groups are hereditarily conjugacy separable.
\end{thm}

Remark that a finite index subgroup of a right angled Artin group may not be a right angled
Artin group itself. The following example was suggested to the author by M.~Bridson:

\begin{ex} Let $S$ be a finite group with a (finite) non-trivial second homology group
$H_2(S)$ (for instance, on can take $S$ to be the alternating group $A_5$,
since $H_2(A_5) \cong \Z/2\Z$). As we know, there is an epimorphism
$\psi:F \to S$ for some finitely
generated free group $F$. Let
$K:=\{(x,y)\in F \times F\,|\, \psi(x)=\psi(y)\}$ be the fibre product associated to $\psi$.
Observe that $F \times F$ is a right angled Artin group (associated to some finite
complete bipartite graph) and $K$ is a finite index subgroup in it. By \cite[Thm. A]{Brid-Mill},
$H_2(S)$ embeds into $H_1(K) \cong K/[K,K]$. Therefore $K$ is not isomorphic to any
right angled Artin group,
because the abelianization of a right angled Artin group is always a free abelian group,
and, thus, it is torsion-free.
\end{ex}

Generally speaking, we think
that hereditary conjugacy separability is a lot stronger than simply conjugacy
separability. Corollaries in the next section can be viewed as a confirmation of this.

Our proof of Theorem \ref{thm:RAAG-main} is purely combinatorial and mostly self-contained
(we use basic properties of right angled Artin groups and HNN-extensions).
The basic idea is to approximate right angled Artin groups by HNN-extensions
of finite groups (which are, of course, virtually free). This is the main step of the proof.
Once this is done, we can use known properties of virtually free groups to obtain the desired
results.

In Section \ref{sec:h_c_s_and_CC} we study the \textit{Centralizer Condition},
which, among other things, shows that a given conjugacy separable group is
hereditarily conjugacy separable. This condition
was originally introduced by Chagas and Zalesskii in \cite{Chag-Zal}, but in a different form.
In Sections \ref{sec:comm_retr}, \ref{sec:conseq_prof_top} and \ref{sec:spec_HNN}
we develop machineries of \textit{commuting retractions} and \textit{special HNN-extensions} which are
the two basic tools behind the proof of Theorem \ref{thm:RAAG-main}.

\noindent{\bf Acknowledgements.} The author is very grateful to Fr\'ed\'eric Haglund and Daniel Wise
for explaining their work in \cite{H-W_1} and \cite{H-W_2}. The author would also like to thank
Martin Bridson, Ilya Kazachkov, Graham Niblo and Pavel Zalesskii for discussions.


\section{Consequences of the main theorem}\label{sec:conseq_main_res}
Recall that a subgroup $H$ of a group $G$ is said to be a \textit{virtual retract} of $G$,
if there is a finite index subgroup $K \le G$ such that $H \le K$ and $H$ is a retract of $K$
(see Section \ref{sec:comm_retr} for the definition).

It is not difficult to show (see Lemma \ref{lem:virt_retract_h_c_s})
that a virtual retract of a hereditarily conjugacy separable group is itself
hereditarily conjugacy separable. Therefore, Theorem~\ref{thm:RAAG-main} immediately yields

\begin{cor}\label{cor:virt_retract_RAAG-h_c_s} If $G$ is a right angled Artin group and
$H$ is a virtual retract of $G$, then $H$ is hereditarily conjugacy separable.
\end{cor}

In view of the above corollary, it makes sense to define two classes of groups:
the class $\mathcal{VR}$ will consist of all groups which are virtual retracts of finitely
generated right angled Artin groups, and the class $\mathcal{AVR}$ will
consist of groups that contain finite index subgroups from the class $\mathcal{VR}$.

Looking at the definition, it might seem that the class of right angled Artin groups is not very large.
However, the class of subgroups and virtual retracts of right angled Artin groups is quite rich
and includes many interesting examples. For instance, in the famous paper \cite{B-B}
M. Bestvina and N. Brady constructed subgroups of right angled Artin groups
which have the property ${\rm FP}_2$ but are not finitely presented.

On the other hand, in the recent
work \cite{H-W_1} F. Haglund and D. Wise introduced a new class of \textit{special}
(or $A$-\textit{special}, in the terminology of \cite{H-W_1})
cube complexes, that admit a combinatorial local isometry to the Salvetti cube
complex (see \cite{Charney}) of a right angled Artin group (possibly, infinitely generated).
They proved that the  fundamental group of every special complex $\mathcal{X}$ embeds into
some right angled Artin group $G$ (if $\mathcal{X}$ is not compact and has infinitely many hyperplanes,
then  the corresponding right angled Artin group $G$
will be associated to an infinite graph $\Gamma$, and, hence, will not be finitely generated).

An important property established by Haglund and Wise in \cite{H-W_1}, states
that if $\mathcal X$ is a compact $A$-special cube complex,
then $\pi_1(\mathcal{X})$ is a virtual retract of some right angled Artin group, i.e.,
$\pi_1(\mathcal{X}) \in \vr$. Therefore, using Corollary \ref{cor:virt_retract_RAAG-h_c_s},
we immediately obtain

\begin{cor} \label{cor:A-spec-c_s} If $H$ is the fundamental group of a compact
$A$-special cube complex, then $H$ is hereditarily conjugacy separable.
\end{cor}

Moreover,  many other groups are \textit{virtually special},
i.e., they possess finite index subgroups that are fundamental groups of
special cube complexes. Among virtually special groups
are all Coxeter groups -- see \cite{H-W_2}, fundamental groups of compact surfaces --
see \cite{Crisp-Wiest},
fundamental groups of compact virtually clean
square $\mathcal{VH}$-complexes (introduced by Wise in \cite{Wise-polyg}) -- see  \cite{H-W_1},
graph braid groups (introduced by A. Abrams in \cite{Abrams}) -- see \cite{Crisp-Wiest},
and some hyperbolic $3$-manifold groups -- see \cite{Ches-Debl-Wilt}.

In this paper we will mainly discuss applications of Theorem \ref{thm:RAAG-main}
to Coxeter groups, even though similar corollaries can be derived for the other classes
of virtually special groups listed above.

Recall that a \textit{Coxeter group} is a group  $G$ given by the presentation
\begin{equation}\label{eq:Cox-def}
G=\langle s_1,\dots,s_n\,\|\, (s_is_j)^{m_{ij}}=1, \,\mbox{ for all $i,j$ with } m_{ij} \in \N \rangle, \end{equation} where $M:=(m_{ij})$
is a symmetric $n \times n$ matrix, whose entries satisfy the following conditions:
$m_{ii}=1$ for every $i=1,\dots,n$, $m_{ij} \in \N \sqcup \{\infty\}$ and $m_{ij} \ge 2$
whenever $1 \le i<j \le n$. In the case, when $m_{ij} \in \{2, \infty\}$ for all $i\neq j$,
$G$ is said to be a \textit{right angled Coxeter group}.

For any Coxeter group $G$, G. Niblo and L. Reeves \cite{Niblo-Reeves} constructed a
locally finite, finite dimensional ${\rm CAT}(0)$ cube complex $\mathcal{C}$
on which $G$ acts properly discontinuously.
In \cite{H-W_2} Haglund and Wise show that $G$
has a finite index subgroup $F$ such that $F$ acts freely on $\mathcal{C}$ and the quotient
$F\setminus \mathcal{C}$ is an $A$-special cube complex.
In the case when $G$ is right angled or word hyperbolic
(in Gromov's sense \cite{Gromov}), Niblo and Reeves proved that the action of $G$ on
$\mathcal{C}$ is cocompact (see \cite{Niblo-Reeves}). These results, combined with the
virtual retraction theorem of Haglund and Wise
mentioned above, imply that word hyperbolic (or right angled) Coxeter groups belong to the class $\avr$.
Therefore, using Corollary  \ref{cor:virt_retract_RAAG-h_c_s} we achieve

\begin{cor} \label{cor:Cox-sbgp-c_s} Every word hyperbolic (or right angled) Coxeter group $G$
contains a finite index subgroup $F$ which is hereditarily conjugacy separable.
\end{cor}

Actually, as the paragraph above Corollary \ref{cor:Cox-sbgp-c_s} shows,
the conclusion of this corollary holds for every finitely
generated Coxeter group $G$, whose action on the corresponding Niblo-Reeves cube complex is cocompact.
Such Coxeter groups were completely characterized by P.-E. Caprace and B. M\"uhlherr
in \cite{Caprace-Muhlherr}.

The sole fact of existence of a conjugacy separable finite index subgroup $F$ in $G$ may seem somewhat
unsatisfactory. However, every Coxeter group is virtually torsion-free,
and in a given Coxeter group $G$ it is usually easy to find some torsion-free subgroup of finite index
(for instance, if $G$ is a right angled Coxeter group \eqref{eq:Cox-def}, then the kernel
of the natural homomorphism from $G$ onto
$\langle s_1 \rangle_2 \times \dots \times \langle s_n \rangle_2 \cong (\Z/2\Z)^n$ is torsion-free).
The following statement is proved in Section \ref{sec:appl_sep}:

\begin{cor}\label{cor:t_f_sbgp_hyp_Cox} If $G$ is a word hyperbolic Coxeter group,
then every torsion-free finite index subgroup $H \le G$ is hereditarily conjugacy separable.
\end{cor}

The above corollary produces a lot of new examples of conjugacy separable groups.
More generally, in Corollary \ref{cor:hyp_t-f_AVR->h_c_s} we show that every torsion-free word hyperbolic
group from the class $\avr$ is hereditarily conjugacy separable.

Now, let us discuss some other consequences of the main result.
Beside being a classical subject of group theory, conjugacy separability has two main applications.
One of the applications was found by E. Grossman in
\cite{Grossman}, where she showed that the outer automorphism group $Out(G)$
of a finitely generated conjugacy separable group $G$ is residually finite,
provided that every pointwise inner automorphism of $G$ is inner (an automorphism
$\phi \in Aut(G)$ is called \textit{pointwise inner} if for every $g \in G$,
$\phi(g)$ is conjugate to $g$ in $G$). Thereafter, Grossman used this observation to prove that the
mapping class group of a compact orientable surface is residually finite.

Note that for a finitely generated residually finite group $G$, the group of outer automorphisms
$Out(G)$ need not be residually finite (this should be compared with the classical
theorem of G. Baumslag \cite{Baumslag} claiming that the automorphism group
$Aut(G)$ of a finitely generated residually finite
group $G$ is residually finite). This is a consequence of the result of I. Bumagina and D. Wise
(\cite{Bum-Wise}) which asserts that for every finitely presented group $S$ there exists
a finitely generated residually finite group $G$ such that $Out(G) \cong S$.

In Section \ref{sec:RAAG} we prove that pointwise inner automorphisms of right angled Artin groups
are inner (see Proposition \ref{prop:pi=inn}). Thus Grossman's result, combined with
Theorem \ref{thm:RAAG-main}, gives

\begin{thm}\label{thm:Out_RAAG} For any right angled Artin group $G$,
the group of outer automorphisms $Out(G)$ is residually finite.
\end{thm}

Presently not much is yet known about the outer automorphisms of an arbitrary
right angled Artin group $G$. M. Laurence \cite{Laurence} showed that $Aut(G)$ (and, hence, $Out(G)$)
is finitely generated. More recently, M. Day \cite{Day} proved that $Aut(G)$ (and, hence, $Out(G)$)
is finitely presented. In \cite{Char-Vogt}  R. Charney and K. Vogtmann showed
that $Out(G)$ is virtually torsion-free and has finite virtual cohomological dimension.
Imposing additional conditions on the finite graph $\Gamma$, corresponding to $G$,
M. Gutierrez, A. Piggott and K. Ruane were able to extract more information about the structure
of $Aut(G)$ and $Out(G)$ in \cite{G-P-R}. After finishing this article the author learned that
Charney and Vogtmann gave a different proof of Theorem 2.5 in \cite{C-V-new}.

On the other hand, in Section \ref{sec:appl_out} we use a recent result
of the author with D. Osin from \cite{norm-aut} to prove the following theorem:

\begin{thm}\label{thm:rel_hyp_gp_a_v_r} If $G \in \mathcal{AVR}$ is a
relatively hyperbolic group, then $Out(G)$ is residually finite.
\end{thm}

Note that Theorem \ref{thm:Out_RAAG} is not a consequence of Theorem \ref{thm:rel_hyp_gp_a_v_r}:
it is not difficult to show that a (non-cyclic) right angled Artin group
$G$ 
is relatively hyperbolic
if and only if  the graph $\Gamma$, corresponding to $G$, is disconnected.

Applying Theorem \ref{thm:rel_hyp_gp_a_v_r} to our favorite class of groups from $\avr$, we achieve
\begin{cor} \label{cor:hyp_Cox-Out_rf} For any word hyperbolic Coxeter group $G$, $Out(G)$
is residually finite.
\end{cor}

Unlike automorphisms of right angled Artin groups, automorphism groups of Coxeter groups have already
attracted a lot of attention. In many particular cases the structure
of the (outer) automorphism group is known: see, for instance, P. Bahls's paper
\cite{Bahls} and references therein.
However, because of its generality,
the statement of Corollary \ref{cor:hyp_Cox-Out_rf} seems to be new.

The second classical application of conjugacy separability was found by A. Mal'cev.
As Mal'cev proved in \cite{Malcev} (see also \cite{Mostow}), a
finitely presented conjugacy separable group $G$ has solvable
conjugacy problem. Observe that finite presentability of $G$ is important here,
because the set of finite quotients of an infinitely presented group does not
need to be recursively enumerable.

It follows that the conjugacy problem is solvable for
every group $G \in \vr$: $G$ is finitely presented as a retract of
a finitely presented group, and $G$ is conjugacy separable by Corollary \ref{cor:virt_retract_RAAG-h_c_s}.
However, most of the groups from the class $\vr$ that
we discussed above are already known to have solvable conjugacy problem. Moreover, J.~Crisp, E. Godelle
and B. Wiest \cite{C-G-W} showed that the conjugacy problem for fundamental groups
of $A$-special cube complexes can be solved in linear time.

Nevertheless, the property of hereditary conjugacy separability for a group $G$
turns out to be powerful enough to yield
conjugacy separability and solvability of the conjugacy problem for
many subgroups which are not virtual retracts of $G$ --
see Corollary \ref{cor:h_c_s->c_s_and_conj_prob}.

Recall that a group $G$ is called \textit{subgroup separable} (or {LERF}) if every finitely
generated subgroup $H \le G$ is separable in $G$. In Section \ref{sec:appl_conj_probl}
we prove

\begin{thm}\label{thm:norm_sbgps_RAAG-conj_probl} Let $N$ be a normal subgroup
of a right angled Artin group $G$ such that the quotient $G/N$ is subgroup separable.
Then $N$ is conjugacy separable. If, in addition,
$N$ is finitely generated, then $N$ has solvable conjugacy problem.
\end{thm}

Note that requiring $G/N$ to be subgroup separable cannot be dropped in the above statement:
in \cite{Miller} C. Miller gives an example of a finitely generated subgroup of
$F_2 \times F_2$ that has unsolvable conjugacy problem (here $F_2$ denotes the free group
of rank $2$, and so
$F_2 \times F_2$ is the right angled Artin group associated to a square).

The second claim of Theorem \ref{thm:norm_sbgps_RAAG-conj_probl} may seem surprising:
in general we cannot use Mal'cev's result, mentioned above, to reach the needed conclusion,
because the conditions imposed on $N$ do not constrain it to be finitely presented. Indeed,
let $G$ be the right angled Artin group associated to a finite graph $\Gamma$ and
given by \eqref{eq:RAAG-def}. Let
$N_\Gamma$ be the kernel of the homomorphism $\psi: G \to \Z$ satisfying
$\psi(v)=1$ for each $v \in \mathcal V$, and let $\mathcal{L}_\Gamma$ be the simplicial flag complex,
whose $1$-skeleton is $\Gamma$.

J. Meier and L. VanWyk \cite{M-V}
proved that the group $N_\Gamma$ is finitely generated if and only if
the graph $\Gamma$ is connected. And
in \cite{B-B} Bestvina and Brady showed that $N_\Gamma$ is finitely presented if and only if
the complex $\mathcal{L}_\Gamma$ is simply connected. In the case when $\Gamma$ is connected,
we will say that $N_\Gamma$ is the
\textit{Bestvina-Brady group} associated to $\Gamma$.

For example, if the graph $\Gamma$ is a cycle of length at least $4$,
then $N_\Gamma$ is finitely generated, but not
finitely presented. Obviously, the quotient $G/N_\Gamma \cong \Z$ is subgroup separable,
hence, by Theorem \ref{thm:norm_sbgps_RAAG-conj_probl}, $N_\Gamma$ is conjugacy separable and
has solvable conjugacy problem. More generally, we have the following corollary:

\begin{cor} \label{cor:norm_sbgps_with_ab_quot} If $N$ is a finitely generated normal subgroup
of a right angled Artin group $G$ such that $G/N$ is abelian, then $N$ is hereditarily conjugacy separable
and has solvable conjugacy problem.
In particular, Bestvina-Brady groups are hereditarily conjugacy separable and
have solvable conjugacy problem.
\end{cor}

Corollary \ref{cor:norm_sbgps_with_ab_quot} is a direct consequence of
Corollary \ref{cor:norm_sbgps_with_polyc_quot} (proved at the end of Section \ref{sec:appl_conj_probl}),
that covers the more general case when $G/N$ is polycyclic.
We have chosen to mention the particular situation when the quotient $G/N$ abelian in
Corollary \ref{cor:norm_sbgps_with_ab_quot},
because in this case one can tell whether or not the given normal subgroup $N$ is finitely generated,
using Bieri-Neumann-Strebel invariants, which were studied for right angled Artin groups by
Meier and  VanWyk in \cite{M-V}.

After finishing this paper the author learned that the positive solution of the conjugacy problem
for the group $N$ from Theorem \ref{thm:norm_sbgps_RAAG-conj_probl}
(or Corollary \ref{cor:norm_sbgps_with_ab_quot}) has been known before. This follows from
a more general result of M. Bridson \cite{Bridson-dec_conj_prob}, claiming that
a normal subgroup $N$ of a bicombable group $G$ has solvable conjugacy problem,
provided $G/N$ has solvable generalized word problem
(see \cite{Bridson-dec_conj_prob} for the definitions).
Indeed, any right angled Artin group $G$  acts properly and cocompactly on a ${\rm CAT}(0)$ space
(which is the universal cover of the corresponding compact
non-positively curved Salvetti cube complex), therefore $G$ is bicombable by a theorem of J. Alonso
and M. Bridson \cite{Alonso-Bridson}. And if $N \lhd G$, the subgroup separability of $G/N$
implies that $G/N$ has solvable generalized word problem, by Mal'cev's result \cite{Malcev}.

Nevertheless, the statement claiming that $N$ is conjugacy separable
in Theorem \ref{thm:norm_sbgps_RAAG-conj_probl} (resp. hereditarily conjugacy separable in
Corollary \ref{cor:norm_sbgps_with_ab_quot}) is new. Our solution of the conjugacy
problem for  $N$ uses a Mal'cev-type argument and can be viewed as another application
of hereditary conjugacy separability of $G$.


\section{Hereditary conjugacy separability and Centralizer Conditions}\label{sec:h_c_s_and_CC}
First, let us specify some notations.
If $G$ is a group, $H \le G$ is a subgroup and $g \in G$, then the $H$-\textit{conjugacy class}
of the element $g\in G$ is the subset $g^H := \{hgh^{-1}\,|\,h \in H\} \subseteq G$.
For any $A \subseteq G$, we denote $A^H:=\{hah^{-1}\,|\,a \in A, h \in H\}$.
The $H$-\textit{centralizer} of $g \in G$ is the subgroup $C_H(g):=\{h \in H~|~hg=gh\} \le G$.
For an epimorphism $\psi: G \to F$ from $G$ onto a group $F$, $\psi^{-1}: 2^{F} \to 2^G$ will denote the
corresponding \textit{full preimage map}. If $A, B$ are two subsets of $G$ then their \textit{product}
$AB$ is defined as the subset $\{ab \,|\,a\in A, b\in B\} \subseteq G$. Note that if either $A$ or
$B$ is empty, then the product $AB$ is empty as well.

\begin{df}\label{df:CC} Suppose that $G$ is a group. We will say that
$G$ satisfies the \textit{Centralizer Condition} (briefly, {\cc}), if for every finite index  normal
subgroup $K \lhd G$  and every $g \in G$ there is a finite index normal
subgroup $L \lhd G$ such that $L \le K$ and
\begin{equation}\label{eq:CC-def}
 C_{G/L} (\bar g)  \subseteq \psi \bigl(C_G(g) K\bigr) ~\mbox{ in }G/L\end{equation}(where
$\psi:G \to G/L$ is the natural epimorphism and $\bar g:=\psi(g)$).
\end{df}

Note that \eqref{eq:CC-def} is equivalent to  $\psi^{-1}\bigl(C_{G/L} (\bar g)\bigr)  \subseteq C_G(g) K$
in $G$, since $\ker(\psi)=L\le K$.

The idea behind this condition is to provide control over the growth of centralizers in finite quotients of
$G$. If the group $G$ is residually finite,
the Centralizer Condition {\cc} can be reformulated in terms of the topology on the \textit{profinite completion}
$\widehat{G}$ of the group $G$. In the Appendix
to this paper (see Corollary \ref{cor:CC<->Chag-Zal}) we prove that the condition
{\cc} from Definition \ref{df:CC} is equivalent to the following:
\begin{equation}
\label{eq:CC-profinite} \overline{C_G(g)}=C_{\widehat{G}}(g) ~\mbox{in $\widehat G$, for every } g \in G
\end{equation}
(where $\overline{C_G(g)}$ denotes the closure of ${C_G(g)}$ in $\widehat G$).

Originally the condition \eqref{eq:CC-profinite} appeared in the recent work of
Chagas and Zalesskii \cite{Chag-Zal}, where they proved that a conjugacy separable group $G$ satisfying \eqref{eq:CC-profinite} is hereditarily conjugacy separable (see \cite[Prop. 3.1]{Chag-Zal}).
We will actually show that, provided $G$ is conjugacy separable,
this condition is equivalent to hereditary conjugacy separability:

\begin{prop}\label{prop:her_c_s-CC} Let $G$ be a group. Then the following are equivalent:
\begin{itemize}
	\item[(a)] $G$ is hereditarily conjugacy separable;
	\item[(b)] $G$ is conjugacy separable and satisfies \cc.	
\end{itemize}
\end{prop}

Before proving Proposition \ref{prop:her_c_s-CC}, let us define two more conditions.

\begin{df}\label{df:CCG} Let $G$ be a group, $H \le G$ and $g \in G$. We will say that the pair $(H,g)$
satisfies the \textit{Centralizer Condition  in $G$} (briefly, \ccg), if for every finite index  normal
subgroup $K \lhd G$
there is a finite index normal subgroup $L \lhd G$ such that $L \le K$ and
$ C_{\bar H} (\bar g)  \subseteq \psi\left (C_H(g) K\right)$ in $G/L$, where
$\psi: G \to G/L$ is the natural homomorphism, $\bar H :=\psi(H) \le G/L$, $\bar g := \psi(g) \in G/L$.

The subgroup $H$ will be said to satisfy the \textit{Centralizer Condition  in $G$} (briefly, \ccg)
if for each $g \in G$, the pair $(H,g)$ has \ccg.
\end{df}

Now, let demonstrate why the Centralizer Conditions are useful.

\begin{lemma}\label{lem:CCG->sep_c_c_for_sbgps} Suppose that $G$ is a group, $H \le G$ and $g \in G$.
Assume that the pair $(G,g)$ satisfies {\ccg} and the conjugacy class $g^G$ is separable in $G$.
If the double coset $C_G(g)H$ is separable in $G$, then the $H$-conjugacy class $g^H$ is also
separable in $G$.
\end{lemma}

\begin{proof} Consider any element $y \in G$ with $y \notin g^H$.

If $y \notin g^G$, then,
using the separability of $g^G$, we can find a finite quotient $Q$ of $G$ and a
homomorphism $\phi: G \to Q$ so that $\phi(y) \notin \phi(g)^Q$. Hence
$\phi(y) \notin \phi(g)^{\phi(H)}=\phi(g^H)$, as required.

Therefore we can assume that $y = z g z^{-1}$ for some $z \in G$. If there existed an element
$f \in C_G(g) \cap z^{-1}H$, then $zf \in H$ and
$y=zgz^{-1}=(zf) g (zf)^{-1} \in g^H$, leading to a contradiction with our assumption on $y$.
Hence  $C_G(g) \cap z^{-1}H=\emptyset$, i.e., $z^{-1} \notin C_G(g) H$. Since $C_G(g) H$
is separable in $G$, there is $K \lhd G$ such that $|G:K|<\infty$ and $z^{-1} \notin C_G(g) HK$.
Now, the condition {\ccg} implies that there exists a finite index normal subgroup $L \lhd G$
such that $L \le K$ and $C_{G/L}(\psi(g)) \subseteq \psi(C_G(g)K)$, where $\psi:G \to G/L$ is
the natural epimorphism.

We claim that $\psi(y) \notin \psi(g^{H})$ in $G/L$. Indeed, if $\psi(y)= \psi(h g h^{-1})$ for some
$h \in H$, then $\psi(z^{-1} h) \in C_{G/L} (\psi(g))$. Hence
$\psi(z^{-1}) \in C_{G/L} (\psi(g)) \psi(H) \subseteq \psi (C_G(g) KH)$,
i.e., $z^{-1} \in  C_G(g) KHL=C_G(g) HK$ because $L \le K\lhd G$. But this yields a contradiction
with the construction of $K$.

Therefore we have found an epimorphism $\psi$ from $G$ to a finite group $G/L$
such that the image of $y$
does not belong to the image of $g^H$. Hence $g^H$ is separable in $G$.
\end{proof}

Observe that for a subgroup $H$ of a group $G$ and
any subset $S \subseteq H$, if $S$ is closed in $\PT(G)$, then $S$ is closed in
$\PT(H)$. Therefore Lemma \ref{lem:CCG->sep_c_c_for_sbgps} immediately implies the following:

\begin{cor} \label{cor:CC->_c_s_for_sbgps} Let $G$ be a conjugacy separable
group satisfying {\cc}, and let $H \le G$
be a subgroup such that $C_G(h)H$ is separable in $G$ for every $h \in H$.
Then $H$ is conjugacy separable. Moreover,
for each $h \in H$ the $H$-conjugacy class $h^H$ is closed in $\PT(G)$.
\end{cor}

It is not difficult to see that Lemma \ref{lem:CCG->sep_c_c_for_sbgps} has a partial converse
(we leave its proof as an exercise for the reader):

\begin{rem}\label{rem:sep_c_c->sep_centr_double_coset} Assume that $H$ is a subgroup of a
group $G$ and $g \in G$ is an arbitrary element.
If $g^H$ is separable in $G$ then the double coset $C_G(g)H$ is separable in $G$.
\end{rem}

In this paper we are going to use a different converse to Lemma \ref{lem:CCG->sep_c_c_for_sbgps}:
\begin{lemma}\label{lem:for_a_given_K_sep_cc->CCG} Let $G$ be a group.
Suppose that $H \le G$, $g \in G$, $K\lhd G$ and $|G:K|<\infty$. If the
subset $g^{H \cap K}$ is separable in $G$, then there is a finite index normal subgroup $L \lhd G$
such that $L \le K$ and $ C_{\bar H} (\bar g)  \subseteq \psi\left (C_H(g) K\right)$
in $G/L$ (in the notations of Definition \ref{df:CCG}).
\end{lemma}

\begin{proof} Denote $k:=|H:(H\cap K)|\le |G:K|<\infty$. Then $H=\bigsqcup_{i=1}^k z_i(H \cap K)$ for some $z_1,\dots,z_k \in H$. Renumbering the elements
$z_i$, if necessary, we can suppose that there is $l \in \{0,1,\dots,k\}$ such that whenever
$1 \le i \le l$, $z_i^{-1} g z_i \notin g^{H\cap K}$, and whenever
$l+1 \le j \le k$, $z_j^{-1} g z_j \in g^{H\cap K}$ in $G$.

By the assumptions, there exists a finite index normal subgroup $L \lhd G$ such that
$z_i^{-1} g z_i \notin g^{H\cap K} L$ whenever $1\le i \le l$.
Moreover, after replacing $L$ with $L \cap K$, we can assume that $L \le K$.

Let $\psi$ be the natural epimorphism from $G$ to $G/L$ and consider any element $\bar x \in C_\bH(\bg)$.
Then $\bar x = \psi(x)$ for some $x \in H$, and
$\psi(x^{-1}gx)=\psi(g)$ in $G/L$, i.e., $x^{-1} g x \in gL$ in $G$. As we know, there is
$i \in \{1,\dots,k\}$ and $y \in {H\cap K}$ such that $x = z_i y$. Consequently,
$z_i^{-1} g z_i \in y g Ly^{-1} = y g y^{-1}L \subseteq g^{H\cap K} L$. Hence, $i \ge l+1$, that is,
$z_i^{-1}g z_i = u g u^{-1}$ for some $u \in {H\cap K}$.

Thus $z_i u  \in C_H(g)$ and $x=z_i y = ( z_i u) (u^{-1}y) \in C_H(g) (H\cap K)
\subseteq C_H(g)K$. Therefore we proved that $\bx \in \psi(C_H(g)K)$ in $G/L$ for every $\bx \in C_\bH(\bg)$.
This yields the inclusion
$ C_{\bar H} (\bar g)  \subseteq \psi\left (C_H(g) K\right)$ in $G/L$, as required.
\end{proof}

We are now ready to prove Proposition \ref{prop:her_c_s-CC}.

\begin{proof}[Proof of Proposition \ref{prop:her_c_s-CC}]
First let us assume (b). Consider an arbitrary finite index subgroup $H \le G$. For every $h \in H$
the double coset $C_G(h)H$ is a finite union of left cosets modulo $H$, hence it is
separable in $G$. Therefore, by Corollary \ref{cor:CC->_c_s_for_sbgps}, $H$ is conjugacy separable.
That is, (b) implies (a).

Now, assume that $G$ is hereditarily conjugacy separable. We need to show that $G$ satisfies {\cc}.
Take any $g \in G$ and any $K \lhd G$ with $|G:K|<\infty$.
Observe that the subgroup $H:=K \langle g \rangle \le G$ has finite index in $G$, and $g^H=g^K=g^{H\cap K}$.
Since $H$ is conjugacy separable, $g^H$ is closed in $\PT(H)$, but then it is also closed in
$\PT(G)$ because any finite index subgroup of $H$ has finite index in $G$. Therefore
$g^{H\cap K}=g^H$ is separable in $G$, and so we can apply Lemma \ref{lem:for_a_given_K_sep_cc->CCG}
to find the finite index normal subgroup $L \lhd G$ from its claim.
Hence the group $G$ satisfies \cc.
\end{proof}


\section{Commuting retractions}\label{sec:comm_retr}
In this section we establish certain properties of commuting retractions
that constitute the core of our approach to studying residual properties of
right angled Artin groups. This approach is based on a simple observation
that canonical retractions of a right angled Artin group onto its special subgroups
pairwise commute (see Remark \ref{rem:spec_retr-comm} in Section \ref{sec:RAAG}).

Let $G$ be a group and let $H$ be a subgroup of $G$.
Recall that an endomorphism $\rho_H: G \to G$ is called a \textit{retraction} of $G$ onto $H$ if
$\rho_H(G)=H$ and $\rho_H(h)=h$ for every $h \in H$. In this case $H$ is said to be a \textit{retract} of $G$.
Note that $\rho_H \circ \rho_H=\rho_H$. If $H$ is a retract and $g \in G$, then the subgroup
$F:=gHg^{-1} \le G$, conjugate to $H$ in $G$, is also a retract.
The corresponding retraction $\rho_F \in End(G)$
is given by the formula $\rho_F(x):=g \rho_H(g^{-1} x g) g^{-1}$ for all $x \in G$.

The following observation is very useful:
\begin{lemma} \label{lem:induced_retr} Let $H$ be a retract of a group $G$ and let $\rho_H: G \to G$ be the
corresponding retraction. Suppose that $M\lhd G$ satisfies $\rho_H(M) \subseteq M$. Then the retraction $\rho_H$
canonically induces a retraction $\rho_{\bar H}: G/M \to G/M$ of $G/M$ onto the natural image $\bar{H}$ of $H$ in
$G/M$, defined by the formula $\rho_{\bar H}(gM)=\rho_H(g)M$ for all $gM \in G/M$.
\end{lemma}

\begin{proof} Evidently, it is enough to check that $\rho_{\bar{H}}$ is well-defined. If $g_1M=g_2M$ for some $g_1,g_2 \in G$,
then $f=g_2^{-1}g_1 \in M$, $g_1=g_2 f$ and $\rho_H(f) \in M$.
Hence $$\rho_{\bar{H}}(g_1M)=\rho_H(g_1)M=\rho_H(g_2) \rho_H(f)M =\rho_H(g_2)M=\rho_{\bar{H}}(g_2M) ,$$
as required.
\end{proof}

Assume that $H$ and $F$ are two retracts of a group $G$ and $\rho_H,\rho_F\in End(G)$ are the corresponding retractions. We will say
$\rho_H$ {\it commutes} with $\rho_F$ if they commute as elements of the monoid of
endomorphisms $End(G)$, i.e., if $\rho_H ( \rho_F(g))=\rho_F ( \rho_H (g))$ for all $g \in G$.

\begin{rem}\label{rem:inter_comm_retr} If the retractions $\rho_H$ and $\rho_F$ commute then
$\rho_H(F)=H \cap F= \rho_F(H)$ and
the endomorphism $\rho_{H \cap F}:=\rho_H \circ \rho_F=\rho_F \circ \rho_H$
is a retraction of $G$ onto $H \cap F$.
\end{rem}

Indeed, obviously the restriction of $\rho_{H\cap F}$ to ${H\cap F}$ is the identity map.
And $\rho_{H\cap F}(G) \subseteq \rho_H(G) \cap \rho_F(G)=H \cap F$, hence $\rho_{H\cap F}(G)=H \cap F$.
Consequently $\rho_H(F)=\rho_H(\rho_F(G))=\rho_{H \cap F}(G)=H \cap F$. Similarly, $\rho_F(H)=H \cap F$.

In the next proposition we establish an important property of commuting retractions that could be of independent interest.

\begin{prop}\label{prop:comm-retr} Let $H_1,\dots,H_m$ be retracts of a group $G$ such that the corresponding retractions
$\rho_{H_1},\dots,\rho_{H_m}$ pairwise commute. Then for any finite index normal subgroup $K \lhd G$
there is a finite index normal subgroup $M \lhd G$ such that $M \le K$ and $\rho_{H_i}(M) \subseteq M$ for each $i=1,\dots,m$.
Consequently, for every $i=1,\dots,m$, the retraction $\rho_{H_i}$ canonically induces a retraction $\rho_{\bar{H_i}}$ of $G/M$ onto the image
$\bar{H_i}$ of $H_i$ in $G/M$.
\end{prop}

\begin{proof} The second claim of the proposition follows from Lemma \ref{lem:induced_retr}, so it suffices to construct
the subgroup $M \lhd G$ with the needed properties.

If $J=\{i_1,\dots,i_k\}$ is a subset of the finite set $I:=\{1,2,\dots,m\}$, we define the retraction
$\rho_J$ of $G$ onto $\bigcap_{j \in J} H_j$ by $$\rho_J:=\rho_{H_{i_1}}\circ \rho_{H_{i_2}} \circ \dots \circ \rho_{H_{i_k}}.$$
This makes sense since our retractions pairwise commute. When $J=\emptyset$, $\rho_J$ will be the identity map of $G$.

Now, for every subset $J$ of $I=\{1,2,\dots,m\}$ we define the subgroup $D_J \le G$
as follows. First we set
$D_I:=\bigcap_{i=1}^m H_i \cap K$ -- a finite index normal subgroup of $(H_1 \cap \dots \cap H_m)$.
Next, if $J$ is a proper
subset  of I, we define $D_J$ recursively, according to the following formula:
\begin{equation}
\label{eq:def_of_D}
D_J:=\rho_J \left( \bigcap_{i \in I\setminus J} \rho^{-1}_{J \cup \{i\}} (D_{J \cup \{i\}}) \right) \cap K,
\end{equation}
where $\rho^{-1}_{J \cup \{i\}} (D_{J \cup \{i\}})$ denotes the full preimage (under $\rho_{J \cup \{i\}}$) of
$D_{J \cup \{i\}}$ in $G$.

Since the intersection of a finite number of finite index normal subgroups is again a finite index normal subgroup, and
images, as well as full preimages, of finite index normal subgroups under homomorphisms are again normal and of finite index (in their respective groups),
we see that $D_J$ is normal and has finite index in $\rho_J(G)=\bigcap_{j\in J} H_j$. Thus, if we set
$M:=D_\emptyset=\bigcap_{i \in I} \rho^{-1}_{H_i} (D_{\{i\}}) \cap K$, we shall have $M \lhd G$, $|G:M|<\infty$ and $M \le K$.

If $J \subset I$ and $i \in I\setminus J$, using \eqref{eq:def_of_D} and the fact that $\rho_{\{i\}}\circ \rho_J=\rho_{J \cup \{i\}}$, we  can observe that $\rho_{\{i\}} (D_J) \subseteq D_{J\cup \{i\}}$.

On the other hand,
let us show that $D_{J\cup \{i\}} \subseteq D_{J}$. We will use induction on the cardinality $|I\setminus J|$.
If $|I\setminus J|=1$ then $I=J \sqcup \{i\}$. And if $g \in D_{J\cup \{i\}}=D_I= \bigcap_{i\in I} H_i \cap K$, then $\rho_{I}(g)=g$, therefore
$g \in \rho_{I}^{-1}(D_I)$ and $g=\rho_J(g) \in \rho_J\left(\rho_{I}^{-1}(D_I)\right)$. Thus $g \in D_J$.

Suppose, now, that the statement has been proved for all proper subsets $J'$ of $I$ with $|J'|>|J|$.
Take any $i \in I\setminus J$ and consider an element $g \in D_{J\cup \{i\}}\le \bigcap_{j\in J \cup \{i\}} H_j \cap K$.
Then $\rho_{J \cup \{i\}}(g)=g$, therefore
$g \in \rho_{J \cup \{i\}}^{-1}(D_{J\cup \{i\}})$. We need to show that for any $i' \in I\setminus (J \cup \{i\})$,
$g \in \rho_{J \cup \{i'\}}^{-1}(D_{J\cup \{i'\}})$, or, equivalently, that $\rho_{J \cup \{i'\}}(g) \in D_{J\cup \{i'\}}$.
But  $$\rho_{J \cup \{i'\}}(g)=\rho_{i'}(\rho_J(g))=\rho_{\{i'\}}(g) \in \rho_{\{i'\}}(D_{J \cup \{i\}})
\subseteq D_{J\cup \{i,i'\}}.$$
And, since $D_{J\cup \{i,i'\}} \subseteq D_{J\cup \{i'\}}$ by the induction hypothesis, we can conclude that
$g \in \bigcap_{i' \in I\setminus J} \rho^{-1}_{J \cup \{i'\}} (D_{J \cup \{i'\}})$. Recalling that
$g \in \bigcap_{j\in J} H_j \cap K$, we achieve $g=\rho_J(g) \in D_J$. Thus $D_{J \cup \{i\}} \subseteq D_J$
and the inductive step is established.

We are now able to show that $\rho_{H_i}(M) \subseteq M$ for every $i \in I$. Indeed, since
$\rho_{H_i}(M) \subseteq D_{\{i\}}$, it is enough to check that $D_{\{i\}} \subseteq M$. Take any $j \in I$.
As we proved,
$\rho_{H_j}(D_{\{i\}})= \rho_{\{j\}}(D_{\{i\}}) \subseteq D_{\{i\} \cup \{j\}} \subseteq D_{\{j\}}$. Therefore,
$D_{\{i\}} \subseteq \rho^{-1}_{H_j}(D_{\{j\}})$ for each $j \in J$. By definition, $D_{\{i\}} \le K$,
consequently, for any $i \in I$, we achieve
$$\rho_{H_i}(M) \subseteq D_{\{i\}} \subseteq \bigcap_{i \in I} \rho^{-1}_{H_i} (D_{\{i\}}) \cap K=M,$$ as required.
\end{proof}

The next observation is an easy consequence of the definition of $M$ using the formula \eqref{eq:def_of_D}.
\begin{rem} In Proposition \ref{prop:comm-retr}, if $G/K$ is a finite $p$-group for some prime number $p$,
then so is $G/M$.
\end{rem}

Given two subgroups $H$ and $F$ of a group $G$, it is usually difficult to find quotient-groups
$Q$ of $G$ such that the image of the intersection of $H$ and $F$ in $Q$ coincides with the
intersection of the images of these subgroups in $Q$.
However, in the case when $H$ and $F$ are retracts and the corresponding retractions commute this will be automatic for many quotients of $G$.

\begin{lemma} \label{lem:pres-intersec} Suppose that the retractions $\rho_H,\rho_F \in End(G)$ commute,
and $M\lhd G$ is a normal subgroup
satisfying $\rho_H(M) \subseteq M$ and $\rho_F(M) \subseteq M$. Then $\varphi(H\cap F)=\varphi(H) \cap \varphi(F)$ in $G/M$, where $\varphi: G \to G/M$
is the natural epimorphism.
\end{lemma}

\begin{proof} By Lemma \ref{lem:induced_retr} $\rho_H$ and $\rho_F$ canonically induce retractions $\rho_{\varphi(H)}$ and $\rho_{\varphi(F)}$ of $G/M$ onto $\varphi(H)$
and $\varphi(F)$ respectively.

Clearly, $\varphi(H\cap F) \subseteq \varphi(H) \cap \varphi(F)$, and so, we only need to establish the inverse inclusion. Consider an arbitrary
$\bar g \in \varphi(H) \cap \varphi(F)$. Then $\bar g =\varphi(g)$ for some $g \in G$, and $\rho_{\varphi (F)} (\bar g)= \bar g$,
$\rho_{\varphi (H)} (\bar g)=\bar g$. Therefore
$$\bar g=\rho_{\varphi (H)}\left(\rho_{\varphi(F)}(\varphi(g)) \right)=\rho_{\varphi (H)}\left(\varphi (\rho_F(g)) \right)=
\varphi \left( \rho_H(\rho_F(g))\right) \in \varphi(H\cap F),$$
where the last inclusion follows from Remark \ref{rem:inter_comm_retr}. Thus $\varphi(H) \cap \varphi(F) \subseteq \varphi(H\cap F)$.
\end{proof}

Lemma \ref{lem:pres-intersec} allows to obtain the first interesting application of Proposition \ref{prop:comm-retr}.

\begin{cor}\label{cor:fin_image_of_intersec}
 Let $H_1,\dots,H_m$ be retracts of a group $G$ such that the corresponding retractions
$\rho_{H_1},\dots,\rho_{H_m}$ pairwise commute. Then for any finite index normal subgroup $K \lhd G$
there is a finite index normal subgroup $M \lhd G$ such that $M \le K$ and $\rho_{H_i}(M) \subseteq M$ for each $i=1,\dots,m$.
Moreover, if $\varphi: G \to G/M$ denotes the natural epimorphism, then
$\varphi(\bigcap_{i=1}^m H_i)=\bigcap_{i=1}^m \varphi(H_i)$.
\end{cor}

\begin{proof} First we apply Proposition \ref{prop:comm-retr} to find the finite index normal subgroup $M$ from its claim.
The last statement of the corollary will be proved by induction on $m$.
If $m=1$ there is nothing to prove. So let us assume that $m \ge 2$ and we have already shown that
$\varphi(\bigcap_{i=1}^{m-1} H_i)=\bigcap_{i=1}^{m-1} \varphi(H_i)$. Using Remark \ref{rem:inter_comm_retr} we see that
the map $\rho_F :=\rho_{H_1} \circ \dots \circ \rho_{H_{m-1}} \in End(G)$ is a retraction of $G$ onto $F:=\bigcap_{i=1}^{m-1} H_i$.
By Proposition \ref{prop:comm-retr}, $\rho_{H_i}(M) \subseteq M$ for each $i=1,\dots,m$, therefore
\begin{multline*}
\rho_F (M) =(\rho_{H_1} \circ \dots \circ \rho_{H_{m-2}}) \bigl(\rho_{H_{m-1}}(M) \bigr) \subseteq \\(\rho_{H_1} \circ \dots \circ \rho_{H_{m-3}})
\bigl(\rho_{H_{m-2}}(M)\bigr) \subseteq
\dots \subseteq \rho_{H_{1}}(M) \subseteq M.$$
\end{multline*}

By the assumptions, the retractions $\rho_F$ and $\rho_{H_m}$ commute, hence we can apply Lemma~\ref{lem:pres-intersec} to conclude that
$\varphi(F\cap H_m)=\varphi(F) \cap \varphi(H_m)$. But $\varphi(F)=\bigcap_{i=1}^{m-1} \varphi(H_i)$ by the induction hypothesis, consequently
$\varphi(\bigcap_{i=1}^m H_i)=\varphi(F\cap H_m)=\bigcap_{i=1}^m \varphi(H_i)$, and the proof is finished.
\end{proof}

Let us now give an example which shows that the statements of Corollary \ref{cor:fin_image_of_intersec}
and Lemma \ref{lem:pres-intersec} are no longer true if the retractions do not commute.

\begin{ex}\label{ex:non-comm_retr} Let $S$ be any infinite simple group, and let $H$ be an arbitrary group possessing
non-trivial finite quotients.
Set $G:=H * S$, fix an element $s \in S\setminus \{1\}$ and denote $F:=s H s^{-1} \le G$. Evidently $H$ is a
retract of $G$, where the retraction $\rho_H: G \to G$ of
$G$ onto $H$ is the identity on $H$ and trivial on $S$. Clearly the endomorphism $\rho_F \in End(G)$ defined by
$\rho_F(g):=s \rho_H  (s^{-1} g s) s^{-1}$ for every $g \in G$, is a retraction of $G$ onto $F$.

It is not difficult to see that the retractions $\rho_H$ and $\rho_F$ do not commute
(for instance, because $(\rho_H \circ \rho_F)(G)=H$, $(\rho_F \circ \rho_H)(G)=F$ and $H \cap F= \{1\}$).

If $K \lhd G$ is an arbitrary proper normal subgroup of finite index, then $S \subset K$
(because $S$ has no non-trivial
finite quotients), hence the kernel $\ker(\rho_H)$ (which is equal to the normal closure of $S$ in $G$) is
contained in $K$. Consequently, $\rho_H(K) \subseteq \rho_H^{-1}(\rho_H(K)) \subseteq K$.
Similarly, $\rho_F(K) \subseteq K$.

Observe  that $H \cap F= \{1\}$ by construction. Denote by $Q$ the quotient $G/K$ and let
$\varphi:G \to Q$ be the natural epimorphism.
Since $s \in S \le  \ker(\varphi)$ we see that
$$\varphi(H) \cap \varphi(F)=\varphi(H)=Q \neq \{1\}=\varphi(H \cap F).$$
That is, in any non-trivial finite quotient $Q$ of $G$ the intersection of the images
of $H$ and $F$ is strictly larger than the image of $H \cap F$.
\end{ex}


\section{Implications for the profinite topology}\label{sec:conseq_prof_top}
Throughout this section we will assume that $A$ and $B$ are retracts of a group $G$ such that
the corresponding retractions $\rho_A \in End(G)$ and $\rho_B \in End(G)$ commute.
Our goal here is to establish several consequences of these settings for the profinite topology on $G$.

\begin{lemma} \label{lem:empty-inter} 
For arbitrary elements $x,y \in G$ define
$\alpha := \rho_A\left(\rho_B(x) x^{-1}\right) x \rho_B \left(x^{-1}\right) \in A x B \subseteq G$ and
$\beta := \rho_A\left(\rho_B(y) y^{-1}\right) y \rho_B\left(y^{-1}\right) \in A y B \subseteq G$.
Then the following two conditions
are equivalent:

\begin{itemize}
\item[\rm (i)] $y \in A x B$;
\item[\rm (ii)] $\beta \in \alpha^{A \cap B}$.
\end{itemize}
\end{lemma}

\begin{proof} Observe that $y \in A x B$ if and only if $A y B = A x B$, which is equivalent to
$A \beta B=A \alpha B$. 
Thus $y \in A x B$ if and only if  $\beta \in A \alpha B$.

To show that (i) implies (ii), suppose that there are $a \in A$ and $b \in B$ such that
\begin{equation}\label{eq:bet-alph} \beta=a \alpha b. \end{equation}
By definition, $\rho_A( \alpha)=1=\rho_A(\beta)$, hence $1=\rho_A(a) \rho_A(b)$. Therefore, \eqref{eq:bet-alph} implies that
$a = \rho_A(a) =\rho_A(b^{-1}) \in \rho_A(B) = A \cap B$ (by Remark \ref{rem:inter_comm_retr}).

Now, since $\rho_B \circ \rho_A=\rho_A \circ \rho_B$ we have $\rho_B(\alpha)=1=\rho_B(\beta)$. Therefore,
applying $\rho_B$ to both sides of the equality \eqref{eq:bet-alph}, we get
$1=\rho_B(a) \rho_B(b)=a b$ because $a,b \in B$. Hence $b=a^{-1} \in A \cap B$ and
$\beta=a \alpha a^{-1} \in \alpha^{A \cap B}$.

Now, suppose that $\beta \in \alpha^{A \cap B}$. Then $\beta \in (A \cap B) \alpha (A \cap B) \subseteq
A \alpha B$. Thus (ii) implies (i).
\end{proof}

Let us look at the proof of the above lemma in the particular case when $y=x$.
Then we see that $\beta = \alpha$, and
$$A \cap xBx^{-1}= \gamma^{-1} \left( A \cap \alpha B \alpha^{-1} \right) \gamma, ~\mbox{ where }
\gamma:=\rho_A\left(\rho_B(x) x^{-1}\right) \in A. $$

We also see that $a \in A \cap \alpha B \alpha^{-1}$ if and only if there is $b \in B$ such that
$\alpha = a^{-1} \alpha b$. But, as we showed in the proof of
Lemma~\ref{lem:empty-inter}, this can happen only if
$b=a \in A \cap B$. I.e,  $\alpha a = a \alpha$ and $a \in A \cap B$,
which is equivalent to $a \in C_{A \cap B}(\alpha)$. Thus, in this particular case
we obtain the following statement:

\begin{lemma}\label{lem:inter_of_conjug} If 
$x \in G$ is an arbitrary element, then
$$A \cap xBx^{-1}= \gamma^{-1} C_{A \cap B}(\alpha) \gamma ~\mbox{ in } G,$$ { where }
$\alpha := \rho_A\left(\rho_B(x) x^{-1}\right) x \rho_B \left(x^{-1}\right) \in A x B$ { and }~
$\gamma:=\rho_A\left(\rho_B(x) x^{-1}\right) \in A$ .
\end{lemma}

Combining Lemma \ref{lem:empty-inter} with Corollary \ref{cor:fin_image_of_intersec}
we achieve

\begin{lemma}\label{lem:double_coset_sep} 
Consider any $x \in G$ and denote $\alpha := \rho_A\left(\rho_B(x) x^{-1}\right) x \rho_B \left(x^{-1}\right) \in A x B \subseteq G$. If the conjugacy class $\alpha^{A \cap B}$
is separable in $G$, then the double coset $AxB$ is also separable in $G$.
\end{lemma}

\begin{proof} Suppose that an element $y \in G$ satisfies $y \notin AxB$. By Lemma \ref{lem:empty-inter},
this is equivalent to $\beta \notin \alpha^{A \cap B}$, where
$\beta := \rho_A\left(\rho_B(y) y^{-1}\right) y \rho_B\left(y^{-1}\right)$.
Since $\alpha^{A \cap B}$ is separable, there is a finite index normal subgroup
$K \lhd G$ such that $\psi(\beta) \notin \psi \left(\alpha^{A \cap B} \right)=
\psi (\alpha)^{\psi(A \cap B)}$,
where $\psi:G \to G/K$ is the canonical epimorphism.

By Corollary  \ref{cor:fin_image_of_intersec}, there exists a finite index normal subgroup $M \lhd G$ such that
$M \le K$, $\rho_A(M) \subseteq M$, $\rho_B(M) \subseteq M$ and $\varphi(A \cap B)=\varphi(A) \cap \varphi(B)$,
where $\varphi$ is the natural epimorphism from $G$ to $G/M$.
Since $\psi$ factors through $\varphi$, we can conclude that
$\varphi(\beta) \notin \varphi (\alpha)^{\varphi(A \cap B)}=
\varphi (\alpha)^{\varphi(A) \cap \varphi(B)}$. But by Lemma \ref{lem:induced_retr}, there are canonically
induced  commuting retractions $\rho_{\bar A}$ and $\rho_{\bar B}$ of $G/M$ onto $\bar A:=\varphi(A)$
and $\bar B:=\varphi(B)$ respectively. Moreover, letting $\bar x := \varphi(x)$, $\bar y := \varphi(y)$
and using the definition of the retractions $\rho_{\bar A}$ and $\rho_{\bar B}$, we obtain
$\varphi(\alpha)=\rho_{\bar A}\left(\rho_{\bar B}(\bar x) \bar x^{-1}\right) \bar x
\rho_{\bar B} \left(\bar x^{-1}\right)$ and $\varphi(\beta)=\rho_{\bar A}\left(\rho_{\bar B}(\bar y) \bar y^{-1}\right) \bar y
\rho_{\bar B} \left(\bar y^{-1}\right)$. Therefore, by Lemma \ref{lem:empty-inter}, applied to
the retracts $\bar A$ and $\bar B$ in $G/M$, we have $\bar y \notin \bar A \bar x \bar B$.
That is, $\varphi(y) \notin \varphi(AxB)$. Hence the double coset $AxB$ is separable in $G$.
\end{proof}

Since the $1^{A\cap B}=\{1\}$ is separable in $G$ whenever $G$ is residually finite, we have the following
immediate consequence of Lemma \ref{lem:double_coset_sep}.

\begin{cor}\label{cor:prod_retr-sep} If $A$ and $B$ are retracts of a residually finite group $G$ such
that the corresponding retractions commute, then the double coset $AB$ is separable in $G$.
\end{cor}

The statement of Corollary \ref{cor:prod_retr-sep} has been known before -- see, for example,
\cite[Lemma 9.3]{H-W_1}, but the proof that we have presented here is new.

The following statement is well-known:
\begin{lemma}\label{lem:sep_subset_of_retr} Suppose that $G$ is a residually finite group
and $A \le G$ is a retract of $G$. If a subset $S \subseteq A$ is closed in
$\PT(A)$, then $S$ is closed in $\PT(G)$.
\end{lemma}

\begin{proof} We will show that $S$ coincides with its closure $cl_G(S)$ in the profinite topology on $G$.
By Corollary \ref{cor:prod_retr-sep} the subgroup $A=AA$ is closed in $\PT(G)$,
hence $cl_G(S) \subseteq A$. Now, if $a \in A \setminus S$, then there is a homomorphism
$\phi: A \to Q$ from $A$ to a finite group $Q$ such that $\phi(a) \notin \phi(S)$.
Since $A$ is a retract of $G$, we have a homomorphism $\psi: G \to Q$ defined by
$\psi:=\phi\circ\rho_A$. Evidently $\psi(a)=\phi(a) \notin \phi(S)=\psi(S)$, hence $a \notin cl_G(S)$.
Thus $S=cl_G(S)$, as required.
\end{proof}

Now, the reason why we need Lemma \ref{lem:inter_of_conjug}, is because it tells us that if one can control
the $A \cap B$-centralizers in $G$, then one can also control the intersections of conjugates of the retracts
$A$ and $B$. As it can be seen from Example \ref{ex:non-comm_retr}, in general we may not be able to
find a finite quotient $Q$ of $G$, in which the image of the intersection of two particular
conjugates of $A$ and $B$ is equal to the intersection of their images. However, provided that a certain
Centralizer Condition is satisfied, we can find many finite quotients $Q$ of $G$ where these two sets
are very close to each other.

\begin{lemma}\label{lem:cond_(I)} Let $x$ be an element of $G$ and let $\alpha:=\rho_A\left(\rho_B(x) x^{-1}\right) x \rho_B \left(x^{-1}\right) \in G$. Suppose that the pair $(A\cap B, \alpha)$ satisfies the
Centralizer Condition in $G$. Then for any finite index normal subgroup $K \lhd G$ there exists
a finite index normal subgroup $M \lhd G$ such that $M \le K$, $\rho_A(M) \subseteq M$,
$\rho_B(M) \subseteq M$ and
$\varphi(A) \cap \varphi(xBx^{-1}) \subseteq \varphi(A \cap xBx^{-1}) \varphi(K)$ in $G/M$, where
$\varphi:G \to G/M$ is the natural epimorphism.
\end{lemma}

\begin{proof} By Lemma \ref{lem:inter_of_conjug}, $A \cap xBx^{-1}= \gamma^{-1} C_{A \cap B}(\alpha) \gamma$,
where $\gamma:=\rho_A\left(\rho_B(x) x^{-1}\right) \in A$.
Since the pair $(A\cap B, \alpha)$ has \ccg, there is a subgroup $L \lhd G$ of finite
index in $G$, such that $L \le K$  and
$\psi^{-1} \left ( C_{\psi(A \cap B)} (\psi(\alpha)) \right) \subseteq C_{A \cap B}(\alpha) K$ in $G$, where
$\psi: G \to G/L$ is the natural epimorphism. Applying Corollary \ref{cor:fin_image_of_intersec}
to $A$, $B$ and $L$ we find a finite index normal subgroup $M \lhd G$, together with the
epimorphism $\varphi:G \to G/M$, such that
$M \le L \le K$, $\rho_A(M) \subseteq M$, $\rho_B(M) \subseteq M$ and
$\varphi(A) \cap \varphi(B) =\varphi(A \cap B)$.

By Lemma \ref{lem:induced_retr}, $\rho_A$ and $\rho_B$ canonically induce retractions $\rho_\bA$ and
$\rho_\bB$ of $G/M$ onto $\bA:=\varphi(A)$ and $\bB:=\varphi(B)$ respectively.
Obviously $\rho_\bA$ commutes with $\rho_\bB$ in $End(G/M)$, because $\rho_A$ commutes with $\rho_B$
in $End(G)$.

Denote $\bar x:= \varphi(x)$, $\bar \alpha = \rho_\bA
\left(\rho_\bB(\bar x) \bar x^{-1}\right) \bar x \rho_\bB \left(\bar x^{-1}\right) \in G/M$ and
$\bar \gamma:=\rho_\bA\left(\rho_\bB(\bx) \bx^{-1}\right) \in \bA$. Observe that
$\bar \alpha = \varphi(\alpha)$ and $\bar \gamma =\varphi(\gamma)$ by the definitions of $\rho_\bA$ and
$\rho_\bB$. Then by Lemma~\ref{lem:inter_of_conjug},
$\bA \cap \bx \bB \bx^{-1} =\bar \gamma^{-1} C_{\bA \cap \bB} (\bar \alpha) \bar \gamma$ in $G/M$.
Therefore, recalling that $\bA \cap \bB= \varphi(A \cap B)$, we get
$$  \varphi^{-1} \left( \bA \cap \bx \bB \bx^{-1} \right)=
\varphi^{-1} \left( \bar \gamma^{-1} C_{\bA \cap \bB} (\bar \alpha) \bar \gamma \right)=
\gamma^{-1} \varphi^{-1} \left( C_{\varphi(A \cap B)} (\bar \alpha) \right)  \gamma.$$

But since $\psi$ factors through $\varphi$ (as $\ker(\varphi)=M \le L=\ker(\psi)$), we obviously have
$$\varphi^{-1} \left( C_{\varphi(A \cap B)} (\bar \alpha) \right) \subseteq
\psi^{-1} \left ( C_{\psi(A \cap B)} (\psi(\alpha)) \right) \subseteq C_{A \cap B}(\alpha) K.$$
Hence we can
conclude that $\varphi^{-1} \left( \bA \cap \bx \bB \bx^{-1} \right)
\subseteq \gamma^{-1} C_{A \cap B}(\alpha) \gamma K= \left( A \cap xBx^{-1} \right) K$.
Consequently, $\varphi(A) \cap \varphi (xBx^{-1})=\bA \cap \bx \bB \bx^{-1}
\subseteq \varphi(A \cap xBx^{-1}) \varphi(K)$, and the lemma is proved.
\end{proof}

In this paper we will need one more criterion for separability of specific double cosets in $G$. In
a certain sense it generalizes Remark \ref{rem:sep_c_c->sep_centr_double_coset}.

\begin{lemma}\label{lem:centr-coset_sep} Consider arbitrary elements  $x,g \in G$. Denote $D:=xBx^{-1} \le G$ and $\alpha:=\rho_A\left(\rho_B(x) x^{-1}\right) x \rho_B \left(x^{-1}\right) \in G$.
Suppose that the conjugacy classes $\alpha^{A \cap B}$ and $g^{A \cap D}$ are separable
in $G$, and the pair $(A\cap B,\alpha)$ satisfies \ccg. Then the double coset $C_A(g) D$ is separable in $G$.
\end{lemma}

\begin{proof} Consider any $z \in G$ with $z \notin C_A(g) D$. First, suppose that $z \notin AD$.
Since $\alpha^{A \cap B}$ is separable in $G$, Lemma \ref{lem:double_coset_sep} implies that $AxB$
is separable, hence $AD=(AxB) x^{-1}$ is separable as well (because multiplication by a
fixed group element on the right is a homeomorphisms of $G$ with respect to the profinite topology).
Therefore there is a finite index normal subgroup $N\lhd G$ such that $z\notin ADN$, hence
$z\notin C_A(g)D N$ because $C_A(g) \le A$.

Thus we can assume that $z \in AD$, i.e., there exist $a_0\in A$ and $d_0 \in D$ such that $z=a_0d_0$.
Since $z \notin C_A(g) D$, $y:=z d_0^{-1} \notin C_A(g)(A \cap D)$. Consequently,
for every $h \in A \cap D$, $(yh) g (yh)^{-1} \neq g$, i.e., $y^{-1} g y \neq h g h^{-1}$,
implying that $y^{-1}g y \notin g^{A\cap D}$ in $G$.

Now, the separability of $g^{A\cap D}$ in $G$ implies that there is a finite index normal subgroup
$K \lhd G$ such that $y^{-1} g y \notin g^{A \cap D}K$. And, by Lemma \ref{lem:cond_(I)}, we can find
a finite index normal subgroup $M \lhd G$ such that $M \le K$ and
$\varphi(A) \cap \varphi(D) \subseteq \varphi(A \cap D) \varphi(K)$, where $\varphi:G \to G/M$ is the
natural epimorphism. Let $\bK$, $\bA$, $\bD$, $\by$ and $\bg$ denote the $\varphi$-images of
$K$, $A$, $D$, $y$ and $g$ respectively.

Since $\bK \lhd G/M$, we have $\bg^{\bA \cap \bD} \subseteq \bg^{\varphi(A \cap D)\bK}
\subseteq \bg^{\varphi(A \cap D)}\bK$ and $\by^{-1} \bg \by \notin  \bg^{\varphi(A \cap D)}\bK$ as $M \le K$.
Hence $\by^{-1} \bg \by \notin \bg^{\bA \cap \bD}$.

To finish the proof, it remains to show that $\varphi(z) \notin \varphi (C_A(g) D)$. Suppose, on the contrary,
that there exist $a \in C_A(g)$ and $d \in D$ such that $\varphi(z)=\varphi(ad)$. Then
$\varphi(a_0d_0)=\varphi(ad)$, thus $\bh:=\varphi(a^{-1}a_0)=\varphi(dd_0^{-1}) \in \bA \cap \bD$, and
$\varphi(z)=\varphi(a) \bh \varphi(d_0)$. Consequently,
$\by=\varphi(z) \varphi(d_0^{-1}) =\varphi(a) \bh$ and $$\by^{-1} \bg \by=
\bh^{-1} \varphi(a^{-1} g a) \bh =\bh^{-1} \bg \bh\in \bg^{\bA \cap \bD},$$
contradicting to our construction.

Thus, for every $z \notin C_A(g) D$ we found $M \lhd G$ with
$|G:M|<\infty$ such that $z \notin C_A(g) D M$. Therefore the double coset $C_A(g) D$ is separable in $G$.
\end{proof}


\section{Some properties of right angled Artin groups}\label{sec:RAAG}
In this section we recall a few properties
of right angled Artin groups, which will be used in
the proof of the main result. At the end of the section we prove that every
pointwise inner automorphism of a right angled Artin group is inner.

Let $\Gamma$ be a finite graph (without loops or multiple edges) with the set of vertices $\mathcal{V}$.
For any vertex $v \in \mathcal{V}$ its \textit{star} $\star(v)$ consists of all vertices
(including $v$ itself) that are adjacent to
$v$ in $\Gamma$. If $\mathcal S \subseteq \mathcal V$, then
$\star(\mathcal S):=\bigcap_{v \in \mathcal S} \star(v)$. Observe that for two subsets
$\mathcal S, \mathcal T \subseteq \mathcal V$, $\mathcal T \subseteq \star(\mathcal S)$
happens if and only if $\mathcal S \subseteq \star(\mathcal T)$.

Let $G=G(\Gamma)$ be the associated right angled Artin group. To simplify
notation, we will identify elements of $\mathcal{V}$ with the corresponding generators of $G$.
Then for each $v \in \mathcal V$, $\star(v)$ contains precisely those elements from $\mathcal V$
that commute with $v$ in $G$.
For any subset $A$ of $G$, $A^{\pm 1}$ will denote the union $A \cup A^{-1} \subseteq G$.
Thus every element
$g \in G$ can be represented as a word $W$ in letters from $\mathcal{V}^{\pm 1}$.
The \textit{support} $\supp(W)$ is the set of all $v \in \mathcal{V}$ such that $v^{\pm 1}$ appears as a letter in $W$.
A word $W$ is said to be {\it graphically reduced} if it has no
subwords of the form $v U v^{-1}$ or $v^{-1} U v$, where $v \in \mathcal{V}$ and $\supp(U) \subseteq \star(v)$.
Evidently, if the word $W$ is not graphically reduced, then one can find a shorter word representing
the same element of the group $G$. This process will eventually terminate (because the length of
$W$ is finite), hence for each element $g \in G$ there exists a graphically reduced word representing it
in $G$.

E. Green \cite{Green} proved that if two graphically reduced words $W$ and $W'$
represent the same element $g \in G$, then $W$ and $W'$ have the same length and $\supp(W)=\supp(W')$.
Moreover, for any given $g \in G$,
graphically reduced words are precisely the shortest possible words representing $g$ in $G$
(proofs of these facts using re-writing systems can also be found in \cite[Sec. 2.2]{Esyp-Kaz-Rem}).
Therefore, for any element $g$ we can define its \textit{length} $|g|$ as the length of any graphically
reduced word $W$ representing $g$ in $G$, and its \textit{support} $\supp(g)$ as $\supp(W)$.

Finally, for any $g \in G$ define $\FL(g)$  -- \textit{the set of first letters of} $g$ -- as
the set of all letters $a \in \mathcal{V}^{\pm 1}$ such that $a$ appears as the first letter of some graphically reduced
word $W$ representing $g$ in $G$.  Similarly, we define \textit{the set of last letters} $\LL(g)$ of $g$
as those $a \in \mathcal{V}^{\pm 1}$ that appear as a last letter of some graphically reduced word
representing $g$ in $G$. A useful fact observed by Green in \cite{Green} states that for any
$g \in G$ the letters in $\FL(g)$ pairwise commute (in $G$). Evidently, $\LL(g)=(\FL(g^{-1}))^{-1}$.

Consider any subset $\mathcal{S}$ of $\mathcal{V}$ and let $\Delta$ be the full subgraph of $\Gamma$ on the vertices from $\mathcal{S}$.
Let $H$ denote the right angled Artin group
corresponding to $\Delta$.
The identity map on $\mathcal S$ can be regarded as a map from the generating set of $H$
into $G$. Since $\Delta$ is a subgraph of $\Gamma$ all the relations between these generators of
$H$ hold between their images in $G$.
Therefore, by von Dyck's Theorem, there is a homomorphism $\xi:H \to G$ extending the identity map
on $\mathcal S$.

On the other hand, since $\Delta$ is a full subgraph of $\Gamma$,
by von Dyck's Theorem, the map $\rho_{\mathcal{S}}: \mathcal{V} \sqcup \{1\}\to \mathcal{V} \sqcup \{1\}$ defined by $\rho_{\mathcal{S}}(1):=1$ and
\begin{equation} \label{eq:rho_S}
\rho_{\mathcal{S}}(v):=\left\{ \begin{array}{ll}
         v & \mbox{if $v \in \mathcal{S}$}\\
        1 & \mbox{if $v \in \mathcal{V} \setminus \mathcal{S}$}\end{array} \right.,
\end{equation}
can be extended to a homomorphism $\rho_H: G \to H$. 
Obviously, the composition $\rho_H \circ \xi: H \to H$ is the identity map on $H$. Therefore $\xi$
is injective, hence it is an isomorphism between $H$ and the subgroup of $G$ generated by
$\mathcal S$. Consequently, $\rho_H$, regarded as an endomorphism of $G$, becomes a (canonical) retraction
of $G$ onto $\langle \mathcal S \rangle \le G$.

For any $\mathcal{S} \subseteq \mathcal{V}$ the subgroup $H:=\langle \mathcal{S} \rangle\le G$ is called
\textit{special} (or \textit{full}, or \textit{canonically parabolic}, depending on the source).
Note that the trivial subgroup $\{ 1\} \le G$ is also special
and corresponds to the empty subset of $\mathcal V$.
 As we saw above, any special subgroup is a right angled Artin group
itself, and is a retract of $G$.
It is easy to see that if ${\mathcal{S}},{\mathcal{T}}$ are two subsets of $\mathcal{V}$
then the corresponding maps
$\rho_{\mathcal{S}}$ and $\rho_{\mathcal{T}}$ defined by \eqref{eq:rho_S} commute
with each other. This leads to the following important observation.
\begin{rem}\label{rem:spec_retr-comm} If $H$ and $F$ are special subgroups of a right angled Artin
group $G$, then $H$ and $F$ are retracts of $G$ and
the corresponding canonical retractions $\rho_H, \rho_F \in End(G)$ commute.
\end{rem}

\begin{rem} \label{rem:inter_special-sbgps}
If ${\mathcal{S}},{\mathcal{T}} \subseteq \mathcal{V}$ then $\langle {\mathcal{S}} \rangle \cap \langle {\mathcal{T}} \rangle =
\langle {\mathcal{S}} \cap {\mathcal{T}} \rangle.$
\end{rem}
Indeed, by Remarks \ref{rem:spec_retr-comm} and  \ref{rem:inter_comm_retr}, we have
$$\langle {\mathcal{S}} \rangle \cap \langle {\mathcal{T}} \rangle = \rho_{\langle {\mathcal{T}} \rangle} (\langle {\mathcal{S}} \rangle)=
\langle \rho_{\langle {\mathcal{T}} \rangle}({\mathcal{S}})\rangle = \langle \rho_{{\mathcal{T}} }({\mathcal{S}})\rangle=
\langle {\mathcal{S}} \cap {\mathcal{T}} \rangle.$$

Recall that a group $G$ is said to have the \textit{Unique Root property} if for any positive integer
$n$ and arbitrary elements $x,y \in G$ the equality $x^n=y^n$ implies $x=y$ in $G$. The group $G$ is called
\textit{bi-orderable} if $G$ can be endowed with a total order $\preceq$, which is \textit{bi-invariant},
i.e., for any $x,y,z \in G$, if $x \preceq y$, then $zx \preceq zy$ and
$xz \preceq yz$.

\begin{lemma}\label{lem:RAAG_URP} Right angled Artin groups have the Unique Root property.
\end{lemma}

\begin{proof} G. Duchamp and D. Krob \cite{Duch-Krob}
(see also \cite{Duch-Thib}) proved that right angled Artin groups are bi-orderable.
Let $G$ be a right angled Artin group, and let $\preceq$
be a total bi-invariant order on $G$.

Suppose that $x^n=y^n$ for some $x,y \in G$, $n \in \N$, and $x \neq y$. Without loss of
generality we can assume that $x \prec y$. Let us show that $x^k \prec y^k$ for every $k \in \N$.
This is true for $k=1$, so, proceeding by induction on $k$, suppose that $k \ge 2$ and
$x^{k-1} \prec y^{k-1}$ has already been shown. Then
$x^k=x^{k-1} x \prec x^{k-1} y \prec y^{k-1} y =y^k$, where we used the induction hypothesis
together with the bi-invariance of the order.

Hence, we have proved that $x^n \prec y^n$, contradicting to $x^n=y^n$. Thus $x=y$.
\end{proof}

The Unique Root property for right angled Artin groups can also be easily established
using the fact that these groups are residually torsion-free nilpotent, which was also proved in
\cite{Duch-Krob}.

\begin{lemma}\label{lem:special_sbgp_not_f_i} Let $H$ be a conjugate of a special subgroup in a
right angled Artin group $G$. If $K \le G$ is a subgroup such that $|K: (K\cap H)|<\infty$ then
$K \subseteq H$.
\end{lemma}

\begin{proof} By the assumptions, $H$ is a retract of $G$. Let $\rho_H \in End(G)$ denote the
corresponding retraction. Take any $x \in K$. Since  $|K: (K\cap H)|<\infty$, there is $n \in \N$
such that $x^n \in H$. Therefore, setting $y:=\rho_H(x) \in H$, we obtain
$x^n=\rho_H(x^n)=y^n$. And the Unique Root property for $G$ implies that $x=y \in H$. Thus $K \subseteq H$.
\end{proof}

After writing down the proof of the
next technical result (Lemma \ref{lem:inter_conj_to_spec_sbgps}), the author learned that it
has already been established by A. Duncan, I. Kazachkov and V. Remeslennikov in their
recent paper \cite[Prop. 2.6]{Dun-Kaz-Rem}.
However, the proof presented here is somewhat different, and the author
decided to keep it in this work for completeness.

\begin{lemma} \label{lem:inter_conj_to_spec_sbgps} Let $G$ be a right angled Artin group associated
to a finite graph $\Gamma$ with vertex set $\mathcal{V}$. Suppose that ${\mathcal{S}}, {\mathcal{T}} \subseteq \mathcal{V}$ and $g \in G$.
Then there are ${\mathcal{P}} \subseteq {\mathcal{T}}$ and $h \in \langle {\mathcal{T}} \rangle$ such that
$g \langle {\mathcal{S}} \rangle g^{-1} \cap \langle {\mathcal{T}} \rangle= h \langle {\mathcal{P}} \rangle h^{-1}$ in $G$. Consequently,
the intersection of conjugates of two special subgroups in $G$ is a conjugate of a special subgroup of $G$.
\end{lemma}

\begin{proof} We will use induction on $(|\mathcal S|+|g|)$, where $|\mathcal S|$ denotes the cardinality of
$\mathcal S$. If $\mathcal S = \emptyset$, then $\langle {\mathcal{S}} \rangle=\{1\}$ and the statement
is trivial. If $|g|=0$, i.e., $g=1$ in $G$, then
$g \langle {\mathcal{S}} \rangle g^{-1} \cap \langle {\mathcal{T}} \rangle=
 \langle {\mathcal{S}} \rangle  \cap \langle {\mathcal{T}} \rangle= \langle \mathcal{S} \cap
{\mathcal{T}} \rangle$ by Remark \ref{rem:inter_special-sbgps}.

Thus we can assume that $\mathcal S \neq \emptyset$ and $g \neq 1$, $n:=|\mathcal S|+|g|\ge 2$,
and the claim has been proved for all $\mathcal{S}$ and $g$ with  $|\mathcal S|+|g| < n$.

If there is $a \in \FL(g) \cap \mathcal{T}^{\pm 1}$, set $f:=a^{-1}g$.
Then  $|f| <|g|$, $a \langle \mathcal T \rangle a^{-1}= \langle \mathcal T \rangle$ and
$$g \langle {\mathcal{S}} \rangle g^{-1} \cap \langle {\mathcal{T}} \rangle=
a\left( f \langle {\mathcal{S}} \rangle f^{-1} \cap \langle {\mathcal{T}} \rangle\right)
a^{-1}= (ah) \langle \mathcal P \rangle (ah)^{-1}
$$ for some $h \in \langle \mathcal T \rangle $ and some $\mathcal P \subseteq \mathcal T$
by the induction hypothesis.

If, on the other hand, there is $b \in \LL(g)  \cap \bigl( \mathcal{S} \cup
\star(\mathcal S) \bigr)^{\pm 1}$, set $f:=gb^{-1}$.
Then   $|f| <|g|$, $b \langle \mathcal S \rangle b^{-1}= \langle \mathcal S \rangle$ and
$g \langle {\mathcal{S}} \rangle g^{-1} \cap \langle {\mathcal{T}} \rangle=
f \langle {\mathcal{S}} \rangle f^{-1} \cap \langle {\mathcal{T}} \rangle$ and
we can apply the induction hypothesis once again.

Therefore, we can suppose that $\FL(g) \cap \mathcal{T}^{\pm 1}=\emptyset$ and
$\LL(g)  \cap \bigl( \mathcal{S} \cup \star(\mathcal S) \bigr)^{\pm 1}=\emptyset$.
We assert that in this case \begin{equation}
\label{eq:reduction}
 g \langle {\mathcal{S}} \rangle g^{-1} \cap \langle {\mathcal{T}} \rangle=
\bigcup_{s \in \mathcal{S}} \left(g\left\langle {\mathcal{S}\setminus \{s\}} \right\rangle g^{-1}
\cap \langle {\mathcal{T}} \rangle \right).\end{equation}
Indeed, if \eqref{eq:reduction} is false, then there exist $x \in \langle \mathcal S \rangle$ and
$y \in \langle \mathcal T \rangle$ such that $\supp(x)=\mathcal S$ and $gxg^{-1}=y$ in $G$.

Choose graphically reduced words $W,X$ and $Y$ representing in $G$ the elements
$g,x$ and $y$ respectively, so that $\supp(X)=\mathcal S$ and  $\supp(Y) \subseteq \mathcal T$.
Let $a$ be the first letter of $W$, then $a=v^{\pm 1}$ for some $v \in \mathcal V$.
According to our assumptions,
$v \in \supp(WXW^{-1}) \setminus\supp(Y)$ and $WXW^{-1}=Y$ in $G$.
Hence the left-hand side of the latter equality
cannot be graphically reduced.

Note that no letter $a$ of $W$ can be cancelled with a letter of $X$ in the word $WXW^{-1}$,
because this would mean
that $a \in \supp(X)^{\pm 1}=\mathcal S^{\pm 1}$ and
$a$ commutes with the suffix of $W$ after it, hence $a \in \LL(g) \cap \mathcal{S}^{\pm 1}=\emptyset$.
Similarly, no letter from $X$ can cancel with a letter from $W^{-1}$, therefore
a reduction
in $WXW^{-1}$ can occur only from the presence of a subword $c U c^{-1}$,
where $c$ is a letter from the initial copy of $W$ and $U$ contains $X$ as a subword. Thus,
$c=w^{\pm 1}$ for some $w \in \mathcal V$,
$\supp(W) \subseteq \supp(U) \subseteq \star(w)$.
Consequently, $\mathcal S \subseteq \star(w)$,
and $c$ is a last letter of $g$, because it commutes with the suffix of $W$ after it.
This implies that $c \in \LL(g) \cap \star(\mathcal S)^{\pm 1} \neq \emptyset$ contradicting
to our assumption.

Therefore  \eqref{eq:reduction} is true, implying that the group
$K:= g \langle {\mathcal{S}} \rangle g^{-1} \cap \langle {\mathcal{T}}\rangle \le G$
is covered by a finite union of its subgroups. A classical theorem of B. Neumann \cite{Neumann}
claims that in this case one of these subgroups must have finite index in $K$. Thus
there is $s_0 \in \mathcal S$ such that
$|K: \left(g\left\langle {\mathcal{S}\setminus \{s_0\}} \right\rangle g^{-1}
\cap \langle {\mathcal{T}} \rangle \right)|<\infty$. Using the induction hypothesis, we can find
$\mathcal P \subseteq \mathcal T$ and $h \in\langle \mathcal T \rangle$ such that
$g\left\langle {\mathcal{S}\setminus \{s_0\}} \right\rangle g^{-1}
\cap \langle {\mathcal{T}} \rangle = h \langle\mathcal P \rangle h^{-1}$. Therefore
$h \langle\mathcal P \rangle h^{-1} \le K$ and $|K:h \langle\mathcal P \rangle h^{-1}|< \infty$.
Hence Lemma \ref{lem:special_sbgp_not_f_i} can be applied to achieve the required equality
$K=h \langle\mathcal P \rangle h^{-1}$.
\end{proof}

Let us recall a few more facts about right angled Artin groups.

An element $g \in G$ is said to be \textit{A-cyclically reduced} if it cannot be written as
$g=a h a^{-1}$, where $a \in \mathcal{V}^{\pm 1}$ and $|h|=|g|-2$ (we have added
the letter ``A'' to avoid confusion with a similar notion for special HNN-extensions
introduced in Section \ref{sec:spec_HNN}). In the paper \cite{Servat}
H. Servatius proved that for every element $g$ of a right angled Artin group $G$ there exists
a unique A-cyclically reduced element $h$ such that $g=fhf^{-1}$ for some $f \in G$ with $|g|=|h|+2|f|$.
In particular, $\supp(h) \subseteq \supp(g)$.
Therefore, if $g \in G$ is not A-cyclically reduced then $|g^2| = |f h^2 f^{-1}| \le 2|h|+2|f|<2|g|$.
Thus we obtain the following

\begin{rem}\label{rem:non-cyc_red} If an element $g \in G$ is not A-cyclically reduced, then for
any graphically reduced word $W$ representing
$g$ in $G$, the word $W^2 \equiv WW$ cannot be graphically reduced.
\end{rem}

Another consequence of the above theorem of Servatius is that every given element of $G$ is
conjugate to a unique (up to a cyclic permutation) A-cyclically reduced element.
In particular, we can make

\begin{rem}\label{rem:A-cyc_red_uniqueness} If the elements $g, h \in G$ are A-cyclically reduced
and conjugate in $G$, then $\supp(g)=\supp(h)$.
\end{rem}

A special subgroup $A$ of the right angled Artin group $G$ is said to be \textit{maximal}
if $A=\langle \mathcal S \rangle$ for some maximal proper subset  $\mathcal S$ of $\mathcal V$
(i.e., if $|\mathcal S|= |\mathcal V|-1$).

\begin{lemma}\label{lem:el_in_RAAG_non_conj_into_max_spec} For any non-trivial
element $g \in G$ there is a maximal special subgroup $A$ in $G$ such that $g \notin A^G$.
\end{lemma}

\begin{proof} Arguing by contradiction, suppose that there are $f_1,\dots,f_n \in G$
such that $g \in \bigcap_{i=1}^n f_iA_if_i^{-1}$ in $G$, where $A_1,\dots,A_n$ is the list
of all maximal special subgroups of $G$. 
Then for each $i \in \{1,\dots,n\}$, there is an A-cyclically reduced element $a_i \in A_i \setminus \{1\}$
such that $g$ is conjugate to $a_i$ in $G$. Choose any letter $v \in \supp(a_1)$ and take
$j \in \{1,\dots,n\}$ such that $A_j=\langle \mathcal V \setminus \{v\}\rangle$. Then
$a_1$ must be conjugate to $a_j$ in $G$, which is impossible by Remark \ref{rem:A-cyc_red_uniqueness},
because $v \in \supp(a_1)\setminus \supp(a_j)$ (as $\supp(a_j) \subseteq \mathcal V \setminus \{v\}$).
This contradiction proves the lemma.
\end{proof}

Recall that an automorphism $\phi$ of a group $G$ is called \textit{pointwise inner} if
for each $g \in G$ there exists $f=f(g) \in G$, such that
$\phi(G)=fgf^{-1}$ in $G$. Let $Aut_{pi}(G)$ denote
the set of all pointwise inner automorphisms of $G$. It is easy to see that
$Aut_{pi}(G)$ is a normal subgroup of the full automorphism group $Aut(G)$, containing the
subgroup of all inner automorphisms $Inn(G)$.

We are now going to prove that the the group of pointwise inner automorphisms of a right angled Artin
group $G$ coincides with the group of inner automorphisms of $G$.

\begin{prop}\label{prop:pi=inn} For any right angled Artin group $G$ we have $Aut_{pi}(G)=Inn(G)$.
\end{prop}

\begin{proof} Suppose there exists $\phi \in Aut_{pi}(G) \setminus Inn(G)$.
Since $\phi$
maps every generator of $G$ to its conjugate, we can replace $\phi$ by its composition with an inner
automorphism of $G$ to assume that $\phi(v)=v$ for some $v \in \mathcal V$.

Then there is an automorphism $\mu$ lying in the right coset
$Inn(G) \phi \subset Aut_{pi}(G)$ such that the set ${\rm Fix}(\mu)$ is maximal,
where ${\rm Fix}(\mu)$ denotes the subset of all elements in $\mathcal V$ fixed by $\mu$.
Note that ${\rm Fix}(\mu) \subsetneqq \mathcal V$ because $\mu \neq id_G$ (the identity map $id_G$
of $G$ does not belong to the coset $Inn(G) \phi$ by our assumption). And
${\rm Fix}(\mu) \neq \emptyset$ since $v \in {\rm Fix}(\phi) \neq \emptyset$.

Let ${\rm Fix}(\mu)=\{v_1,\dots,v_k\} \subset \mathcal V$ and pick any $w \in \mathcal{V}\setminus
{\rm Fix}(\mu)$. By the assumptions, there is $f \in G$ such that $\mu(w)=fwf^{-1}$.
Choose a shortest element $f$ with this property.


 First, suppose that there is $a \in \FL(f) \cap
\left(\bigcap_{i=1}^k \star(v_i)\right)^{\pm 1}$.
 Then $a \in \bigcap_{i=1}^k C_G(v_i)$, hence after
defining a new automorphism $\lambda \in Inn(G) \phi$  by
$\lambda(g):=a^{-1} \mu(g) a$ for all $g \in G$, we will have ${\rm Fix}(\mu) \subseteq {\rm Fix}(\lambda)$,
$\lambda(w)=(a^{-1}f) w (a^{-1} f)^{-1}$ and $|a^{-1} f|<|f|$.  Note that $a^{-1} f \neq 1$
because, otherwise, we would have $w \in {\rm Fix}(\lambda)$ contradicting to the maximality of ${\rm Fix}(\mu)$.
Thus we can replace $\mu$ with $\lambda$, making
$f$ shorter. We can continue doing the same for $\lambda$, and so on. Eventually we will end up
with an automorphism $\mu \in Inn(G) \phi$ (we keep the same notation for it,
although the actual automorphism
may be different) such that $\mu(w)=f w f^{-1}$, $f$ is a shortest element
with this property,  $f \neq 1$, 
and 
\begin{equation}
\label{eq:FL(u)-empty_inter}
\FL(f) \cap \star\left({\rm Fix}(\mu)\right)^{\pm 1}= \FL(f) \cap
\left(\bigcap_{i=1}^k \star(v_i)\right)^{\pm 1}=\emptyset .\end{equation}

Denote $\mathcal S := {\rm Fix}(\mu) \cap \FL(f)^{\pm 1} \subset \mathcal V$.
Using \eqref{eq:FL(u)-empty_inter}, we see that for every $s \in \mathcal S$ there is
$v(s) \in {\rm Fix}(\mu)$ such that $v(s) \notin \star(s)$.
Note that in this case
$v(s) \notin \FL(f)^{\pm 1}$
because $s \in \FL(f)^{\pm 1}$ and any two elements of $\FL(f)^{\pm 1}$ commute.
Set $\mathcal T := \{v(s)~|~s \in \mathcal S\} \subseteq {\rm Fix}(\mu) \setminus \FL(f)^{\pm 1}$,
and write ${\rm Fix}(\mu)=\{t_1,\dots,t_l\} \sqcup \{s_1,\dots,s_m\}$, where
$\mathcal T=\{t_1,\dots,t_l\}$ (consequently, $\mathcal S \subseteq \{s_1,\dots,s_m\}$).
Finally, define the words $T:=t_1 \dots t_l$, $S:=s_1 \dots s_m$ and let $g$ be the element
represented by the word by $TST w$ in $G$.

Since $\mu \in Aut_{pi}(G)$, there exists $x \in G$ such that $\mu(g)=xgx^{-1}$.
On the other hand, $\mu(g)=TST U w U^{-1}$, for some graphically reduced non-empty
word $U$ representing $f$ in $G$.
Note that the word $UwU^{-1}$ is graphically reduced (otherwise, we could make $U$, and hence $f$,
shorter) and $w \notin \supp(TST)={\rm Fix}(\mu)$.
Therefore the only possible reduction which could occur in the word
$W \equiv TST UwU^{-1}$ would
arise from cancellation of a letter from $TST$ with  a letter from $U$ or $U^{-1}$.
However, no letter $t$ from the first copy of $T$ could cancel with a letter from $U$ or $U^{-1}$
in $W$, because $\supp(S) \not\subset \star(t)$ (as $t=v(s)$ for some $s \in \mathcal S$
and $s \in \supp(S)\setminus \star(v(s))$). On the other hand, if some letter $t$ from the
second copy of $T$ in $W$ cancelled with some letter $a$ from $U$, then $a \in \FL(f)$,
but this would contradict to $t \in \mathcal T$ and $\mathcal T \cap \FL(f)^{\pm 1}=\emptyset$.
If this letter $t$ cancelled with a letter $a$ from $U^{-1}$ in $W$, then we would have
$\supp(U) \subset \star(t)$, which is impossible as $t=v(s)$ for some $s \in \supp(U)$ and
$s \notin \star(v(s))$. Therefore, if $W$ is not graphically reduced, then a letter $s$ from $S$
must cancel with a letter $a=s^{-1}$ from $U$ or $U^{-1}$, in particular,
$\mathcal T=\supp(T) \subseteq \star(s)$, implying that $s \notin \mathcal S$.
However, if $s$ cancels with a letter from $U$, we see that $a \in \FL(f)$,
hence $s \in \mathcal S$, which is false. And if $s$ cancelled with a letter $a$
from $U^{-1}$, then we would get a contradiction with the fact that $UwU^{-1}$ is graphically
reduced, because $a^{-1}=s$ is a letter of $U$ lying between $s$ and $a$ in $W$.

Therefore $W \equiv TSTUwU^{-1}$ is a graphically reduced word representing $\mu(g)$ in $G$.
Consequently, $|\mu(g)|=2\|T\|+\|S\|+2\|U\|+1 > 2\|T\|+\|S\|+1=|g|$ since $\|U\|>0$ ($\|U\|$
denotes the length of the word $U$). But $\mu(g)=xgx^{-1}$, hence $\mu(g)$ is not A-cyclically reduced.
By Remark \ref{rem:non-cyc_red}, a reduction can be made in the word
$W^2 \equiv TSTUwU^{-1}TSTUwU^{-1}$. But an argument similar to the above shows that this is impossible.

Thus we have arrived to a contradiction, which proves that $Aut_{pi}(G)=Inn(G)$, as needed.
\end{proof}

\begin{rem}\label{rem:end_RAAG-inner} The reader could have noticed that in the proof
of Proposition \ref{prop:pi=inn} we have actually shown more than it claims. In fact, we
have proved that any endomorphism $\phi$ of a right angled Artin group $G$, which
maps each conjugacy class of $G$ into itself, is an inner automorphism of $G$.
\end{rem}


\section{Special HNN-extensions}\label{sec:spec_HNN}
The purpose of this section is to develop necessary tools for dealing with special
HNN-extensions.

Let $A$ be a group and let $H \le A$ be a subgroup.

\begin{df}\label{df:spec_HNN}
The \textit{special HNN-extension of $A$ with respect to $H$} is the group $G$ given by
the presentation \begin{equation} \label{eq:spec_HNN-def}
G=\langle A,t\,\|\,tht^{-1}=h~\mbox{ for every } h \in H \rangle.\end{equation}
\end{df}

In other words, the special HNN-extension $G$ is a particularly simple HNN-extension of $A$, where both of
the associated subgroups are equal to $H$ and the isomorphism between these subgroups is the identity map
on $H$.

Let $\Gamma$ be a
finite graph with the set of vertices  $\mathcal V$ of cardinality $n \in N$.
The reason why we are interested in special HNN-extensions is the the observation below.

\begin{rem}\label{rem:constr_of_RAAG}
Let $G$ be the right angled Artin group associated to $\Gamma$. Then $G$ can be obtained from the trivial
group via a sequence of special HNN-extensions. More precisely, there are right angled Artin groups
$\{1\}=G_0$, $G_1$, $\dots$, $G_n=G$ such that
$G_{i+1}$ is a special HNN-extension of $G_i$ with respect to some
special subgroup $H_i \le G_i$ for every $i=0,\dots,n-1$.
\end{rem}

The groups $G_i$ can be constructed as follows.
Let $\mathcal V=\{v_1,\dots,v_n\}$ and denote
$\mathcal S_{i+1}:=\{v_1,\dots,v_i\} \cap \star(v_{i+1}) \subset \mathcal V$ for $i=1,\dots,n-1$.
Set $G_0:=\{1\}$, $G_1:=\langle v_1 \rangle$ (the infinite cyclic group generated by $v_1$),
and
$$G_{i+1}:=\langle G_i,v_{i+1} \, \|\, v_{i+1} s v_{i+1}^{-1}=s ~\mbox{ for every } s \in \mathcal{S}_{i+1}\rangle, ~i=1,\dots,n-1.$$ Clearly, $G=G_n$ and for each $i$, $G_{i+1}$
is a special HNN-extension of $G_i$ with respect to
the special subgroup $\langle \mathcal S_{i+1}\rangle$ of $G_i$, and
$G_i$ is a special subgroup of $G$ generated by $\{v_1,\dots,v_i\}$.

\begin{rem}\label{rem:RAAG_spec_HNN_over_max_sbgp} If $G$ is a right angled Artin group
associated to $\Gamma$, then for any maximal special subgroup $A \le G$, $G$
splits as a special $HNN$-extension
\eqref{eq:spec_HNN-def} of $A$ with respect to a certain special subgroup $H$ of $A$.
\end{rem}

Indeed,  if $A=\langle \mathcal S \rangle$, where $\mathcal S \subset \mathcal V$
and $\mathcal V=\mathcal S \sqcup \{t\}$, set $\mathcal U:=\star(t) \setminus \{t\} \subseteq \mathcal S$.
Then $G=\langle A,t\,\|\,tut^{-1}=u,~\forall\, u \in \mathcal U \rangle$ is
a special HNN-extension of $A$ with respect to the subgroup
$H:=\langle \mathcal U  \rangle \le A$.

Special HNN-extensions are usually much easier to deal with than general HNN-extensions.
Throughout this section we will limit ourselves to considering only the former ones,
even though most of the statements can be re-formulated in the general situation.

Let $G$ be the special HNN-extension given by \eqref{eq:spec_HNN-def}. von Dyck's Theorem
yields the following \textit{Universal Property} of special HNN-extensions, which will
be important for us:

\begin{rem}\label{rem:univ_spec_HNN} For any group $B$, every homomorphism
$\psi:A \to B$ can be naturally
extended to a homomorphism $\tilde\psi: G \to P$, where
$P:=\langle B, s \,\|\, s x s^{-1}=x, ~\forall\, x\in \psi(H)\rangle$
is the special HNN-extension of $B$ with respect to $\psi(H)$, so that
$\tilde \psi|_A=\psi$ and $\tilde\psi(t) = s$.
\end{rem}

\begin{lemma} \label{lem:ker_of_tilde_psi} In the notations of Remark \ref{rem:univ_spec_HNN},
$\ker(\tilde\psi)=N$, where $N\lhd G$ is the normal closure of $\ker(\psi) \le A\le G$ in $G$.
\end{lemma}

\begin{proof} Obviously, $N \le \ker(\tilde\psi)$, and hence $N \cap A= \ker(\tilde\psi) \cap A=\ker(\psi)$.
Let $\phi:G \to Q:=G/N$ be the natural epimorphism with $\ker(\phi)=N$. Consequently, if we define
$\theta: Q \to P$ to be the natural
epimorphism with the kernel $\phi(\ker(\tilde\psi))$, then we will have
$\tilde\psi=\theta \circ \phi$.

Observe that
$\phi(t)\phi(x)\phi(t)^{-1}=\phi(t x t^{-1})=\phi(x)$ in $Q$ for every $x \in H$,
and the map $\xi: \psi(A) \to \phi(A)$ defined by
$\xi(\psi(a)):=\phi(a)$ for all $a \in A$, is an isomorphism, since $\ker(\psi)=\ker(\phi)\cap A$.
Therefore,
by von Dyck's Theorem, there is a homomorphism $\tilde\xi: P \to Q$ such that
$\tilde\xi (\psi(a)) = \phi(a)$ for every $a \in A$ and $\tilde\xi(s)=\phi(t)$. It is easy to see that
$\tilde\xi \circ \theta: Q \to Q$ is the identity map on $Q$. Hence, $\theta$ is injective,
that is $\{1\}=\ker(\theta)=\phi(\ker(\tilde\psi))$ in $Q$, implying that
$\ker(\tilde\psi) \le \ker(\phi)=N$. Thus $\ker(\tilde\psi)=N$.
\end{proof}

\begin{lemma}\label{lem:spec_HNN_retractions} Suppose that $\rho \in End(A)$ is a retraction of
$A$ onto its subgroup $B$. Then there is a retraction $\tilde \rho_1 \in End(G)$
of $G$ onto $B \le G$ such that
$\tilde {\rho_1}|_A=\rho$.
\end{lemma}

\begin{proof} Define the map $\rho_1: A\sqcup\{t\} \to A \sqcup\{t\}$
by $\rho_1(a):=\rho(a)$ for each $a\in A$ and $\rho_1(t):=1$.
Then $\rho_1$ can be extended to an endomorphism $\tilde \rho_1 :G \to G$ by von Dyck's Theorem.
Obviously, $\tilde \rho_1$ is a retraction of $G$ onto $B$.
\end{proof}

Every element of the special HNN-extension $G$, given by \eqref{eq:spec_HNN-def}, is a product of the form
\begin{equation} \label{eq:word} x_0 t^{\e_1}x_1 t^{\e_2} \dots t^{\e_n}x_{n} \end{equation}
for some $n \in \N \cup \{0\}$, $x_i \in A$, $i=0,\dots,n$, and $\e_j \in \Z\setminus\{0\}$,
$j=1,\dots,n$.

The product \eqref{eq:word} is said to be \textit{reduced}, if $x_i \notin H$ for every
$i\in \{1,2,\dots,n-1\}$.
Since $t$ commutes with every element of $H$, it is easy to see that any $g \in G$
is equal to some reduced product in $G$. By Britton's Lemma (see \cite[IV.2]{L-S}) a non-empty
reduced product represents a non-trivial element in $G$. It follows, that if
two reduced products $x_0 t^{\e_1}x_1 t^{\e_2} \dots t^{\e_n}x_{n}$ and
$y_0 t^{\z_1}y_1 t^{\z_2} \dots t^{\z_m}y_{m}$ are equal in $G$, then $m=n$ and
$\e_i=\z_i$ for every $i=1,\dots,n$ (see \cite[IV.2.3]{L-S}).

Suppose that an element $g \in G$ is equal to a product $t^{\e_1}x_1 t^{\e_2} \dots t^{\e_n}x_{n}$.
Let us fix this presentation for $g$. Any product of the form
$t^{\e_k}x_k t^{\e_{k+1}} \dots t^{\e_n} x_{n} t^{\e_1} x_1 \dots
t^{\e_{k-1}} x_{k-1} \in G$, for some $k \in \{1,\dots,n\}$, is said to be a \textit{cyclic permutation}
of $g$. The product $g=t^{\e_1}x_1 t^{\e_2} \dots t^{\e_n}x_{n}$ is called \textit{cyclically reduced}
if each of its cyclic permutations is reduced.
A \textit{prefix} of $g$ is an element of the form
$t^{\e_1}x_1 t^{\e_2} \dots t^{\e_k}x_{k}$ for some $k \in \{0,1,\dots,n\}$ (if $k=0$ then
we have the empty prefix, corresponding to the trivial element of $G$). Similarly, a \textit{suffix}
of $g$ is an element of the form $t^{\e_l}x_l  \dots t^{\e_n}x_{n}$ for some
$l \in \{1,2,\dots,n+1\}$.


It is not difficult to see that every element $f \in G$ either belongs to $A^G$ in $G$
or is conjugate to some non-trivial cyclically reduced
element in $G$.

Below is the statement of Collins's Lemma (see \cite[IV.2.5]{L-S}) in the case of special HNN-extensions.

\begin{lemma}\label{lem:Collins} Suppose that $g=t^{\e_1}x_1 t^{\e_2} \dots t^{\e_n}x_{n}$ and
$f=t^{\z_1}y_1 t^{\z_2} \dots t^{\z_m}y_{m}$ are cyclically reduced in $G$, with $n\ge 1$.
Then $g \notin A^G$. And if
$f$ is conjugate to $g$ in $G$ then $m=n$ and there exist $h \in H$ and a cyclic permutation $f'$ of
$f$ such that $f'=hgh^{-1}$ in $G$.
\end{lemma}

We will also use the following description of centralizers in special HNN-extensions:

\begin{prop}\label{prop:descr_centr} Let $G$ be the special HNN-extension given by
\eqref{eq:spec_HNN-def}. Suppose that the product $g=t^{\e_1}x_1 t^{\e_2} \dots t^{\e_n}x_{n} \in G$ is
cyclically reduced and $n\ge 1$.

If $x_n \in H$ then $n=1$ and $C_G(g)=\langle t \rangle \times C_H(x_1)=\langle t \rangle \times C_H(g) \le G$.

If $x_n \in A\setminus H$, let $\{p_1,\dots,p_k\}$, $1 \le k \le n+1$,
be the set all of prefixes of $g$ satisfying $p_{i}^{-1}g p_{i} \in g^H$ in $G$.
For each $i=1,\dots,k$, choose $h_i \in H$ such that $h_ip_i^{-1} g p_ih_i^{-1} =g$, and define
the finite subset $\Omega \subset G$
by $\Omega:=\{h_ip_{i}^{-1}\,|\,i=1,\dots,k\}$. Then $C_G(g)=C_H(g) \langle g \rangle \Omega$.
\end{prop}

In order to prove Proposition \ref{prop:descr_centr} we will need two lemmas below.
The proof of the next statement is similar to the proof of Collins's Lemma.

\begin{lemma}\label{lem:conj_in_spec_HNN} Let $g=t^{\e_1}x_1 t^{\e_2} \dots t^{\e_n}x_{n}$ and
$f=t^{\z_1}y_1 t^{\z_2} \dots t^{\z_n}y_{n}$ be cyclically reduced elements of $G$,
with $n\ge 1$ and $x_n \notin H$.
Assume that $c g c^{-1}=f$ in $G$, where the element $c$ is equal to the reduced product
$z_0t^{\eta_1}z_1 t^{\eta_2} \dots t^{\eta_m}z_{m}$. Then there are the following three
mutually exclusive possibilities.
\begin{itemize}
	\item[a)] $m=0$ and $c \in H$;
	\item[b)] $m \ge 1$, $z_m \in H$ 
						and there is a prefix $p$ of $g$
						such that $c=hp^{-1}g^l$ in $G$, for some $h \in H$ and $l \in \Z$, $l \le 0$;
	\item[c)] $m \ge 1$, $x_nz_m^{-1} \in H$ 
						and there is a suffix $s$ of $g$
						such that $c=hsg^l$ in $G$, for some $h \in H$ and $l \in \Z$, $l \ge 0$.
\end{itemize}
\end{lemma}

\begin{proof} First, suppose that $m=0$, i.e., $c=z_0 \in A$. Then $f^{-1}cgc^{-1}=1$ in $G$. Yielding
$$y_n^{-1} t^{-\z_n} \dots y_{1}^{-1}t^{-\z_1} z_0 t^{\e_1}x_1 t^{\e_2} \dots t^{\e_n}x_{n} z_0^{-1}=1.$$
Therefore the left-hand side is not reduced, hence $c=z_0 \in H$.

Assume now that $m \ge 1$.
The equality $c g c^{-1}=f$ in $G$ gives rise to the equation
$$z_0t^{\eta_1}  \dots 
t^{\eta_m}z_{m} t^{\e_1}x_1 t^{\e_2} \dots t^{\e_n}x_{n}
z_m^{-1}t^{-\eta_m} 
\dots t^{-\eta_1}z_{0}^{-1}=
t^{\z_1}y_1  \dots t^{\z_n}y_{n}. $$
The left-hand side cannot be reduced because it contains more $t$-letters than the right-hand side.
Hence either $z_m \in H$ or $x_nz_m^{-1} \in H$
(note that both of these inclusions cannot happen
simultaneously since $x_n \notin H$).
Let us consider the case when $z_m \in H$, as the second case is similar. Then
$t^{\eta_m}z_{m} t^{\e_1}=z_m t^{\eta_m+\e_1}$ and, thus, we get
\begin{equation}
\label{eq:ind_step}
z_0t^{\eta_1}  \dots t^{\eta_{m-1}}(z_{m-1} z_m )t^{\eta_m+\e_1}x_1 t^{\e_2} \dots t^{\e_n}x_{n}
z_m^{-1}t^{-\eta_m} z_{m-1}^{-1} \dots t^{-\eta_1}z_{0}^{-1}=
t^{\z_1}y_1 \dots t^{\z_n}y_{n}.\end{equation}
Once again, we see that the left-hand side of the above equation cannot be reduced.
In the case when $m=1$, this implies that $\eta_m+\e_1=0$. On the other hand, if $m>1$, then
$z_{m-1} \in A\setminus H$, hence $z_{m-1}z_m \notin H$, and again, in order for a
reduction to be possible we must have $\eta_m+\e_1=0$. Hence $\e_1=-\eta_m$ and
$z_m^{-1}t^{-\eta_m}z_{m-1}^{-1}=t^{\e_1} x_1 h_1^{-1}$, where $h_1:=z_{m-1}z_{m}x_1 \in A$.
And \eqref{eq:ind_step} becomes
\begin{equation} \label{eq:h_1_in_ind}
(z_0t^{\eta_1}  \dots z_{m-2}t^{\eta_{m-1}}h_1) (t^{\e_2} x_2 \dots t^{\e_n}x_{n}t^{\e_1}x_1)
(z_0t^{\eta_1}  \dots t^{\eta_{m-1}}h_1)^{-1}= \\
t^{\z_1}y_1 \dots t^{\z_n}y_{n}.\end{equation}

Now, if $m=1$, i.e., $c=z_0t^{-\e_1}z_1=z_0z_1t^{-\e_1}$, then we have
$$1=f^{-1}cgc^{-1}=(y_n^{-1} t^{-\z_n} \dots y_{1}^{-1}t^{-\z_1})h_1 (t^{\e_2}x_2 \dots t^{\e_n}x_{n}t^{\e_1}x_1)h_1^{-1}.$$ Hence $h_1 \in H$ and $c=h_1 (t^{\e_1}x_1)^{-1}$,
where $t^{\e_1}x_1$ is a prefix of $g$.

Otherwise, if $m=M \ge 2$ we will use induction on $m$ to prove the claim b) of the lemma. Thus, we will
assume that b) has been established for all elements $c$ with $1\le m \le M-1$.

Note that $x_1 h_1^{-1}=z_m^{-1}z_{m-1}^{-1} \notin H$ since $z_m \in H$
and $z_{m-1} \in A\setminus H$ (because $m\ge 2$ and the product $z_0t^{\eta_1}\dots
t^{\eta_{m-1}} z_{m-1}t^{\eta_m}z_{m}$ was assumed to be reduced).

Let us look at the equation \eqref{eq:h_1_in_ind}. Since $m-1\ge 1$,
the left-hand side cannot be reduced.
And a reduction in it can only occur if $h_1 \in H$ because $x_1 h_1^{-1}\notin H$. Hence
we are in the case b) of the lemma, and
can apply the induction hypothesis to \eqref{eq:h_1_in_ind}. Thus there is a prefix $p$
of the element $g_1:=t^{\e_2}x_2 \dots t^{\e_n}x_{n}t^{\e_1}x_1$, $h \in H$ and $l \in \Z$, $l\le 0$,
such that $z_0t^{\eta_1}  \dots z_{m-2}t^{\eta_{m-1}}h_1=hp^{-1}g_1^l$. Consequently,
$$c=z_0t^{\eta_1}  \dots t^{\eta_{m-1}}z_{m-1} t^{\eta_m}z_m=
z_0t^{\eta_1}  \dots t^{\eta_{m-1}}h_1 x_1^{-1}t^{-\e_1}=hp^{-1}g_1^l x_1^{-1}t^{-\e_1}=hq^{-1}g^{l},$$
where $q=t^{\e_1} x_1 p$. It is easy to see that either $q$ is a prefix of $g$, or
$q=g t^{\e_1} x_1$. In the latter case, $c=h(t^{\e_1} x_1)^{-1}g^{l-1}$ and $t^{\e_1} x_1$
is a prefix of $g$. Thus the step of induction is established, and the proof of the lemma is finished.
\end{proof}

The next lemma treats the case which was not covered by Lemma \ref{lem:conj_in_spec_HNN}.

\begin{lemma}\label{lem:centr-partic_case} Suppose that $g=t^\e x\in G$, where $\e \in \Z\setminus\{0\}$
and $x \in H$. Then $C_G(g)=\langle t \rangle \times C_H(x)=\langle t \rangle \times C_H(g)$.
In particular, $C_G(t^\e)= \langle t \rangle \times H \le G$.
\end{lemma}

\begin{proof} For any $c \in C_G(g)$ we have $c g c^{-1}=g$. Let $z_0t^{\eta_1}z_1 t^{\eta_2} \dots t^{\eta_m}z_{m}$  be a reduced product  representing $c$ in $G$. Then we have
\begin{equation}
\label{eq:c-equality}
z_0t^{\eta_1}z_1 t^{\eta_2} \dots t^{\eta_m}z_{m}  t^\e x z_{m}^{-1}t^{-\eta_m}\dots t^{-\eta_2}z_1^{-1}t^{-\eta_1}z_0^{-1} = t^\e x.\end{equation}
Arguing as in the proof of Lemma  \ref{lem:conj_in_spec_HNN}, we see that if $m=0$ then $c=z_0 \in H$ and
$t^\e x=z_0 t^\e x z_0^{-1}=t^\e z_0xz_0^{-1}$, thus $1=z_0xz_0^{-1}$, i.e., $c=z_0\in C_H(x)$.

So, assume now that $m \ge 1$. Then the equation \eqref{eq:c-equality} implies that either
$z_m \in H$ or $xz_m^{-1} \in H$. But either of these inclusions leads to $z_m \in H$ because
$x \in H$ by the assumptions. Therefore  $t^{\eta_m}z_{m}  t^\e x z_{m}^{-1}t^{-\eta_m}=
z_{m} x z_{m}^{-1} t^\e$ in $G$ and \eqref{eq:c-equality} becomes
$$z_0t^{\eta_1} \dots z_{m-2}t^{\eta_{m-1}} (z_{m-1} z_{m}   x  z_{m}^{-1})
t^\e z_{m-1}^{-1}t^{\eta_{m-1}} z_{m-2}^{-1} \dots t^{-\eta_1}z_0^{-1} = t^\e x.$$
If $m \ge 2$, then $z_{m-1} \in A\setminus H$, hence $z_{m-1} z_{m}   x  z_{m}^{-1} \notin H$,
and the above equation contradicts to Britton's Lemma: the left-hand side is reduced, but
contains more $t$-letters than the right-hand side.

Therefore $m=1$, i.e., $c=z_0t^{\eta_1}z_1$
and $z_1 \in H$. Consequently, $$1=g^{-1}cgc^{-1}=x^{-1}t^{-\e} z_0t^{\eta_1}z_1 t^\e x z_1^{-1}
t^{-\eta_1} z_0^{-1}=x^{-1}t^{-\e} (z_0z_1  x z_1^{-1}) t^\e z_0^{-1}.$$
Applying Britton's Lemma again, we achieve $z_0z_1  x z_1^{-1} \in H$, implying that $z_0 \in H$ and
$1=x^{-1}z_0z_1  x z_1^{-1}z_0^{-1}$ in $G$. That is, $z_0z_1 \in C_H(x)$, and
$c=z_0t^{\eta_1}z_1=t^{\eta_1}z_0z_1 \in \langle t \rangle C_H(x)$.

Thus we proved that $C_G(g) \subseteq \langle t \rangle C_H(x)$.
Finally, since $t$ commutes with every element from $H$, it is clear that $t \in C_G(g)$ and
$C_H(x) \subset C_H(g) \subset C_G(g)$. Hence
$C_G(g) = \langle t \rangle C_H(x)=\langle t \rangle \times C_H(x) \le G$. The equality
$C_H(g)=C_H(x)$ is immediate.
\end{proof}

\begin{proof}[Proof of Proposition \ref{prop:descr_centr}] If $x_n \in H$, $g$ can be
cyclically reduced only when $n=1$, and the claim follows from Lemma \ref{lem:centr-partic_case}.

So, we can assume that $x_n \in A\setminus H$. Therefore we are able to apply
Lemma \ref{lem:conj_in_spec_HNN} to $g$ and $f:=g$, which tells us that for any $c \in C_G(g)$
there exist $h \in H$ and $l \in \Z$ such that either there is a prefix $p$ of $g$ with
$c=h p^{-1} g^l$,  or there is a suffix $s$ of $g$ with $c=h s g^l$.
Note that in the latter case there is a prefix $p$ of $g$
such that $s=p^{-1} g$. Hence we can assume that $c=h p^{-1} g^l$ for some prefix $p$ of $g$, $h \in H$
and $l \in \Z$.

But then $g=cgc^{-1}=h p^{-1}g p h^{-1}$, hence $p=p_i$ for some $i \in \{1,\dots,k\}$.
The equalities $h p_i^{-1}g p_i h^{-1}=g=h_i p_i^{-1}g p_i h_i^{-1}$ yield
$h^{-1}g h=p_i^{-1}g p_i=h_i^{-1} g h_i$. Consequently $hh_i^{-1} \in C_H(g)$ and $h \in C_H(g)h_i$.
Thus $c=hp_{i}^{-1}g^l \in C_H(g)h_i p_{i}^{-1}g^l \subseteq C_H(g) \Omega \langle g \rangle$.

We have shown that $C_G(g) \subseteq C_H(g) \Omega \langle g \rangle$. Observe that
$\Omega \subset C_G(g)$ by definition, hence  $C_G(g)= C_H(g) \Omega \langle g \rangle=
C_H(g)  \langle g \rangle \Omega$.
\end{proof}



Now we formulate a criterion for conjugacy in special HNN-extensions.

\begin{lemma}\label{lem:conj_crit} Let $G$ be the special HNN-extension \eqref{eq:spec_HNN-def}.
Suppose that $B \le A$ is a subgroup and  $g,f \in G$ are elements represented by reduced products
$x_0t^{\e_1}x_1 t^{\e_2} \dots t^{\e_n}x_{n}$ and $y_0t^{\z_1}y_1 t^{\z_2} \dots t^{\z_m}y_{m}$
respectively, with $n \ge 1$. Then $f \in g^B$ in $G$ if and only if all of the following conditions hold:
\begin{itemize}
  \item[(i)] $m=n$ and $\e_i=\z_i$ for all $i=1,\dots,n$;
	\item[(ii)] $y_0y_1 \dots y_n \in (x_0x_1 \dots x_n)^B$ in $A$;
	\item[(iii)] for every $b_0 \in B$ with $y_0y_1 \dots y_n = b_0 (x_0x_1 \dots x_n)b_0^{-1}$ in $A$,
				the intersection $I \subseteq A$ is non-empty, where
				$$I:=b_0 C_B(x_0 \dots x_n) \cap y_0 H x_0^{-1} \cap (y_0y_1)H(x_0x_1)^{-1} \cap \dots \cap
				(y_0\dots y_{n-1})H(x_0 \dots x_{n-1})^{-1}.$$
\end{itemize}
\end{lemma}

\begin{proof} First we establish the sufficiency. Assume that the conditions (i),(ii) and (iii) hold.
Take any $b_0 \in B$ satisfying (iii) (it exists by (ii)) and let $I$ be the corresponding intersection.
Then there exists an element $b \in I$. The  inclusion
$b \in b_0 C_B(x_0 \dots x_n)$ implies that $y_0y_1 \dots y_n = b (x_0x_1 \dots x_n)b^{-1}$.
We will show that $f^{-1}b g b^{-1}=1$, which is equivalent (in view of (i)) to
\begin{equation} \label{eq:resolution}
y_n^{-1}t^{-\e_n}  \dots t^{-\e_2}y_1^{-1} t^{-\e_1}y_{0}^{-1} b x_0t^{\e_1}x_1 t^{\e_2}
\dots t^{\e_n}x_{n} b^{-1}=1.\end{equation}
Note that since $b \in I$, $y_{0}^{-1} b x_0 \in H$, hence $y_1^{-1} t^{-\e_1}y_{0}^{-1} b x_0t^{\e_1}x_1=
y_1^{-1} y_{0}^{-1} b x_0x_1 \in H$. Therefore we can continue reducing the left-hand side
of \eqref{eq:resolution}, until it becomes $y_n^{-1} \dots y_1^{-1}y_0^{-1} b x_0x_1 \dots x_n b^{-1}$,
which is equal to $1$ in $G$. Thus the sufficiency is proved.

To obtain the necessity, assume that $f=bgb^{-1}$ for some $b \in B$, that is,
$$b x_0t^{\e_1}x_1 t^{\e_2} \dots t^{\e_n}x_{n} b^{-1}=y_0t^{\z_1}y_1 t^{\z_2} \dots t^{\z_m}y_{m}.$$
Since both of the sides of the above equality are reduced, applying Britton's Lemma we obtain (i).
Therefore the equation \eqref{eq:resolution} holds in $G$.

Note that there is a canonical retraction $\rho_A \in End(G)$ of $G$ onto $A$, such that
$\rho_A(t)=1$ (apply Lemma \ref{lem:spec_HNN_retractions} to the identity map on $A$). Hence,
$y_0\dots y_n=\rho_A(f)=\rho_A(bgb^{-1})=b(x_0\dots x_n)b^{-1}$, yielding (ii).

To achieve (iii), take any $b_0 \in B$ satisfying $y_0y_1 \dots y_n = b_0 (x_0x_1 \dots x_n)b_0^{-1}$.
Then $b_0^{-1}b \in C_B(x_0 \dots x_n)$, i.e., $b \in b_0C_B(x_0 \dots x_n)$. By Britton's Lemma
the left-hand side of  \eqref{eq:resolution} cannot be reduced, therefore $y_0^{-1}bx_0 \in H$ and
$b \in y_0H x_0^{-1}$. Consequently, if $n \ge 2$,
$y_1^{-1}t^{-\e_1}y_{0}^{-1} b x_0t^{\e_1}x_1=y_1^{-1}y_{0}^{-1} b x_0x_1$ and  \eqref{eq:resolution}
becomes
$$y_n^{-1}t^{-\e_n}  \dots t^{-\e_2}(y_1^{-1} y_{0}^{-1} b x_0x_1 )t^{\e_2}
\dots t^{\e_n}x_{n} b^{-1}=1.$$
Applying Britton's Lemma to the above equation, we see again that $y_1^{-1} y_{0}^{-1} b x_0x_1 \in H$,
hence $b \in (y_0y_1)H(x_0x_1)^{-1}$. Clearly, we can continue this process, showing
that $b \in (y_0\dots y_{i})H(x_0\dots x_{i})^{-1}$ for every $i \in \{0,\dots,n-1\}$. Thus,
$b \in I\neq \emptyset$ and the condition (iii) is satisfied.
\end{proof}

Next comes a similar statement about centralizers.
\begin{lemma}\label{lem:centr_crit} Let $G$ be the special HNN-extension \eqref{eq:spec_HNN-def}.
Suppose that $B \le A$ is a subgroup, and an element $g\in G$ is represented by a reduced product
$x_0t^{\e_1}x_1 t^{\e_2} \dots t^{\e_n}x_{n}$ in $G$, with $n \ge 1$. Then the equality $C_B(g)=I$ holds
in $G$, where
$$I:=C_B(x_0 \dots x_n) \cap x_0 H x_0^{-1} \cap (x_0x_1)H(x_0x_1)^{-1} \cap \dots \cap
				(x_0\dots x_{n-1})H(x_0 \dots x_{n-1})^{-1}.$$
\end{lemma}

\begin{proof} Basically, we have already shown this while proving Lemma \ref{lem:conj_crit}.
Indeed, denote $f:=g$. For any $b \in I$, if we take $b_0=1$,
the proof of sufficiency in Lemma \ref{lem:conj_crit} asserts that $bgb^{-1}=g$, i.e., $b \in C_B(g)$.
On the other hand, if $bgb^{-1}=g$ for some $b\in B$, then the proof of necessity in
Lemma \ref{lem:conj_crit} claims that $b \in I$. Therefore $C_B(g)=I$.
\end{proof}


\section{Proof of the main result}
Throughout this section we assume that $G$ is a right angled Artin group associated to some
fixed finite graph $\Gamma$ with the set of vertices $\mathcal V$.
The \textit{rank} $\rank(G)$ of $G$ is, by definition, the number of elements
in $\mathcal V$.

Our proof of the main result will make use of the following two statements below.
The next Lemma \ref{lem:RAAG-rf} was proved by Green in \cite{Green}.
It can be easily demonstrated by induction on the number of vertices in the graph associated to
a right angled Artin group using Remark \ref{rem:constr_of_RAAG} and Britton's Lemma.
On the other hand, it also follows from the linearity of such groups, which was
established by S. Humphries in \cite{Hump}.

\begin{lemma}\label{lem:RAAG-rf} Right angled Artin groups are residually finite.
\end{lemma}

We are also going to use the following important fact proved by Dyer in \cite{Dyer}.
\begin{lemma}\label{lem:virt_free-cs} Virtually free groups are conjugacy separable.
\end{lemma}

The main Theorem \ref{thm:RAAG-main} will be proved by induction on the rank of the right Artin group
$G$. The lemma below will be used to establish the inductive step:

\begin{lemma}\label{lem:ind_step} Assume that every special subgroup $B$ of $G$
satisfies the condition {\ccg} from Definition \ref{df:CCG}, and for each $g \in G$, the
$B$-conjugacy class $g^B$ is separable in $G$.

Suppose  that $A_1,\dots,A_n$ are special subgroups of $G$, $A_0$ is a conjugate of a special subgroup of
$G$, and $b,x_0,\dots,x_n,y_1,\dots,y_n \in G$ are arbitrary elements. Then for any finite index normal
subgroup $K\lhd G$ there exists a finite index normal subgroup $L \lhd G$ such that $L \le K$ and
\begin{equation} \label{eq:main_ind_step}
 \bb C_{\bA_0}(\bx_0) \cap \bigcap_{i=1}^n \bx_i \bA_i \by_i  \subseteq
\psi \left(\left[ b C_{A_0}(x_0) \cap \bigcap_{i=1}^n x_i A_i y_i\right] K\right) \,\mbox{ in } G/L,
\end{equation} where $\psi: G \to G/L $ is the natural epimorphism, $\bb:=\psi(b)$,
$\bA_i:=\psi(A_i)$, $\bx_i:=\psi(x_i)$, $i=0,\dots,n$, and $\by_j:=\psi(y_j)$, $j=1,\dots,n$.
\end{lemma}

\begin{proof} By the assumptions, $A_0=hAh^{-1}$ for some special subgroup $A$ of $G$ and
some $h\in G$. The proof will proceed by induction on $n$.

If $n=0$, then the existence of a finite index normal subgroup $L\lhd G$, $L \le K$,
enjoying \eqref{eq:main_ind_step}, follows
from the fact that the pair $(A_0,x_0)$ satisfies the Centralizer Condition \ccg, because
the pair $(A,g)$ has  \ccg, where $g:=h^{-1} x_0 h \in G$,
and $b C_{A_0}(x_0)=bh C_A(g) h^{-1}$ in $G$.

\fbox{\textit{Base of induction:}} assume $n=1$.

\underline{\textit{Case 1:}} suppose that $bC_{A_0}(x_0) \cap x_1 A_1 y_1 = \emptyset$, which
is equivalent to the condition
$x_1 \notin bC_{A_0}(x_0) y_1^{-1}A_1= bh\left[C_A(g)D\right]h^{-1}y_1^{-1}$,
where $D:=( y_1h)^{-1} A_1 (y_1h)$.

In this case, by Lemma \ref{lem:inter_conj_to_spec_sbgps}, the intersections $A \cap A_1$ and
$A \cap D$ are conjugates of special subgroups in $G$,
hence for any $f \in G$, the conjugacy classes $f^{A \cap A_1}$ and $f^{A \cap D}$
are separable in $G$
($A \cap D=cSc^{-1}$ for some special subgroup $S$ of $G$, hence $f^{A \cap D}=
c\left[(c^{-1}fc)^S\right]c^{-1}$), and the pair $(A\cap A_1,f)$ satisfies \ccg. Therefore,
Lemma \ref{lem:centr-coset_sep} allows us to conclude that the double coset $C_A(g)D$ is separable
in $G$. Thus the double coset $bC_{A_0}(x_0) y_1^{-1}A_1$  is separable as well, implying that
there is a finite index normal subgroup $N \lhd G$ such that $x_1\notin bC_{A_0}(x_0) y_1^{-1}A_1N$.
After replacing $N$ with $N \cap K$ we can assume that $N \le K$.

Now, since the pair $(A_0,x_0)$ has \ccg, there exists a finite index normal subgroup $L \lhd G$
such that $L \le N \le K$ and $\psi^{-1} \left( C_{\bA_0}(\bx_0) \right) \subseteq C_{A_0}(x_0)N$ in $G$
(using the same notations as in the formulation of the lemma).
Therefore, $\psi^{-1} \left( \bb C_{\bA_0}(\bx_0) \by_1^{-1} \bA_1  \right) \subseteq
bC_{A_0}(x_0) y_1^{-1}A_1N$, yielding that
$x_1 \notin \psi^{-1} \left( \bb C_{\bA_0}(\bx_0) \by_1^{-1} \bA_1  \right)$. Hence
$\bb C_{\bA_0}(\bx_0) \cap \bx_1\bA_1 \by_1 = \emptyset$ in $G/L$, implying that \eqref{eq:main_ind_step}
holds in this first case.

\underline{\textit{Case 2:}} $bC_{A_0}(x_0) \cap x_1 A_1 y_1 \neq \emptyset$ in $G$.

Let us make the following general observation:
\begin{rem}\label{rem:inter_cosets-gen} Let $H,F$ be subgroups of a group $G$ such that
$bH \cap xFy \neq \emptyset$ in $G$ for some elements $b,x,y \in G$. Then for any
$a \in bH \cap xFy$ we have $bH \cap xFy=a(H \cap y^{-1} F y)$.
\end{rem}

Thus we can pick any $a \in bC_{A_0}(x_0) \cap x_1 A_1 y_1$, and according to
Remark \ref{rem:inter_cosets-gen}, we will have
$bC_{A_0}(x_0) \cap x_1 A_1 y_1 = a\left( C_{A_0}(x_0) \cap y_1^{-1} A_1 y_1\right)=
aC_{E}(x_0)$ in $G$, where $E:=A_0 \cap y_1^{-1} A_1 y_1 \le G$.
By Lemma \ref{lem:inter_conj_to_spec_sbgps}, $E=cSc^{-1}$ for some special subgroup $S$ of $G$ and
some $c \in A_0$.

As we saw in the beginning of the proof, it follows from our assumptions that
the pair $(E,x_0)$ satisfies \ccg. Hence there must exist a finite index normal subgroup
$M\lhd G$ such that $M \le K$ and
$C_{\varphi(E)}(\varphi(x_0)) \subseteq \varphi(C_E(x_0)K)$ in $G$,
where $\varphi: G \to G/M$ denotes the natural epimorphism. On the other hand, by
Lemma \ref{lem:cond_(I)}, there is $L \lhd G$ such that $|G:L|<\infty$, $L \le M \le K$ and
$\psi(A) \cap \psi\bigl((y_1h)^{-1} A_1 y_1h_1\bigr)
\subseteq \psi\bigl(A \cap (y_1h)^{-1} A_1 y_1h_1\bigr)\psi(M)$ in $G/L$.
Therefore
\begin{multline}\label{eq:ca_0-inclusion}
\bA_0 \cap \by_1^{-1} \bA_1 \by_1 = \psi(h) \left[ \psi(A) \cap \psi\bigl((y_1h)^{-1} A_1
y_1h_1\bigr)\right]
\psi(h^{-1}) \subseteq \\ \psi(h) \left[\psi\bigl(A \cap (y_1h)^{-1} A_1 y_1h_1\bigr)\psi(M)\right]
\psi(h^{-1})=\psi(A_0 \cap y_1^{-1} A_1 y_1)\psi(M)=\psi(E)\psi(M) \end{multline}
(in the notations from the formulation of the lemma).
Observe that
$\ba:=\psi(a) \in \bb C_{\bA_0}(\bx_0) \cap \bx_1 \bA_1 \by_1$, hence, by Remark \ref{rem:inter_cosets-gen},
$\bb C_{\bA_0}(\bx_0) \cap \bx_1 \bA_1 \by_1=\ba\left( C_{\bA_0}(\bx_0) \cap \by_1^{-1} \bA_1 \by_1\right)$,
and applying \eqref{eq:ca_0-inclusion} we achieve
\begin{equation}
\label{eq:one_more_eq_main}
\bb C_{\bA_0}(\bx_0) \cap \bx_1 \bA_1 \by_1 \subseteq \ba C_{\psi(E M)}(\bx_0).\end{equation}

Since $L \le M$, there is an epimorphism $\xi: G/L \to G/M$ such that $\ker(\xi)=\psi(M)$
and $\varphi=\xi \circ \psi$. Note that
$\xi( \psi(EM))=\xi( \psi(E) \psi(M))=\varphi(E)$ and $\xi(\bx_0)=
\varphi(x_0)$. Consider any $z \in C_{\psi(EM)}(\bx_0)$ in $G/L$. Then
$$\xi(z) \in C_{\varphi(E)}(\varphi(x_0)) \subseteq \varphi (C_E(x_0)K)=\xi
\bigl(\psi (C_E(x_0)K) \bigr).$$ Hence $z \in \psi (C_E(x_0)K)\ker (\xi)=\psi (C_E(x_0)K)$
because $\ker(\xi)=\psi(M) \le \psi(K)$. Thus we have shown that
$ C_{\psi(EM)}(\bx_0) \subseteq \psi (C_E(x_0)K)$ in $G/L$.
Finally, combining this with
\eqref{eq:one_more_eq_main}, we obtain $$\bb C_{\bA_0}(\bx_0) \cap \bx_1 \bA_1 \by_1 \subseteq
\ba \psi (C_E(x_0)K)=\psi(aC_E(x_0)K)=\psi(\left[bC_{A_0}(x_0) \cap x_1 A_1 y_1\right] K),$$
therefore \eqref{eq:main_ind_step} holds in  Case 2.

Thus we have established the base of induction.

\fbox{\textit{Step of induction:}} suppose that $n \ge 2$ and
the statement of the lemma has been proved for $n-1$.

If $ b C_{A_0}(x_0) \cap \bigcap_{i=1}^{n-1} x_i A_i y_i=\emptyset$ in $G$, then, by the
induction hypothesis, there is a finite index
normal subgroup $L\lhd G$ such that $L \le K$ and
$$
 \bb C_{\bA_0}(\bx_0) \cap \bigcap_{i=1}^{n-1} \bx_i \bA_i \by_i  \subseteq
\psi \left(\left[ b C_{A_0}(x_0) \cap \bigcap_{i=1}^{n-1} x_i A_i y_i\right] K\right)=\emptyset
~\mbox{ in } G/L.$$
Hence, the left-hand side of \eqref{eq:main_ind_step} will also be empty, and thus \eqref{eq:main_ind_step}
will be true.

Therefore, we can assume that $ b C_{A_0}(x_0) \cap \bigcap_{i=1}^{n-1} x_i A_i y_i \neq \emptyset$ in $G$.
But in this case we can apply Remark \ref{rem:inter_cosets-gen} $(n-1)$ times to find some $a \in G$
such that $ b C_{A_0}(x_0) \cap \bigcap_{i=1}^{n-1} x_i A_i y_i=
a\left( C_{A_0}(x_0) \cap \bigcap_{i=1}^{n-1} y_i^{-1} A_i y_i\right)=aC_E(x_0)$,
where $E:=A_0 \cap \bigcap_{i=1}^{n-1} y_i^{-1} A_i y_i$ is a conjugate of
a special subgroup in $G$ by Lemma \ref{lem:inter_conj_to_spec_sbgps}. Now we can use the base
of induction $n=1$, to find a finite index normal subgroup $M \lhd G$ such that $M \le K$ and
for the natural epimorphism $\varphi:G \to G/M$ we have
\begin{equation} \label{eq:main_lem_M}
\varphi^{-1} \left( \varphi(a) C_{\varphi(E)}(\varphi(x_0)) \cap \varphi(x_n A_n y_n)\right)
\subseteq \left[aC_E(x_0) \cap x_n A_n y_n \right]K ~\mbox{ in } G.\end{equation}

By the induction hypothesis, there exists a finite index
normal subgroup $L\lhd G$ such that $L \le M \le K$ and
\begin{equation} \label{eq:ind_step_within_main_ind_step}
\psi^{-1} \left( \bb C_{\bA_0}(\bx_0) \cap \bigcap_{i=1}^{n-1} \bx_i \bA_i \by_i \right) \subseteq
\left[ b C_{A_0}(x_0) \cap \bigcap_{i=1}^{n-1} x_i A_i y_i\right] M~\mbox{ in } G.
\end{equation}

Combining \eqref{eq:ind_step_within_main_ind_step} with \eqref{eq:main_lem_M}
and recalling that $\ker(\psi)=L \le M=\ker(\varphi)$ we obtain the following in $G$:
\begin{multline*} \psi^{-1} \left( \bb C_{\bA_0}(\bx_0) \cap \bigcap_{i=1}^{n} \bx_i \bA_i \by_i \right)=
\psi^{-1} \left( \bb C_{\bA_0}(\bx_0) \cap \bigcap_{i=1}^{n-1} \bx_i \bA_i \by_i \right) \cap
\psi^{-1} \left( \bx_n \bA_n \by_n \right) \subseteq \\
\left[ b C_{A_0}(x_0) \cap \bigcap_{i=1}^{n-1} x_i A_i y_i\right] M \cap
x_n A_n y_n L \subseteq aC_E(x_0)M \cap x_n A_n y_n M
\subseteq \\ \varphi^{-1} \left( \varphi(a) C_{\varphi(E)}(\varphi(x_0))\right) \cap
\varphi^{-1} \left( \varphi(x_n A_n y_n) \right)=
\varphi^{-1} \left( \varphi(a) C_{\varphi(E)}(\varphi(x_0))
\cap \varphi(x_n A_n y_n)\right)
\subseteq \\ \left[aC_E(x_0) \cap x_n A_n y_n \right]K=
\left[b C_{A_0}(x_0) \cap \bigcap_{i=1}^{n} x_i A_i y_i\right] K.
\end{multline*}

Hence \eqref{eq:main_ind_step} holds in $G/L$ and
we have verified the inductive step, finishing the proof of the lemma.
\end{proof}

The next two statements basically establish the main result. These Lemmas \ref{lem:main-CS} and
\ref{lem:main-CC} will be proved by simultaneous induction on $\rank(G)$.
The proofs of each of these two lemmas when $\rank(G)=r$
will use both of their conclusions about right angled Artin groups of ranks
strictly less than $r$. 

\begin{lemma}\label{lem:main-CS} If $G$ is a right angled Artin group of rank $r$, then for any special
subgroup $B$ of $G$ and any element $g \in G$, the $B$-conjugacy class $g^B$ is separable in $G$.
\end{lemma}

\begin{lemma}\label{lem:main-CC} Let $G$ be a right angled Artin group of rank $r$. Then every special
subgroup $B$ of $G$ satisfies the Centralizer Condition {\ccg} from Definition \ref{df:CCG}.
\end{lemma}

The base of induction for both lemmas is $r=0$, that is,  when $G$ is the trivial group.
In this case the two  statements are trivial.
Therefore the proofs of Lemmas \ref{lem:main-CS} and \ref{lem:main-CC} start with assuming that both of
their claims have been established for all right angled Artin groups of rank $<r$, and
will aim to prove the inductive step by considering the case when $\rank(G)=r \ge 1$.

The proofs of Lemmas \ref{lem:main-CS} and \ref{lem:main-CC} make use of the four auxiliary statements below.
These statements -- Lemmas \ref{lem:claim_A} through \ref{lem:claim_D} -- start with a right angled Artin
group $G$ of rank $r$ (presented as a special HNN-extension 
\begin{equation}\label{eq:def_of_A-main}
G=\langle A,t \,\|\, tht^{-1}=h, ~\forall\,h \in H\rangle.\end{equation}
of a maximal special subgroup $A \le G$ with respect to some special subgroup $H \le A$ -- see
Remark \ref{rem:RAAG_spec_HNN_over_max_sbgp}), and assume that Lemmas \ref{lem:main-CS} and
\ref{lem:main-CC} have already been established for $A$, since $\rank(A)=r-1<r=\rank(G)$.

\begin{lemma}\label{lem:claim_A} Suppose that $B$ is a special subgroup
of $G$ contained in $A$, $g \in G \setminus A$ and  $f \in G \setminus g^B$. Then
there exists an epimorphism $\psi$ from $A$ onto a finite group $Q$
such that 
for the corresponding extension $\tilde \psi: G \to P$ from $G$ onto the
special HNN-extension $P$ of $Q$ (with respect to $\psi(H)$), obtained according to
Remark \ref{rem:univ_spec_HNN},
we have $\tilde\psi(f) \notin \tilde\psi(g)^{\tilde\psi(B)}$ in $P$.
\end{lemma}

\begin{proof}
Let $x_0t^{\e_1}x_1 t^{\e_2} \dots t^{\e_n}x_{n}$ and $y_0t^{\z_1}y_1 t^{\z_2} \dots t^{\z_m}y_{m}$
be reduced products representing $g$ and $f$ in $G$ respectively. Since $g \notin A$ we have $n \ge 1$.

\underline{\textit{Case 1:}} suppose, at first, that the condition (i)
from Lemma \ref{lem:conj_crit} does not hold.
By Corollary \ref{cor:prod_retr-sep} and Lemma \ref{lem:RAAG-rf}, the special subgroup $H=HH$ of $A$
is closed in $\PT(A)$, hence there is a finite index normal subgroup $L \lhd A$ and the corresponding
epimorphism $\psi:A \to Q:=A/L$
such that such that $\psi(x_i) \notin \psi(H)$, $i=1,\dots,n-1$, and
$\psi(y_j) \notin \psi(H)$, $j=1,\dots,m-1$.
Let \begin{equation} \label{eq:def_of_P}
P:=\langle Q,s\,\|\, sqs^{-1}=q,~\forall\, q \in \psi(H) \rangle\end{equation} be the special
HNN-extension of $Q$ with respect to $\psi(H)$. By Remark \ref{rem:univ_spec_HNN}, $\psi$
can be extended to a homomorphism $\tilde \psi:G \to P$ such that $\tilde \psi|_A=\psi$ and $\tilde\psi(t)=s$.
Therefore,
$\tilde \psi(g)=\bx_0s^{\e_1}\bx_1 s^{\e_2} \dots s^{\e_n}\bx_{n}$,
$\tilde \psi(f)=\by_0s^{\z_1}\by_1 s^{\z_2} \dots s^{\z_m}\by_{m}$ and these products are reduced in $P$,
where $\bx_i:=\psi(x_i)$ and $\by_j:=\psi(y_j)$ for all $i=0,\dots,n$, $j=0,\dots,m$.
And since the condition (i) did not hold for $g$ and $f$, this condition will not
hold for $\tilde\psi(g)$ and $\tilde \psi(f)$. Therefore,
$\tilde\psi(f) \notin \tilde\psi(g)^{\tilde\psi(B)}$ by Lemma  \ref{lem:conj_crit}.

Thus we can now assume that $n=m$ and $\e_i=\z_i$ for $i=1,\dots,n$.
Denote $x:=x_0 \dots x_n \in A$ and $y:=y_0\dots y_n \in A$.

\underline{\textit{Case 2:}} suppose that $y\notin x^B$ in $A$. Then,
by the induction hypothesis, $x^B$ is separable in $A$, hence there is a finite group $Q$
and an epimorphism $\psi:A \to Q$ such that $\psi(y) \notin \psi(x)^{\psi(B)}$.
Let $P$ be the special HNN-extension of $Q$ defined by \eqref{eq:def_of_P}, and let
$\tilde\psi:G \to P$ be the corresponding extension of $\psi$ with $\tilde\psi(t)=s$.
By Lemma \ref{lem:spec_HNN_retractions}, there is a retraction $\tilde\rho_Q \in End(P)$
of $P$ onto $Q$ (extending the identity map on $Q$) satisfying $\tilde\rho_Q(s)=1$.
Therefore, using the above notations, we have
$$\tilde\rho_Q \left(\tilde\psi(f)\right) =\tilde\rho_Q(\by_0s^{\e_1}\by_1 s^{\e_2} \dots s^{\e_n}\by_{n})=
\by_0 \by_1 \dots \by_n=\psi(y) \in Q,$$
similarly, $\tilde\rho_Q \left(\tilde\psi(g)\right)=\psi(x) \in Q$. And since
$\tilde\rho_Q \left(\tilde\psi(B)\right)=\psi(B)$
and $\psi(y) \notin \psi(x)^{\psi(B)}$ we can conclude that
$\tilde\psi(f) \notin \tilde\psi(x)^{\tilde\psi(B)}$.

\underline{\textit{Case 3:}} we can now assume that both of the conditions (i) and (ii) of
Lemma~\ref{lem:conj_crit} are satisfied. Choose any $b_0 \in B$ such that $y=b_0 x b_0^{-1}$.
As $f \notin g^B$ in $G$, according to  Lemma~\ref{lem:conj_crit}
we must have $I=\emptyset$ in $A$, where
$$I:=b_0 C_B(x) \cap y_0 H x_0^{-1} \cap (y_0y_1)H(x_0x_1)^{-1} \cap \dots \cap
(y_0\dots y_{n-1})H(x_0 \dots x_{n-1})^{-1}.$$

As we saw earlier, $H$ is separable in $A$, therefore there is a finite index normal subgroup
$K \lhd A$ such that $x_i \notin HK$ and $y_i \notin HK$ for $1\le i \le n-1$. Now, since
$\rank(A)=r-1<r$, the right angled Artin group $A$ satisfies the claims of Lemmas \ref{lem:main-CS}
and \ref{lem:main-CC} by the induction hypothesis.
Consequently, we can apply Lemma \ref{lem:ind_step} to $A$ and $K$, finding a finite index normal
subgroup $L \lhd A$ such that $L \le K$ and
\begin{equation}\label{eq:IK-empty}
\bb_0 C_{\bB}(\bx) \cap \bigcap_{i=1}^n \bx_i \bH \by_i \subseteq
\psi(IK)=\emptyset ~\mbox{ in } Q:=A/L,\end{equation}
where $\bb_0, \bB, \bx, \bx_i, \bH,\by_i$ denote the $\psi$-images of
$b_0, B, x, x_i, H,y_i$ in $Q$ respectively.
As before we can extend $\psi$ to a homomorphism $\tilde\psi:G \to P$, where $P$ is given by
\eqref{eq:def_of_P}, and $\tilde\psi(t)=s$. Since $L \le K$ we have
$\bx_i,\by_i \notin \psi(H)$ for $i=1,\dots,n$, and so
$\bx_0s^{\e_1}\bx_1 s^{\e_2} \dots s^{\e_n}\bx_{n}$ and
$\by_0s^{\e_1}\by_1 s^{\e_2} \dots s^{\e_n}\by_{n}$ are reduced products
in $P$ representing the elements $\tilde \psi(g)$ and $\tilde \psi(f)$ respectively.
Thus, Lemma \ref{lem:conj_crit}, in view of \eqref{eq:IK-empty}, implies that
$\tilde\psi(f)\notin \tilde\psi(g)^{\tilde\psi(B)}$ in $P$. And Lemma \ref{lem:claim_A} is proved.
\end{proof}

\begin{lemma}\label{lem:claim_B} Suppose that $g_0,f_0,f_1,\dots, f_m \in G$,
and the products $g_0=t^{\e_1}x_1 \dots t^{\e_n}x_{n}$, $f_0=t^{\z_1}y_1  \dots t^{\z_k}y_{k}$
are cyclically reduced in $G$, with $n \ge 1$. If $f_j \notin g_0^H$ for every $j=1,\dots,m$, then
there is a finite group $Q$ and an epimorphism $\psi: A \to Q$
such that for the corresponding epimorphism $\tilde \psi:G \to P$, extending $\psi$,
with $\tilde\psi(t)=s$ (where $P$ is the special HNN-extension
given by \eqref{eq:def_of_P}), all of the following are true:
\begin{itemize}
	\item $\tilde\psi(f_j) \notin \tilde\psi(g_0)^{\tilde\psi(H)}$ in $P$, for each $j \in \{1,\dots,m\}$;
	\item the products $\tilde\psi(g_0)=s^{\e_1}\bx_1  \dots s^{\e_n}\bx_{n}$ and
			  $\tilde\psi(f_0)=s^{\z_1}\by_1  \dots s^{\z_k}\by_{k}$ are
			  cyclically reduced in $P$, where $\bx_i:=\tilde\psi(x_i)$, $i=1,\dots,n$,
			  $\by_l:=\tilde\psi(y_l)$, $l=1,\dots,k$.										
\end{itemize}
\end{lemma}

\begin{proof} For every $j=1,\dots,m$, since $f_j \notin g_0^H$ in $G$,
we can apply Lemma \ref{lem:claim_A} (as $H \le A$ is a special subgroup of $G$),
to find a finite index normal subgroup $L_j \lhd A$,
such that $\tilde\psi_j(f_j) \notin \tilde\psi_j(g_0)^{\tilde\psi_j(H)}$ in $P_j$,
where $\tilde\psi_j: G \to P_j$
is the homomorphism (obtained according to Remark \ref{rem:univ_spec_HNN})
extending the natural epimorphism $\psi_j:A \to A/L_j$, and $P_j$ is the special
HNN-extension of $A/L_j$ with respect to $\psi_j(H)$.

Now, since $H$ is separable in $A$ (by Lemma \ref{lem:RAAG-rf} and Corollary \ref{cor:prod_retr-sep}),
there is a finite index normal subgroup $K \lhd A$
such that $x_i \notin HK$ whenever $x_i \notin H$, for all $i=1,\dots,n$, and
$y_l \notin HK$ whenever $y_l \notin H$, for all $l=1,\dots,k$. Define
the finite index normal subgroup $L$ of $A$ by $L:=L_1 \cap \dots \cap L_m \cap K$,
and let $\psi: A \to Q:=A/L$ be the natural epimorphism. Observe that for each $j$,
the map $\psi_j$ factors through the map $\psi$. Hence, once we let $\tilde\psi:G \to P$
be the extension of $\psi$ as in the formulation of Lemma \ref{lem:claim_B},
the Universal Property of special HNN-extensions (Remark \ref{rem:univ_spec_HNN})
will imply that $\tilde\psi_j$ factors through $\tilde\psi$, for every $j=1,\dots,m$.
Consequently, $\tilde\psi(f_j) \notin \tilde\psi(g_0)^{\tilde\psi(H)}$ in $P$,
for each $j \in \{1,\dots,m\}$. The second assertion of Claim B holds due to the choice
of $K$ and because $L \le K$. Thus Lemma \ref{lem:claim_B} is proved.
\end{proof}

\begin{lemma}\label{lem:claim_C} Let $K \lhd G$ be a normal subgroup of finite index,
let $B$ be a special subgroup of $G$ with $B \le A$,
and let an element $g \in G\setminus A$ be represented
by a reduced product $x_0t^{\e_1}x_1 t^{\e_2} \dots t^{\e_n}x_{n}$ in $G$, with $n \ge 1$.
Then there is a finite group $Q$ and an epimorphism $\psi: A \to Q$
such that for the corresponding homomorphism $\tilde \psi:G \to P$, extending $\psi$
and obtained according to Remark \ref{rem:univ_spec_HNN},
with $\tilde\psi(t)=s$ (where $P$ is the special HNN-extension
given by \eqref{eq:def_of_P}), all of the following are true:
\begin{itemize}
	\item $C_{\tilde\psi(B)}\left(\tilde\psi(g)\right) \subseteq \tilde\psi \left( C_B(g) K\right)~\mbox{ in }P$;
	\item $\ker(\psi) \le A \cap K$ and $\ker(\tilde\psi) \le K$.
\end{itemize}
\end{lemma}

\begin{proof}
Since $A$ is residually finite (Lemma \ref{lem:RAAG-rf}), the special subgroup $H=HH$ is closed in $\PT(A)$
by Corollary \ref{cor:prod_retr-sep}. Therefore there exists a finite index normal subgroup
$M_1 \lhd A$ such that $x_i \notin HM_1$ in $A$ for all $i=1,\dots,n-1$. As usual, we
can replace $M_1$ with $M_1 \cap K_1$, to make sure that $M_1 \le K_1$, where $K_1:=A \cap K$.

Note that, according to Lemma \ref{lem:centr_crit}, $C_B(g)=I$ in $G$, where
$$I:=C_B(x) \cap x_0 H x_0^{-1} \cap (x_0x_1)H(x_0x_1)^{-1} \cap \dots \cap
				(x_0\dots x_{n-1})H(x_0 \dots x_{n-1})^{-1},$$
and $x:=x_0 \dots x_n\in A$.

Since $\rank(A)=r-1<r$, by the induction hypothesis the claims of Lemmas \ref{lem:main-CS}
and \ref{lem:main-CC} hold for $A$. Hence,
we can use Lemma \ref{lem:ind_step} to find a finite index normal subgroup
$L_1\lhd A$ such that $L_1 \le M_1 \le K_1$ and, for the corresponding epimorphism $\psi:A \to Q:=A/L_1$,
we have
$$J:=C_\bB(\bx) \cap \bx_0 \bH \bx_0^{-1} \cap (\bx_0\bx_1)\bH(\bx_0\bx_1)^{-1} \cap \dots \cap
				(\bx_0\dots \bx_{n-1})\bH(\bx_0 \dots \bx_{n-1})^{-1} \subseteq \psi(IM_1)$$
in $Q$, where $\bB$, $\bx$, $\bH$ and $\bx_i$ denote the $\psi$-images of $B$, $x$, $H$ and $x_i$ in $Q$, $i=0,\dots,n$.

Let $P$ be the special HNN-extension of $Q$ given by \eqref{eq:def_of_P}, and let
$\tilde\psi:G \to P$ be the extension of $\psi$ provided by Remark \ref{rem:univ_spec_HNN},
with $\tilde \psi(t)=s$. Since $x_i \notin HL_1$ in $A$ for $i=1,\dots,n-1$,
the product $\bx_0s^{\e_1}\bx_1  \dots s^{\e_n}\bx_{n}$ is reduced and
represents the element $\tilde\psi(g)$ in $P$. Consequently, Lemma \ref{lem:centr_crit}
tells us that $C_{\tilde\psi(B)}\left(\tilde\psi(g)\right)=J$ in $P$. And
noting that $\psi(M_1) \le \psi(K_1) =\tilde\psi(K_1) \le \tilde\psi(K)$, we achieve
$$C_{\tilde\psi(B)}\left(\tilde\psi(g)\right)=J \subseteq \psi(I) \psi(M_1) \subseteq
\tilde\psi(I) \tilde\psi (K)= \tilde\psi\left( C_B(g) K\right)~\mbox{ in }P.$$

Finally, observe that $\ker(\psi)=L_1  \le K_1= A \cap K$ and
$\ker(\tilde\psi)$ is the normal closure of $L_1$ in $G$ (by Lemma \ref{lem:ker_of_tilde_psi}).
And since $L_1 \le K \lhd G$, we see that
$\ker(\tilde\psi)\le K$ in $G$.
Thus Lemma \ref{lem:claim_C} has been established.
\end{proof}

\begin{lemma}\label{lem:claim_D} Let $K \lhd G$ be a normal subgroup of finite index and let $g_0=t^{\e_1}x_1 \dots t^{\e_n}x_{n}$ be a cyclically
reduced product in $G$, with $n \ge 1$. Then
there exists an epimorphism $\psi$ from $A$ onto a finite group $Q$
such that 
for the corresponding extension $\tilde \psi: G \to P$ from $G$ onto the
special HNN-extension $P$ of $Q$ (given by \eqref{eq:def_of_P}), with $\tilde\psi|_A=\psi$
and $\tilde\psi(t)=s$, we have
$$\ker(\psi) \le A \cap K, ~\ker(\tilde\psi) \le K~\mbox{in $G$, and }~
C_P\left(\tilde\psi(g_0)\right) \subseteq  \tilde\psi\bigl(C_G(g_0)K\bigr) \,\mbox{ in }P.$$
\end{lemma}

\begin{proof} Clearly there exists $m \in \{0,1,\dots,n\}$ such that we can enumerate all the prefixes
$p_1,\dots,p_{n+1}$ of $g_0$ so that $p_j^{-1} g_0 p_j \notin g_0^H$ in $G$ whenever
$1 \le j \le m$, and $p_j^{-1} g_0 p_j \in g_0^H$ in $G$ whenever
$m+1 \le j \le n+1$. For each $j \in \{m+1,m+2,\dots, n+1\}$, choose $h_j \in H$ such that
$h_jp_j^{-1} g_0 p_jh_j^{-1}=g_0$ in $G$, and set $\Omega:=\{h_jp_j^{-1} \,|\,m+1 \le j \le n+1\} \subset G$.

Let $f_0:=g_0=t^{\e_1}x_1 \dots t^{\e_n}x_{n}$ and
$f_j:=p_j^{-1} g_0 p_j$ for $j=1,\dots,m$. Applying Lemma~\ref{lem:claim_B} to $g_0,f_0,\dots,f_m \in G$
we find a finite group $Q_1$, an epimorphism $\psi_1: A \to Q_1$, the special HNN-extension
$P_1$ of $Q_1$ with respect to $\psi_1(H)$, and the corresponding extension $\tilde\psi_1: G \to P_1$
of $\psi_1$ (obtained by Remark \ref{rem:univ_spec_HNN}), such that
$\tilde\psi_1(f_j) \notin \tilde\psi_1(g_0)^{\tilde\psi_1(H)}$ in $P_1$,
for each $j \in \{1,\dots,m\}$,
and the element $\tilde\psi_1(g_0)$ is cyclically reduced in $P_1$.

On the other hand, by Lemma \ref{lem:claim_C}, there exist a finite group $Q_2$,
an epimorphism $\psi_2: A \to Q_2$, the special HNN-extension
$P_2$ of $Q_2$ with respect to $\psi_2(H)$, and the corresponding extension $\tilde\psi_2: G \to P_2$
of $\psi_2$, such that
$C_{\tilde\psi_2(H)}\left(\tilde\psi_2(g_0)\right) \subseteq \tilde\psi_2 \left( C_H(g_0) K\right)$ in
$P_2$, $\ker(\psi_2) \le A \cap K$ and $\ker(\tilde\psi_2) \le K$.

Define a finite index normal subgroup $L_0 \lhd A$ by
$L_0:=\ker(\psi_1) \cap \ker(\psi_2)\le A \cap K$, and let $\psi: A \to Q:=A/L_0$ be the natural epimorphism.
By Remark \ref{rem:univ_spec_HNN}, there is an epimorphism $\tilde\psi:G \to P$,
extending $\psi$ so that $\tilde\psi(t)=s$, where $P$ is the special HNN-extension of $Q$
given by \eqref{eq:def_of_P}. Since $\ker(\psi)=L_0 \le\ker(\psi_i)$,
the maps $\psi_i: A \to Q_i$ factor through $\psi$ for $i=1,2$. Consequently, according to the
Universal Property of special HNN-extensions (Remark \ref{rem:univ_spec_HNN}),
the maps $\tilde\psi_i: G \to P_i$ factor through $\tilde\psi:G \to P$ for $i=1,2$.
Therefore we have
\begin{equation} \label{eq:oooh1} \tilde\psi(g_0)\mbox{ is cyclically reduced and}~
\tilde\psi(f_j) \notin \tilde\psi(g_0)^{\tilde\psi(H)}
~\mbox{ in $P$, $\forall\, j \in \{1,\dots,m\}$}.
\end{equation}

On the other hand, since $\ker(\tilde\psi) \le \ker(\tilde\psi_2) \le K$ in $G$, we also have
$$ \tilde\psi^{-1}\left(C_{\tilde\psi(H)}\left(\tilde\psi(g_0)\right) \right)
\subseteq \tilde\psi_2^{-1}\left(C_{\tilde\psi_2(H)}\left(\tilde\psi_2(g_0)\right) \right) \subseteq
C_H(g_0) K~\mbox{ in } G,
$$ which implies
\begin{equation} \label{eq:oooh2}
C_{\tilde\psi(H)}\left(\tilde\psi(g_0)\right) \subseteq \tilde\psi \left( C_H(g_0) K\right)~\mbox{ in }P.
\end{equation}

\underline{\textit{Case 1:}} $x_n \in H$. Then, according to Proposition \ref{prop:descr_centr},
$n=1$, $C_G(g_0)=\langle t \rangle C_H(g_0)$ in $G$, and
$C_P \left(\tilde\psi(g_0)\right)=\langle s \rangle  C_{\tilde\psi(H)}\left(\tilde\psi(g_0)\right)$ in $P$.
Recalling \eqref{eq:oooh2}, we see that
$$C_P \left(\tilde\psi(g_0)\right) \subseteq \left\langle \tilde\psi(t) \right\rangle \,
\tilde\psi
\left( C_H(g_0) K\right)=
\tilde\psi \left( C_G(g_0) K\right)~\mbox{ in }P.$$

\underline{\textit{Case 2:}} $x_n \in A\setminus H$. In this case \eqref{eq:oooh1}
implies that $\tilde\psi(p_{m+1}), \dots, \tilde\psi(p_{n+1})$ is the list of all prefixes
of $\tilde\psi(g_0)$ satisfying $\tilde\psi(p_j)^{-1} \tilde\psi(g_0) \tilde\psi(p_j)
\in \tilde\psi(g_0)^{\tilde\psi(H)}$, because if $1 \le j \le m$, then
$\tilde\psi(p_j)^{-1} \tilde\psi(g_0) \tilde\psi(p_j)=\tilde\psi(p_j^{-1}g_0p_j)=\tilde\psi(f_j)
\notin \tilde\psi(g_0)^{\tilde\psi(H)}$ in $P$.

Therefore, by Proposition \ref{prop:descr_centr},
$C_P\left(\tilde\psi(g_0)\right)=C_{\tilde\psi(H)}\left(\tilde\psi(g_0)\right)
\left\langle \tilde\psi(g_0) \right\rangle \bar \Omega$, where
$\bar\Omega:=\{\tilde\psi(h_j)\tilde\psi(p_j)^{-1} \,|\,m+1 \le j \le n+1\} =\tilde\psi(\Omega)\subset P$.
Thus, recalling \eqref{eq:oooh2}, we achieve
$$C_P\left(\tilde\psi(g_0)\right) \subseteq \tilde\psi\bigl(C_H(g_0) K \langle g_0 \rangle \Omega\bigr)
= \tilde\psi(C_G(g_0)K)~\mbox{ in }P.$$

In either of the two cases we have shown that
$C_P\left(\tilde\psi(g_0)\right) \subseteq  \tilde\psi(C_G(g_0)K)$ in $P$. This completes the
proof of Lemma \ref{lem:claim_D}.
\end{proof}

We are finally ready to prove the two main Lemmas  \ref{lem:main-CS} and \ref{lem:main-CC}
announced above.

\begin{proof}[Proof of Lemma \ref{lem:main-CS}] There are two separate cases to consider.

\fbox{\textit{Case 1:}} $B \neq G$.

Choose a maximal special subgroup $A$ of $G$
containing $B$. Then $A$ is a right angled Artin group of rank $r-1$,
and, according to Remark \ref{rem:RAAG_spec_HNN_over_max_sbgp}, $G$ splits as a special HNN-extension \eqref{eq:def_of_A-main} of $A$ with respect to
some special subgroup $H$ of $A$.

If $g \in A$, then $g^B$ is closed in $\PT(A)$ by the induction hypothesis. Since $G$ is residually finite
(Lemma \ref{lem:RAAG-rf}), $g^B$ is separable in $G$ by Lemma \ref{lem:sep_subset_of_retr}.

Thus we can suppose that $g \in G \setminus A$.
Take any element $f \in G \setminus g^B$. 
Let $Q$, $P$, $\psi: A \to Q$ and $\tilde \psi:G \to P$ be given by Lemma \ref{lem:claim_A},
so that $\tilde\psi(f) \notin \tilde\psi(g)^{\tilde\psi(B)}$ in $P$.

Observe that $P$ is a virtually free group as an HNN-extension of the finite group $Q$,
hence $P$ is residually finite. Since $\tilde\psi(B)=\psi(B) \subseteq Q$ is finite,
$\tilde\psi(g)^{\tilde\psi(B)}$ is a finite subset of $P$. Hence there is a homomorphism
$\xi: P \to R$ from $P$ to a finite group $R$ such that
$\xi\left(\tilde\psi(f)\right)\notin \xi\left(\tilde\psi(g)^{\tilde\psi(B)}\right)$ in $R$.
Consequently, the homomorphism $\varphi: G \to R$, defined by $\varphi:=\xi \circ \tilde\psi$,
satisfies the condition $\varphi(f) \notin \varphi(g^B)$. Therefore we have shown that
$g^B$ is separable in $G$ in Case 1.

\fbox{\textit{Case 2:}} $B=G$.

If $g=1$ then $g^G=\{1\}$ is separable in $G$ because $G$ is residually finite (Lemma \ref{lem:RAAG-rf}).
Hence we can suppose that $g \neq 1$. But then, by Lemma \ref{lem:el_in_RAAG_non_conj_into_max_spec},
there exists a maximal special subgroup $A$ of $G$ such that $g \notin A^G$. The group $G$
is a special HNN-extension \eqref{eq:def_of_A-main} of $A$ with respect to a certain special subgroup
$H \le A$ (by Remark \ref{rem:RAAG_spec_HNN_over_max_sbgp}).
Obviously, $g$ is conjugate in $G$ to some cyclically reduced product
$g_0=t^{\e_1}x_1 \dots t^{\e_n}x_{n}$ with $n \ge 1$, because $g \notin A^G$.
This implies that $g^G=g_0^G$ in $G$.

To show that $g^G$ is closed in $\PT(G)$, consider any element $f \in G\setminus g^G$.

\underline{\textit{Sub-case 2.1:}} suppose, at first, that $f \notin A^G$.
Then we can find a cyclically reduced product
$f_0=t^{\z_1}y_1  \dots t^{\z_m}y_{m} \in f^G$.
Let $f_1,f_2,\dots,f_m$ be the list of all cyclic permutations of $f_0$ in $G$.

Observe that $f_j \notin g_0^H \subset g^G$ for every $j=1,\dots,m$, because $f_0 \notin g^G$.
Therefore we can apply Lemma \ref{lem:claim_B} to find $Q$, $P$, $\psi:A \to Q$ and
$\tilde\psi: G \to P$ from its claim.

Since $\tilde\psi(f_1), \dots, \tilde\psi(f_m)$ is the list of all cyclic permutations
of $\tilde\psi(f_0)$ in $P$,  Lemma \ref{lem:claim_B}, together with Lemma \ref{lem:Collins},
imply that $\tilde\psi(f_0) \notin \tilde\psi(g_0)^P$ in $P$. Now, according to
Lemma \ref{lem:virt_free-cs}, there is a homomorphism $\xi: P \to R$ such that
$R$ is a finite group and $\xi\left(\tilde\psi(f_0)\right) \notin \xi\left(\tilde\psi(g_0)\right)^R$.
Therefore, defining the homomorphism $\varphi:G \to R$ by $\varphi:=\xi\circ \tilde\psi$ we achieve
$\varphi(f_0) \notin \varphi(g_0)^R$ in $R$.
But since $\varphi(f)$ is conjugate to $\varphi(f_0)$ and $\varphi(g_0)$ is conjugate to
$\varphi(g)$ in $R$, we can conclude that $\varphi(f) \notin \varphi(g)^R=\varphi(g^G)$ in $R$.

To finish proving Case 2, it remains to consider

\underline{\textit{Sub-case 2.2:}} $f \in A^G$. Set $m:=0$ and denote $f_0=g_0 \in G$.
Applying Lemma \ref{lem:claim_B} to $g_0$ and $f_0$,
we can find a homomorphism $\tilde \psi$ from $G$ to a special $HNN$-extension $P$
of a finite group $Q$ such that $\tilde\psi(g_0)=s^{\e_1}\tilde\psi(x_1)  \dots s^{\e_n}\tilde\psi(x_{n})$
is cyclically reduced in $P$. Since $n\ge 1$, by Lemma \ref{lem:Collins} we have
$\tilde\psi(g_0) \notin \tilde\psi(A)^P=\tilde\psi(A^G)$ in $P$,
hence $\tilde\psi(f)\notin \tilde\psi(g_0)^{P}=\tilde\psi(g)^{P}$ in $P$.
Arguing as above, we can find a finite quotient $R$ of $P$ (and, hence, of $G$) such that
the images of $f$ and $g$ are not conjugate in $R$.

We can now conclude that the conjugacy class $g^G$ is closed in $\PT(G)$. Thus Case 2 is considered.
This finishes the proof of Lemma \ref{lem:main-CS}.
\end{proof}

\begin{proof}[Proof of Lemma \ref{lem:main-CC}] Take any element $g \in G$ and
any finite index normal subgroup $K\lhd G$.
As in  Lemma \ref{lem:main-CS}, the proof splits into two main cases.

\fbox{\textit{Case 1:}} $B \neq G$.

Choose a maximal special subgroup $A$ of $G$
containing $B$. Then $A$ is a right angled Artin group of rank $r-1<r$,
and $G$ is the special HNN-extension \eqref{eq:def_of_A-main} of $A$
with respect to a certain special subgroup $H\le A$
(by Remark \ref{rem:RAAG_spec_HNN_over_max_sbgp}).
Define the finite index normal subgroup $K_1$ of $A$ by $K_1:=K \cap A$.

\underline{\textit{Sub-case 1.1:}} $g \in A$. Then, according to the induction hypothesis,
the pair $(B,g)$ satisfies the Centralizer Condition ${\rm CC}_A$ in $A$, hence
there exists $L_1 \lhd A$ such that $|A:L_1|<\infty$, $L_1 \le K_1$, and the natural epimorphism
$\psi: A \to Q:=A/L_1$ satisfies
\begin{equation}\label{eq:cond_CC_A}
C_{\psi(B)}(\psi(g)) \subseteq \psi \left( C_B(g) K_1\right)~\mbox{ in }Q.
\end{equation}

Let $\rho_A: G \to A$ be the canonical retraction and set $L:=\rho_A^{-1}(L_1) \cap K$.
Then  $L \lhd G$, $|G:L|<\infty$, $L \le K$ and $\rho_A(L)=L_1\le K_1$
(since $K_1 = K \cap A \subseteq \rho_A(K)$).
Let $\varphi:G \to R:=G/L$ be the natural epimorphism.
Observe that $\ker(\psi)=\ker(\varphi) \cap A$ in $G$.
Indeed, $\ker(\psi)=L_1$, $\ker(\varphi)=L$, and
$L_1\subseteq \rho_A^{-1}(L_1) \cap K \cap A=L\cap A$,
$L\cap A \subseteq \rho_A(L)= L_1$.

Therefore, without loss
of generality, we can assume that $Q \le R$, and
the restriction of $\varphi$ to $A$ coincides with $\psi$.
Then we have
$\psi(K_1)=\varphi(K_1)\subseteq \varphi(K)$ in $R$. Since $g \in A$ and $B \le A$,
\eqref{eq:cond_CC_A} implies that
$$C_{\varphi(B)}(\varphi(g))= C_{\psi(B)}(\psi(g)) \subseteq \psi \left( C_B(g)\right) \psi(K_1)
\subseteq \varphi \left( C_B(g)\right) \varphi(K)\,\mbox{ in }R,$$ which shows that the pair
$(B,g)$ has {\ccg} in Sub-case 1.1.

\underline{\textit{Sub-case 1.2:}} $g \in G\setminus A$. Then the element $g$ can be represented
as a reduced product $x_0t^{\e_1}x_1 t^{\e_2} \dots t^{\e_n}x_{n}$ in $G$, with $n \ge 1$.
Therefore we can find the groups $Q$, $P$ and the maps $\psi:A \to Q$,
$\tilde\psi: G \to P$ from the claim Lemma \ref{lem:claim_C}, so that all of the
assertions of that lemma hold.

Note that the subgroup
$\tilde\psi(B) \cap \tilde\psi(K) \le Q \le P$ is finite, therefore, since $P$ is residually
finite (as a virtually free group), the finite set $\tilde\psi(g)^{\tilde\psi(B) \cap \tilde\psi(K)}$
is separable in $P$. Consequently, by Lemma \ref{lem:for_a_given_K_sep_cc->CCG},
there exists a finite group $R$ and an epimorphism $\xi: P \to R$ such that $\ker(\xi) \le \tilde\psi(K)$
and
$$C_{\xi\left(\tilde\psi(B)\right)}\left(\xi\left(\tilde\psi(g)\right)\right)
\subseteq \xi \left( C_{\tilde\psi(B)}\left(\tilde\psi(g)\right) \tilde\psi(K)\right) ~\mbox{ in } R.$$

Define the epimorphism $\varphi:G \to R$ by $\varphi:=\xi \circ \tilde\psi$, and
observe that $\ker(\varphi) = \tilde\psi^{-1} \left(\ker(\xi)\right) \subseteq
\tilde\psi^{-1} \left( \tilde\psi(K) \right)=K \ker(\tilde\psi)$. But $\ker(\tilde\psi) \le K$
according to the second assertion of  Lemma \ref{lem:claim_C}, hence
$L:=\ker(\varphi) \le K$ in $G$.

Finally, recalling the first assertion of  Lemma \ref{lem:claim_C}, we see that
\begin{multline*} C_{\varphi(B)}\left(\varphi(g)\right)= C_{\xi\left(\tilde\psi(B)\right)}\left(\xi\left(\tilde\psi(g)\right)\right)\subseteq
\xi \left( C_{\tilde\psi(B)}\left(\tilde\psi(g)\right) \tilde\psi(K)\right) \\ \subseteq
\xi \left( \tilde\psi \left( C_B(g) K\right)\tilde\psi(K)\right) = \varphi(C_B(g)K)~\mbox{ in } R.
\end{multline*}
Thus we have shown that the pair $(B,g)$ has {\ccg} in Sub-case 1.2.
Therefore, $B$ has {\ccg} in Case 1.

\fbox{\textit{Case 2:}} $B = G$.

The pair $(G,1)$ evidently satisfies {\ccg}, therefore we can assume that $g \neq 1$.
In this case, by Lemma \ref{lem:el_in_RAAG_non_conj_into_max_spec},
there exists a maximal special subgroup $A$ of $G$ such that $g \notin A^G$. By
Remark \ref{rem:RAAG_spec_HNN_over_max_sbgp}, $G$ splits
as a special HNN-extension \eqref{eq:def_of_A-main} of $A$ with respect to a certain special subgroup
$H \le A$. Obviously, there exists $z \in G$ such that
$g=zg_0 z^{-1}$ in $G$, for some cyclically reduced element
$g_0=t^{\e_1}x_1 \dots t^{\e_n}x_{n}$, where $n \ge 1$ because $g \notin A^G$.

Now we apply Lemma \ref{lem:claim_D} to find $Q$, $P$, $\psi:A \to Q$ and $\tilde\psi:G \to P$
from its claim, so that $\ker(\tilde\psi) \le K$ and
$C_P\left(\tilde\psi(g_0)\right) \subseteq \tilde\psi(C_G(g_0)K)$ in $P$. Note that $P$ is virtually
free (being an HNN-extension of a finite group $Q$), hence every subgroup of $P$ is virtually
free as well. Therefore, by Lemma \ref{lem:virt_free-cs}, $P$ is hereditarily conjugacy
separable, and, thus, by Proposition \ref{prop:her_c_s-CC},
$P$ satisfies the Centralizer Condition \cc.

Consequently, there exists a finite group $R$ and an epimorphism $\xi: P \to R$ such that $\ker(\xi) \le \tilde\psi(K)$
and
$$C_{R}\left(\xi\left(\tilde\psi(g_0)\right)\right)
\subseteq \xi \left( C_{P}\left(\tilde\psi(g_0)\right) \tilde\psi(K)\right) ~\mbox{ in } R.$$

Defining the epimorphism $\varphi:G \to R$ by $\varphi:=\xi \circ \tilde\psi$, and arguing in
the same manner as in Sub-case 1.2, we can show that $L:=\ker(\varphi) \le K$ and
$C_{R}\left(\varphi(g_0)\right) \subseteq \varphi(C_G(g_0)K)$ in $R$. Conjugating both sides of the
latter inclusion by $\varphi(z)$ (and recalling that $g=zg_0z^{-1}$ in $G$), we achieve
$C_{R}\left(\varphi(g)\right) \subseteq \varphi(C_G(g)K)$.

Hence $B=G$ has {\ccg} in Case 2,
and Lemma \ref{lem:main-CC} is proved.
\end{proof}

Thus Lemmas  \ref{lem:main-CS} and  \ref{lem:main-CC} have been proved when $\rank(G)=r$.
Therefore, by induction they are true for all $r \in \N \cup \{0\}$, and we are
ready to prove the main result of this paper.

\begin{proof}[Proof of Theorem \ref{thm:RAAG-main}]
Let $G$ be a right angled Artin group associated to a finite simplicial graph $\Gamma$.
Then for every
$g \in G$, the conjugacy class $g^G$ is separable in $G$ by
Lemma  \ref{lem:main-CS}. And Lemma \ref{lem:main-CC} tells us that $G$ satisfies
the Centralizer Condition \cc. Therefore, by Proposition \ref{prop:her_c_s-CC}, $G$ is hereditarily
conjugacy separable.
\end{proof}


\section{Applications to separability properties}\label{sec:appl_sep}
The first two applications that we mention do not directly follow from the statement of
Theorem  \ref{thm:RAAG-main}, but are consequences of its proof.

\begin{cor}\label{cor:RAAG-double_coset_sep} Let $A$ and $B$ be conjugates of special subgroups
of a right angled Artin group $G$. Then for any element $x \in G$,
the double coset $AxB$ is separable in $G$.
\end{cor}

\begin{proof} Evidently, it is enough to consider the case when $A$ and $B$ are special
subgroups of $G$. Then $A$ and $B$ are retracts of $G$ and the corresponding retractions
commute (Remark \ref{rem:spec_retr-comm}). By Remark \ref{rem:inter_special-sbgps},
$A\cap B$ is also a special subgroup of $G$, hence Lemma  \ref{lem:main-CS} implies that
the subset $\alpha^{A \cap B}$ is separable in $G$ for every $\alpha \in G$. Therefore,
$AxB$ is separable in $G$ by Lemma \ref{lem:double_coset_sep}.
\end{proof}

In the case when $x=1$ and
$A$,$B$ are special subgroups of $G$ (not conjugates of them),
Corollary \ref{cor:RAAG-double_coset_sep} was proved by
Haglund and Wise in \cite[Cor. 9.4]{H-W_1} using
different arguments, based on Niblo's
criterion for separability of double cosets (see \cite{Niblo}). Unfortunately, in general
this criterion cannot be used to prove separability of double cosets of the form $AxB$ if $x \in G$
is an arbitrary element (because the retractions onto $A$ and $xBx^{-1}$ may no longer commute).

Similarly, using Lemmas \ref{lem:main-CS} and \ref{lem:main-CC}
together with Lemmas \ref{lem:inter_conj_to_spec_sbgps} and \ref{lem:centr-coset_sep},
we can obtain the following:

\begin{cor}\label{cor:RAAG-centr_double_coset_sep} Suppose that
$A$ and $D$ are conjugates of special subgroups in a right angled Artin group $G$, and
$g \in G$ is an arbitrary element. Then the double coset $C_A(g)D$ is separable in $G$.
\end{cor}

The rest of applications in this section discuss conjugacy separability of various groups.
Let us start with the following well-known observation:

\begin{lemma}\label{lem:retr_in_c_s} If $H$ is a retract of a conjugacy separable group $G$,
then $H$ is conjugacy separable.
\end{lemma}

\begin{proof}
Indeed, let $\rho_H \in End(G)$ be a retraction of $G$ onto $H$. Suppose that
$x,y \in H$ and $y \notin x^H$ in $H$. If there existed $g \in G$ such that $y=gxg^{-1}$ in $G$,
then we would have $y=\rho_H(y)=\rho_H(g) \rho_H(x)\rho_H(g)^{-1}=\rho_H(g) x\rho_H(g)^{-1}$ in $H$,
contradicting to our assumption. Therefore, $y \notin x^G$, and since $G$ is conjugacy separable,
there is a finite group $R$ and a homomorphism $\varphi:G \to R$ such that $\varphi(y) \notin \varphi(x)^R$.
Let $Q:=\varphi(H)$ and $\psi:H \to Q \le R$ be the restriction of $\varphi$ to $H$. By construction,
we have that $\psi(y) \notin \psi(x)^Q$ in $Q$. Therefore $H$ is conjugacy separable.
\end{proof}

\begin{rem}\label{rem:f_i_sbgp_virt_retr} If $F$ is a finite index subgroup in a virtual retract $H$
of a group $G$, then $F$ itself is a virtual retract of $G$.
\end{rem}

Indeed, let $K \le G$ be a finite index subgroup containing $H$,
and let $\rho_H$ be a retraction of $K$ onto $H$. Then $M:=\rho_H^{-1}(F) \le K$ has finite index in $K$,
and, hence, in $G$.
Evidently the restriction of $\rho_H$ to $M$ is a retraction of $M$ onto $F$.

Combining Remark \ref{rem:f_i_sbgp_virt_retr} with Lemma  \ref{lem:retr_in_c_s} we obtain the following
statement (cf. \cite[Thm. 3.4]{Chag-Zal}):

\begin{lemma}\label{lem:virt_retract_h_c_s} A virtual retract of a hereditarily conjugacy separable
group is hereditarily conjugacy separable itself.
\end{lemma}

Next comes a classical fact about conjugacy separable groups:
\begin{lemma}\label{lem:UR-c_s} Suppose $G$ is a group satisfying the Unique Root property.
If $G$ contains a conjugacy separable subgroup $H$ of finite index, then $G$ is conjugacy separable.
\end{lemma}

\begin{proof} Consider any element $x \in G$. We need to show that the conjugacy class
$x^G$ is separable in $G$.

Assume, first, that $x \in H$. Then $x^H$ is closed in $\PT(H)$, and since $|G:H|<\infty$,
it is also closed in $\PT(G)$.
Choose $z_1,\dots,z_k \in G$ so that $G=\bigsqcup_{i=1}^k z_i H$. Then
$x^G= \bigcup_{i=1}^k z_i x^H z_i^{-1}$ is a finite union of closed sets in $\PT(G)$,
hence $x^G$ is separable in $G$.

Now, if $x \in G$ is an arbitrary element, then there is $n \in \N$ such that
$g:=x^n \in H$, and, as we have
shown above, $g^G$ is separable in $G$. Take any $y \in G \setminus x^G$. Since $G$ has the
Unique Root property, we see that $y^n \notin g^G$. Hence, there exists a finite index normal subgroup
$N\lhd G$ such that $y^n \notin g^G N$ in $G$. Consequently, $y \notin x^GN$
(because the inclusion $y \in x^G N$ implies the inclusion $y^n \in (x^n)^GN$).
Thus $G$ is conjugacy separable.
\end{proof}

It is easy to see that Lemma \ref{lem:UR-c_s} can be generalized as follows:
\begin{lemma}\label{lem:URP-h_c_s} If a group $G$ has the Unique Root property and
contains a hereditarily conjugacy separable subgroup of finite index, then
$G$ is itself hereditarily conjugacy separable.
\end{lemma}

Note that the assumption of Lemma \ref{lem:UR-c_s} demanding $G$ to satisfy the Unique Root property
is important: in \cite{Gor} A. Goryaga  constructed an example of a finitely generated group $G$ which
is not conjugacy separable, but contains a conjugacy separable subgroup of index $2$.

\begin{cor}\label{cor:UR_AVR->h_c_s} If a group $G \in \avr$ has the Unique Root property,
then $G$ is hereditarily conjugacy separable.
\end{cor}

\begin{proof} Let  $K \le G$ be a subgroup of finite index. By definition, $G$ contains
a finite index subgroup $H$ which is a virtual retract of some right angled Artin group $A$.
Since $|H:(K\cap H)|\le |G:K|\le \infty$, $K\cap H$ is a virtual retract of $A$ by
Remark \ref{rem:f_i_sbgp_virt_retr}. But the index $|K:(K \cap H)|$ is also finite,
hence $K \in \avr$.

Now, Theorem \ref{thm:RAAG-main} and Lemma \ref{lem:virt_retract_h_c_s} imply
that $K \cap H$ is conjugacy separable. Hence $K$ is conjugacy separable by Lemma \ref{lem:UR-c_s}.
Thus $G$ is hereditarily conjugacy separable.
\end{proof}

Recall that two groups $G_1$ and $G_2$ are said to be \textit{commensurable},
if there exist finite index subgroups $H_1 \le G_1$ and $H_2 \le G_2$
such that $H_1$ is isomorphic to $H_2$.
The proof of Corollary \ref{cor:UR_AVR->h_c_s} allows to conclude that the class $\avr$
is closed under passing to subgroups of finite index. Therefore we can make

\begin{rem} \label{rem:AVR_f_i_sbgps} If $G_1$ is commensurable to $G_2$ and $G_1 \in \avr$,
then $G_2 \in \avr$.
\end{rem}

As we observed in Lemma \ref{lem:RAAG_URP}, right angled Artin groups have the Unique Root property.
Another well-known class of groups with this property is the class of torsion-free word hyperbolic groups.

\begin{lemma}\label{lem:hyp-URP} Torsion-free word hyperbolic groups have the Unique Root property.
\end{lemma}

\begin{proof} Let $G$ be a torsion-free word hyperbolic group.
Suppose that $x^n=y^n$ in $G$ for some $x,y \in G$ and $n \in \N$.
Since $G$ is torsion-free, we can assume that the orders of $x$ and $y$ are infinite.
It is well-known that every element $g \in G$, of infinite order,
belongs to a unique maximal virtually cyclic subgroup $E(g) \le G$ (see, for instance,
\cite[Lemma 1.16]{Olsh-G-sbgps}).

Note that the element $g:=x^n \in G$ has infinite
order and $g \in E(x) \cap E(y)$. Therefore, $E(x)=E(y)$, thus $y \in E(x)$.
But since $G$ is torsion-free, the virtually cyclic subgroup $E(x) \le G$ must be
cyclic. That is, there exists $z \in G$ such that $x=z^k$ and $y=z^l$ for some $k,l \in \Z$.
Obviously, the equality $x^n=y^n$ implies that $k=l$. Thus $x=y$, and, hence, $G$ enjoys
the Unique Root property.
\end{proof}

Combining Lemma \ref{lem:hyp-URP}  with Corollary \ref{cor:UR_AVR->h_c_s} we obtain
\begin{cor}\label{cor:hyp_t-f_AVR->h_c_s} If $G \in \avr$ is a torsion-free
word hyperbolic group, then $G$ is hereditarily conjugacy separable.
\end{cor}

Using Osin's results from \cite{Osin-elem-sbgps_rel_hyp}, it is not difficult to
generalize Lemma \ref{lem:hyp-URP} as follows: if a group $G$ is torsion-free
and hyperbolic relative to a collection of proper subgroups, each of which has the Unique Root property,
then $G$ has the Unique Root property. As a result, Corollary \ref{cor:hyp_t-f_AVR->h_c_s}
can also be restated for this kind of relatively hyperbolic groups.

We can also establish Corollary \ref{cor:t_f_sbgp_hyp_Cox},  mentioned in Section \ref{sec:conseq_main_res}.

\begin{proof}[Proof of Corollary \ref{cor:t_f_sbgp_hyp_Cox}]
Since $H$ has finite index in $G$, it is also word hyperbolic (\cite{Gromov}),
thus, according to Lemma  \ref{lem:hyp-URP}, $H$ enjoys the Unique Root property.

By Corollary \ref{cor:Cox-sbgp-c_s}, $G$ has a hereditarily conjugacy separable
subgroup $F \le G$ of finite index. Define $K:=H \cap F \le G$. Then $K$ will be
hereditarily conjugacy separable (because $|F:K|<\infty$).
And since $|H:K|<\infty$, Lemma \ref{lem:URP-h_c_s} implies that $H$ is
hereditarily conjugacy separable.
\end{proof}


\section{Applications to outer automorphism groups}\label{sec:appl_out}

We have already discussed a few applications of Theorem \ref{thm:RAAG-main} to
outer automorphism groups in Section \ref{sec:conseq_main_res}. This section's aim is to prove
Theorem \ref{thm:rel_hyp_gp_a_v_r}.

We refer the reader to Osin's monograph \cite{Osin-rel_hyp}
for the definition and basic properties of relatively hyperbolic groups. All relatively hyperbolic
groups that we consider here are hyperbolic relative to families of \textit{proper}
subgroups. In the  sense of B. Farb \cite{Farb}, this would correspond to \textit{weak relative
hyperbolicity} together with the \textit{Bounded Coset Penetration Condition} (the equivalence
of Osin's and Farb's definitions for finitely generated groups is proved
in \cite[Thm. 6.10]{Osin-rel_hyp}).

The following lemma is not difficult to prove but its statement is very useful
(see, for example, \cite[Lemma 5.4]{G-L}).

\begin{lemma}\label{lem:out_norm_sub}
Suppose that $G$ is a finitely generated group and $N$ is a centerless normal
subgroup of finite index in $G$. Then some finite index subgroup of
$Out (G)$ is isomorphic to a quotient of a subgroup of $Out (N)$ by
a finite normal subgroup. In particular, if $Out(N)$ is residually finite,
then $Out (G)$ is residually finite.
\end{lemma}

Recall, that a group $G$ is called \textit{elementary}, if it contains a cyclic subgroup of finite index.

\begin{lemma}\label{lem:rel_hyp-fin_center} If $G$ is a non-elementary relatively hyperbolic
group, then its center $Z(G)$ is finite.
\end{lemma}

\begin{proof} Suppose that $G$ is hyperbolic relative to a family of proper non-trivial subgroups
$\{H_{\lambda}\}_{\lambda \in \Lambda}$.

First, if $|\Lambda|=\infty$, then $G$ splits as a non-trivial free product by \cite[Thm. 2.44]{Osin-rel_hyp},
and, thus, $Z(G)=\{1\}$. If the set $\Lambda$ is finite and each parabolic subgroup
$H_\lambda$, $\lambda \in \Lambda$, is finite, then $G$ is word hyperbolic (in the sense of Gromov) by
\cite[Cor. 2.41]{Osin-rel_hyp}.
And it is well-known that the center of a non-elementary word hyperbolic group is finite.

Therefore we can assume that there is some $\mu \in \Lambda$ such that $|H_\mu|=\infty$.
A theorem of Osin \cite[Thm. 1.4]{Osin-rel_hyp} claims that the intersection
$H_\mu \cap g H_\mu g^{-1}$ is finite for every $g \in G\setminus H_\mu$.
If $z \in Z(G)$, then $H_\mu=H_\mu \cap z H_\mu z^{-1}$ is infinite, hence $z \in H_\mu$,
i.e., $Z(G) \subseteq H_\mu$. On the other hand, there exists $g \in G \setminus H_\mu$
because $H_\mu$ is a proper subgroup of $G$. By Osin's theorem, 
$H_\mu \cap g H_\mu g^{-1}$ is finite. And since $Z(G)\subseteq  g H_\mu g^{-1}$,
we see that $Z(G) \subseteq H_\mu \cap g H_\mu g^{-1}$ must be finite.
\end{proof}

The proof of Theorem  \ref{thm:rel_hyp_gp_a_v_r} will use the following fact, established
in \cite[Cor. 1.4]{norm-aut}:
\begin{lemma}\label{lem:rel_hyp-pi=inn} If $G$ is a torsion-free non-elementary relatively hyperbolic group,
then $Aut_{pi}(G)=Inn(G)$.
\end{lemma}

\begin{proof}[Proof of Theorem \ref{thm:rel_hyp_gp_a_v_r}] Let $G$ be a relatively hyperbolic group
from the class $\avr$. If $G$ is virtually cyclic, then $Out(G)$ is finite (cf. \cite[Lemma 6.6]{norm-aut}).
Hence we can further suppose that $G$ is non-elementary. By the assumptions, $G$ contains
a finite index subgroup $N \in \vr$, and, in view of Remark \ref{rem:f_i_sbgp_virt_retr}, we can assume
that $N\lhd G$.

Note that $N$ is finitely generated (and even finitely presented),
as a virtual retract of a finitely presented group.
And since $N$ has finite index in $G$, it is non-elementary
and relatively hyperbolic. The latter is an immediate consequence of
Bowditch's definition of relatively hyperbolic groups given in \cite[Def. 2]{Bowditch}
(which is equivalent to Osin's definition, as shown in \cite[Thm. 6.10]{Osin-rel_hyp});
this also follows
from the powerful result of C. Dru\c tu \cite[Thm. 1.2]{Drutu}, which claims that relative
hyperbolicity is invariant under quasi-isometries.

By construction, $N$ is a virtual retract of some right angled Artin group $A$. And since $A$ is
torsion-free, $N \le A$ is torsion-free as well. Therefore, according to Lemma \ref{lem:rel_hyp-fin_center},
$Z(N)=\{1\}$. The group $N$ is finitely generated and
conjugacy separable by Corollary \ref{cor:virt_retract_RAAG-h_c_s}, and
$Aut_{pi}(N)=Inn(N)$ by Lemma  \ref{lem:rel_hyp-pi=inn}. Hence we can apply Grossman's
theorem \cite[Thm. 1]{Grossman} to conclude that $Out(N)$ is residually finite. Consequently,
$Out(G)$ is residually finite by Lemma \ref{lem:out_norm_sub}.
\end{proof}


\section{Applications to the conjugacy problem}\label{sec:appl_conj_probl}
As it was shown by Mal'cev \cite{Malcev} and Mostowskii \cite{Mostow}, a finitely presented conjugacy
separable group has solvable conjugacy problem. This result can be generalized as follows:

\begin{lemma} \label{lem:dec_conj_probl} Suppose that $H$ is a finitely generated subgroup of
a finitely presented group $G$, such that for every $h \in H$ the $H$-conjugacy class $h^H$
is separable in $G$. Then the conjugacy problem for $H$ is solvable.
\end{lemma}

\begin{proof} Without loss of generality we can assume that $G= \langle X \,\| \, R \rangle$,
for some finite set $X$ and a finite set of words $R$ over the alphabet $X^{\pm 1}$,
and $H$ is generated by a finite subset $Y$ of $X$.
Let $F(X)$ denote the free group on the set $X$ and let $F(Y)$ be the subgroup of $F(X)$ generated by $Y$.
Then the identity map on $X$ gives rise to the epimorphism $\theta: F(X) \to G$, such that
$\ker \theta= N$ is the normal closure of $R$ in $F(X)$.

Since $N$ is the normal closure of only finitely many words in $F(X)$ and $Y$ is finite, a standard argument
(cf. \cite{Mostow}) shows that there is a partial algorithm  $\mathfrak A$, which, given two reduced
words $U,W \in F(Y)$, terminates if and only
if $U \in W^{F(Y)}N$ (i.e., if $\theta(U) \in \theta(W)^H$ in $G$). The algorithm $\mathfrak A$ lists
every word from $W^{F(Y)}N$ in $F(X)$, freely reduces it and compares it with $U$; it stops once
it finds a word in $W^{F(Y)}N$ that is equal to $U$ in $F(X)$.

On the other hand, as $G \cong F(X)/N$ is finitely presented, there is an effective procedure listing all
homomorphisms $\psi$ from $F(X)$ to all finite groups $Q$, satisfying $N \subseteq \ker \psi$
(see \cite{Mostow}). Given such a homomorphism $\psi$ and any two reduced words $U,W \in F(Y)$,
one can decide in finitely many steps whether
or not $\psi(U) \in \psi(W)^{\psi(F(Y))}$ in $Q$, because $\psi(F(Y))=\langle \psi(Y) \rangle$ and $Y$
is finite.

For any $U,W \in F(Y)$ denote $u:=\theta(U)$ and $w:=\theta(W)$.
If $u \notin w^H$, the separability of $w^H$ in $G$
implies the existence of a finite group $Q$ and
a homomorphism $\phi:G \to Q$ such that $\phi(u) \notin \phi(w^H)=\phi(w)^{\phi(H)}$ in $Q$.
Thus the homomorphism $\psi:=\phi \circ \theta: F(X) \to Q$ satisfies $N \subseteq \ker \psi$ and
$\psi(U) \notin \psi(W)^{\psi(F(Y))}$ in $Q$. And, of course, the existence of such a homomorphism $\psi$
tells us that $u \notin w^H$ in $G$.

Hence, there is a partial algorithm $\mathfrak B$, which takes on input two words
$U,W \in F(Y)$ and terminates if and only if $\theta(U) \in \theta(W)^H$ in $G$. This algorithm
goes through all the homomorphisms $\psi$ from $F(X)$ to finite groups $Q$ with
$N \subseteq \ker \psi$, and stops when it finds one such that
$\psi(U) \notin \psi(W)^{\psi(F(Y))}$ in $Q$.

The solution of the conjugacy problem for $H$  amounts to taking on input
two reduced words $U,V \in F(Y)$ and running the two partial algorithms
$\mathfrak A$ and $\mathfrak B$ simultaneously. One (and only one) of these two algorithms will eventually
terminate, thus answering whether or not $\theta(U)$ is conjugate to $\theta(W)$ in $H$.
\end{proof}

\begin{cor}\label{cor:h_c_s->c_s_and_conj_prob} Let $G$ be a hereditarily conjugacy separable group.
Suppose that $H$ is a subgroup of $G$ such that the double coset $C_G(h)H$ is separable in $G$ for every
$h \in H$. Then $H$ is conjugacy separable. If, in addition, $G$ is finitely presented and $H$ is finitely
generated, then $H$ has solvable conjugacy problem.
\end{cor}

\begin{proof} The first claim is a direct consequence of Proposition \ref{prop:her_c_s-CC} and
Corollary \ref{cor:CC->_c_s_for_sbgps}. They also imply that $h^H$ is separable in
$G$ for every $h \in H$. Therefore, the second claim follows from Lemma \ref{lem:dec_conj_probl}.
\end{proof}

We are now in a position to prove Theorem \ref{thm:norm_sbgps_RAAG-conj_probl}, announced in
Section \ref{sec:conseq_main_res}.

\begin{proof}[Proof of Theorem \ref{thm:norm_sbgps_RAAG-conj_probl}]
In \cite{Servat} Servatius
completely described centralizers of elements in right angled Artin groups. In particular, it follows
from his description that $C_G(h)$ is finitely generated for every $h \in G$.

Let $\psi:G \to G/N$ be the natural epimorphism and consider any $h \in N$.
Then $E:=\psi(C_G(h))$ is a finitely
generated subgroup of $G/N$, hence $E$ is closed in  $\PT(G/N)$ by the assumptions.
Since the map $\psi$ is continuous (when $G$ and $G/N$ are considered as topological groups
equipped with their profinite topologies), we can conclude that $C_G(h)N=\psi^{-1}(E)$
is closed in $\PT(G)$ for every $h \in N$.

Therefore the claim of
Theorem \ref{thm:norm_sbgps_RAAG-conj_probl} follows from Theorem \ref{thm:RAAG-main} and
Corollary~\ref{cor:h_c_s->c_s_and_conj_prob}.
\end{proof}

\begin{cor} \label{cor:norm_sbgps_with_polyc_quot} Let $N$ be a finitely generated normal subgroup
of a right angled Artin group $G$ such that the quotient $G/N$ is virtually polycyclic.
Then every finite index subgroup $K$ of $N$ is conjugacy separable and has solvable conjugacy problem.
\end{cor}

\begin{proof} Since $N$ is finitely generated, $K$ contains a
characteristic subgroup $L$ of $N$ with $|N:L|<\infty$. And since $N \lhd G$, we can conclude that
$L \lhd G$, and the group $G/L$ is an extension of the finite group $N/L$ by the virtually polycyclic
group $G/N$.  An easy induction
on the length of the series with cyclic quotients shows that every finite-by-polycyclic group
is polycyclic-by-finite. Thus $G/L$ is virtually polycyclic,
hence it is subgroup separable -- see \cite[Ex. 11 in Chapter 1.C]{Segal}.

Arguing as in the proof of Theorem \ref{thm:norm_sbgps_RAAG-conj_probl},
we see that the double coset
$C_G(h)L$ is separable in $G$ for each $h \in G$. But $K= \bigcup_{i=1}^k L x_i$ for some
$k \in \N$ and $x_1,\dots,x_k \in K$. Therefore $C_G(h)K=\bigcup_{i=1}^k C_G(h)L x_i$
is separable in $G$ (as a finite union of separable subsets) for all $h \in K$.

Note that $K$ is finitely generated as a finite index subgroup of $N$, hence $K$
is conjugacy separable and has solvable conjugacy problem by Theorem \ref{thm:RAAG-main} and
Corollary~\ref{cor:h_c_s->c_s_and_conj_prob}.
\end{proof}


\section{Appendix: the Centralizer Condition in profinite terms}
Our intention here is to prove that for residually finite groups
the condition {\cc} from the Definition \ref{df:CC} is equivalent to the condition
\eqref{eq:CC-profinite} of Chagas and Zalesskii. We refer the reader to the book
\cite{Rib-Zal} for the background on profinite completions.

\begin{prop} \label{prop:equiv_def_CC} Let $H$ be a subgroup of a  residually finite group $G$ and
let $g \in G$. The following are equivalent:
\begin{itemize}
	\item[1)] the pair $(H,g)$ satisfies the condition {\ccg} from Definition \ref{df:CCG};
	\item[2)] $\overline{C_H(g)}=C_{\overline{H}}(g)$, where $\overline{H}\le \widehat{G}$ is the closure of
						$H$ in the profinite completion $\widehat G$ of $G$.
\end{itemize}
\end{prop}

\begin{proof} The profinite completion $\widehat G$ of $G$ is the inverse limit of finite quotients of $G$.
There is a canonical embedding of $\widehat G$ into the Cartesian product $\prod_{N \in \mathcal{N}} G/N$,
where $\mathcal N$ is the set of all finite index normal subgroups of $G$.
Thus $\widehat G$ can be equipped with the product topology, making it a compact 
topological group (each finite group $G/N$ is endowed with the discrete topology).

For each $N \in \mathcal N$ let $\psi_N$ denote the natural epimorphism from $G$ to $G/N$.
Then the map $\psi:G \to \widehat G$, defined by
$\psi(x) :=(\psi_N(x))_{N\in \mathcal{N}}$ for every $x \in G$, is a homomorphism. And since $G$
is residually finite, $\psi$ is injective. Therefore we can assume that $G \le \widehat G$, and
so the condition 2) makes sense. Every homomorphism $\psi_M$, $M \in \mathcal N$,
can be uniquely extended to a continuous homomorphism
$\hat{\psi}_M: \widehat G \to G/M$ ($\hat{\psi}_M$ can also be regarded as a restriction to $\widehat G$
of the canonical projection from $\prod_{N \in \mathcal{N}} G/N$ to $G/M$).

First, suppose that the pair $(H,g)$ satisfies \ccg.
Consider any $h \in \overline{H}$ such that $h \notin \overline{C_{H}(g)}$.
Then there exists $K \in \mathcal N$ such that $\hat{\psi}_K(h) \notin \psi_K(C_H(g))$. Hence,
by {\ccg}, there is $L \in \mathcal N$ satisfying $L \le K$ and
$\psi_L^{-1} \left(C_{\psi_L(H)}(\psi_L(g))\right) \subseteq C_H(g) K = \psi_K^{-1}(\psi_K(C_H(g)))$.
Therefore $\hat\psi_K$ factors through $\hat\psi_L$, hence $\hat\psi_L(h)  \notin \psi_L(C_H(g)K)$.
Consequently, $\hat\psi_L(h)  \notin C_{\psi_L(H)}(\psi_L(g))$, and so $h \notin C_{\overline{H}}(g)$.
Thus we established the inclusion $C_{\overline{H}}(g) \subseteq \overline{C_H(g)}$.
Since the inverse inclusion is evident, we have proved that 1) implies 2).

Now, let us assume that the condition 2) holds. Choose any $K \in \mathcal N$ and denote
$\mathcal{L}:=\{L \in \mathcal{N}~|~L \le K\}$. Arguing by contradiction,
suppose that for each $L \in \mathcal L$ there is $x_L \in  H$ such that
$\psi_L(x_L) \in C_{\psi_L(H)}(\psi_L(g)) \setminus \left(\psi_L(C_H(g) K)\right)$. Note that $\mathcal L$ is
a directed set (if $L_1, L_2 \in \mathcal L$ then $L_1 \preceq L_2$ if and only if $L_2 \subseteq L_1$),
hence $(x_L)_{L \in \mathcal L}$ is a net in $\widehat G$. Since $\widehat G$ is compact, this net
has a cluster point $h \in \overline{H} \le \widehat G$.

Consider any $N \in \mathcal N$ and set $L=N\cap K \in \mathcal L$. Then, according to the definition
of the topology on $\widehat G$, there is $M \in \mathcal L$
such that $M \subseteq L$ and $\psi_L(x_{M})=\hat\psi_L(h)$. By construction,
$\psi_L(x_M) \in C_{\psi_L(H)}(\psi_L(g))$, hence
$\hat\psi_L(h) \in C_{\psi_L(H)}(\psi_L(g))$,
implying that $\hat\psi_N(h) \in C_{\psi_N(H)}(\psi_N(g))$ because $L \le N$. Since the latter
holds for every $N \in \mathcal N$, we can conclude that $h \in C_{\overline{H}}(g)$.

On the other hand, since $h$ is a cluster point of the net $(x_L)_{L \in \mathcal L}$ and $K \in \mathcal L$,
there exists $M \in \mathcal L$ such that $\psi_K(x_M) = \hat \psi_K (h)$. But since $M \le K$ we have
$x_M \notin C_H(g) KM=C_H(g)K=\psi_K^{-1}(\psi_K(C_H(g)))$. Thus $\hat \psi_K (h)=\psi_K(x_M) \notin
\psi_K(C_H(g))$, which implies that $h \notin \overline{C_H(g)}$.

Thus we found an element $h \in C_{\overline{H}}(g) \setminus \overline{C_H(g)}$,
contradicting to the condition 2). Consequently, 2) implies 1).
\end{proof}

Proposition \ref{prop:equiv_def_CC} implies that for residually finite groups
the Centralizer Condition {\cc} from Definition \ref{df:CC}
is equivalent to the condition \eqref{eq:CC-profinite} introduced by Chagas and Zalesskii in \cite{Chag-Zal}:

\begin{cor}\label{cor:CC<->Chag-Zal} A is residually finite group $G$ satisfies {\cc}
if and only if $\overline{C_G(g)}=C_{\widehat{G}}(g)$ for every $g \in G$.
\end{cor}

It is well known that conjugacy separability of a residually finite group $G$ is equivalent
to the condition
\begin{equation}\label{eq:c_s-prof}
g^{\widehat G} \cap G=g^G~\mbox{in $\widehat G$, for all } g \in G.
\end{equation}

In other words, the condition \eqref{eq:c_s-prof} says that two elements $g$ and $g'$ of $G$
are conjugate in $\widehat G$ if and only if they are conjugate in $G$.

We are now able to reformulate the hereditary conjugacy separability of $G$ in purely profinite terms:

\begin{cor}\label{cor:h_c_s-prof}
Suppose that $G$ is a residually finite group. Then $G$ is hereditarily conjugacy separable if and only if
for every $g \in G$ both of the following hold in the profinite completion $\widehat G$ of $G$:
\begin{itemize}
	\item $g^{\widehat G} \cap G=g^G$;
	\item $\overline{C_G(g)}=C_{\widehat{G}}(g)$.
\end{itemize}

\begin{proof} The necessity follows from Proposition \ref{prop:her_c_s-CC} and
Corollary \ref{cor:CC<->Chag-Zal}.

The sufficiency is given by the result of Chagas and Zalesskii \cite[Prop. 3.1]{Chag-Zal}.
It can also be deduced by first applying Corollary~\ref{cor:CC<->Chag-Zal} and then
Proposition  \ref{prop:her_c_s-CC}.
\end{proof}

\end{cor}

\end{document}